\documentclass[12pt, dvipsnames]{amsart}

\usepackage[shortlabels, inline]{enumitem}
\usepackage{xcolor}
\usepackage[framemethod=TiKZ]{mdframed}
\usepackage{textcomp}
\usepackage{eurosym}
\usepackage[all]{xy}
\usepackage{soul}
\usepackage{amsmath,mathtools}
\usepackage{amscd}

\usepackage[colorlinks=true, citecolor=BrickRed, urlcolor=RoyalPurple]{hyperref}
\usepackage{graphicx}
\usepackage[utf8]{inputenc}
\usepackage[english]{babel}
\usepackage[a4paper, portrait, margin=2cm]{geometry}
\usepackage{wrapfig}
\usepackage{tikz-cd}
\usepackage{tikz}
\usetikzlibrary{shapes.misc}
\usepackage[normalem]{ulem}

\usepackage{multicol}

\newcommand{\A}{\mathcal{A}}
\newcommand{\C}{\mathcal{C}}
\newcommand{\D}{\mathcal{D}}
\newcommand{\T}{\mathcal{T}}

\newcommand{\G}{\mathcal{G}}

\renewcommand{\P}{\mathcal{P}}
\newcommand{\Q}{\mathcal{Q}}
\newcommand{\B}{\mathcal{B}}
\newcommand{\Lap}{\Delta}
\newcommand{\HH}{\mathcal{H}}

\newcommand{\PP}{\mathbb{P}}
\newcommand{\N}{\mathbb{N}}
\newcommand{\Z}{\mathbb{Z}}

\newcommand{\R}{\mathbb{R}}

\newcommand{\op}{\operatorname{op}}

\newcommand{\Mor}{\operatorname{Mor}}

\newcommand{\Ima}{\operatorname{Im}}

\newcommand{\caln}{\mathcal{N}}

\newcommand{\defeq}{\overset{\text{\textup{def}}}{=}}

\newcommand{\hadgesh}[1]{\textcolor{Brown}{\emph{#1}}}

\definecolor{ForestGreen}{rgb}{0.0, 0.27, 0.13}
\definecolor{DukeBlue}{HTML}{001A57}
\definecolor{myblue}{rgb}{0.29,0.47,1}

\newcommand{\chg}[1]{\textcolor{black}{#1}}
\newcommand{\im}{\mathrm{Im}}
\newcommand{\cok}{\mathrm{CoKer}}
\renewcommand{\mod}{\textup{-}\curs{mod}}

\newcommand{\Hom}{\mathrm{Hom}}
\newcommand{\Ext}{\mathrm{Ext}}
\newcommand{\End}{\mathrm{End}}

\newcommand{\colim}{\mathop{\textup{colim}}\limits}
\renewcommand{\lim}{\mathop{\textup{lim}}\limits}
\newcommand{\Dimk}{\curs{H}}

\newcommand{\higherlim}[2]{\displaystyle\setbox1=\hbox{\rm lim}
\setbox2=\hbox to \wd1{\leftarrowfill} \ht2=0pt \dp2=-1pt
\setbox3=\hbox{$\scriptstyle{#1}$}
\def\test{#1}\ifx\test\empty
\mathop{\mathop{\vtop{\baselineskip=5pt\box1\box2}}}\nolimits^{#2}
\else
\ifdim\wd1<\wd3
\mathop{\hphantom{^{#2}}\vtop{\baselineskip=5pt\box1\box2}^{#2}}_{#1}
\else
\mathop{\mathop{\vtop{\baselineskip=5pt\box1\box2}}_{#1}}%
\nolimits^{#2}
\fi\fi}

\newcommand{\cals}{\mathcal{S}}
\newcommand{\cala}{\mathcal{A}}
\newcommand{\rk}{\operatorname{rk}\nolimits}

\newcommand{\Ker}{\operatorname{Ker}\nolimits}
\newcommand{\coKer}{\operatorname{coKer}\nolimits}
\newcommand{\VVect}{{\mathrm{Vect}_k}}

\newcommand{\calh}{\mathcal{H}}
\newcommand{\calc}{\mathcal{C}}

\newcommand{\xto}[1]{\xrightarrow{#1}}

\DeclareMathAlphabet\EuR{U}{eur}{m}{n}
\SetMathAlphabet\EuR{bold}{U}{eur}{b}{n}

\newcommand{\curs}{\EuR}

\renewcommand{\mod}{\mbox{-}\curs{mod}}

\newcommand{\Gr}{\curs{Gr}}

\newcommand{\obj}{\mathrm{Obj}}

\newcommand{\Pone}{\widehat{\P}}
\newcommand{\Cone}{\widehat{\C}}

\usepackage{amsthm}
\usepackage{amssymb,amsmath,extarrows}
\newtheorem{Thm}{Theorem}[section]
\newtheorem{Ex}[Thm]{Example}
\newtheorem{Prop}[Thm]{Proposition}

\newtheorem{Defi}[Thm]{Definition}

\newtheorem{Lem}[Thm]{Lemma}
\newtheorem{Rem}[Thm]{Remark}
\newtheorem{Cor}[Thm]{Corollary}

\newtheorem{Th}{Theorem}

\title{Foundations of differential calculus for modules over posets}

\author{Jacek Brodzki}
\address{School of Mathematical Sciences, Southampton University, Southampton, UK}
\email{j.brodzki@soton.ac.uk}
\author{Ran Levi}
\address{Institute of Mathematics, University of Aberdeen, Aberdeen, UK}
\email{r.levi@abdn.ac.uk}
\author{Henri Riihim\"aki}
\address{Nordita, Stockholm University, Stockholm, Sweden}
\email{henri.riihimaki@su.se}
\thanks{J. Brodzki and R. Levi are partially supported by an EPSRC grant EP/Y028872/1}
\thanks{Riihim\"aki is supported by The Wallenberg Initiative on Networks and Quantum Information (WINQ)}
\begin{document}

\begin{abstract}
Let $k$ be a field and let $\C$ be a small category. A $k$-linear representation of $\C$, or a $k\C$-module, is a functor from  $\C$  to the category of finite dimensional vector spaces over $k$. Unsurprisingly, it turns out that when the  category $\C$ is more general than a linear order, then its representation type is generally infinite and in most cases wild. Hence the task of understanding such representations in terms of their indecomposable factors becomes difficult at best, and impossible in general. This paper offers a new set of ideas designed to enable studying modules locally. Specifically, inspired by work in discrete calculus on  graphs, we set the foundations  for a calculus type analysis of $k\C$-modules, under some restrictions on the category $\C$. 
As a starting point, for a $k\C$-module $M$ we define its gradient \emph{gradient} \(\nabla[M]\) as a virtual module in the  Grothendieck group of isomorphism classes of $k\C$-modules, where the operation is induced by direct sum.  Pushing the analogy with ordinary differential calculus and discrete calculus on graphs, we define left and right divergence via the appropriate left and right Kan extensions  and two bilinear pairings on modules and study their properties, specifically with respect to adjointness relations between the gradient and the left and right divergence. The left and right divergence are shown to be rather easily computable in favourable cases. Having set the scene, we concentrate specifically on the case where the category $\C$ is a finite poset. Our main result is a necessary and sufficient condition for the gradient of a module $M$ to vanish under certain hypotheses on the poset.  We next investigate  implications for two modules whose gradients are equal. Finally we consider the resulting left and right Laplacians, namely the compositions of the divergence with the gradient, and study an example of the relationship between the vanishing of the Laplacians and the gradient.
\end{abstract}
\maketitle


Let $\C$ be a small category and let $\A$ be an abelian category. The category $\A^\C$ of functors $M\colon\C\to\A$ and natural transformations between them is again an abelian category. When $\A$ is the category $\VVect$ of finite dimensional vector spaces over a field $k$, objects of $\A^\C$ are frequently referred to as \hadgesh{$k$-linear representations of $\C$}. When \(\C\) has finitely many objects the representations can be identified as \hadgesh{$k\C$-modules}, i.e.~modules over the category algebra. In this article we shall be interested in the category $k\C\mod$ of such modules. 
A special case of interest, motivated in part by its prevalence in topological data analysis,  is where $\C$ is a finite category, i.e.~all its morphisms form a finite set, or even more specifically, a finite poset. 

Drozd trichotomy theorem \cite{Drozd} states that if $\C$ is a finite category generated by an acyclic quiver and $k$ is an algebraically closed field, then there are exactly three possibilities: either $k\C\mod$ contains finitely many isomorphism types of indecomposable modules (finite representation type), or it contains an infinite family of indecomposable modules that can be parametrised by a $1$-parameter family (tame representation type), or there is an $n$-parameter family of indecomposable modules for arbitrarily large $n$ in $k\C\mod$ (wild representation type).  Arbitrary finite categories, including finite posets,  are typically of infinite (tame or wild) representation type, which makes classification by decomposition into indecomposable summands  difficult in the tame case and  impossible in the wild case.

In this article we introduce a new approach to the study of $k\C$-modules. Instead of attempting to understand modules globally as a sum of indecomposable modules, or by taking projective resolutions, we propose a \hadgesh{calculus of $k\C$-modules}. That is, a very general methodology that enables one to extract local information about any given module. \chg{By \hadgesh{local information} we mean studying the behaviour of the restriction of modules to carefully chosen subcategories. In the case where $\C$ is a poset these could be sub-posets of $\C$ generated by vertex neighbourhoods,  by more general subtrees, or by Aleksandrov open subsets \cite{Alek} of the Hasse diagram of $\C$. }

The analogy with ordinary calculus is clear:  in exploring properties of a nice real valued function of several variables one typically employs standard differential techniques, which allow studying the function locally. The notions of gradient, divergence and Laplacian come to mind in this context. Inspired by these ideas, there is a discrete version of multivariable calculus for weighted directed graphs \cite{Lim}. Our treatment of $k\C$-module theory puts the ideas of discrete calculus on graphs into a categorical framework. \chg{Specifically, we shall define the notion of a gradient  for isomorphism classes of $k\C$-modules and show that, when restricted to the discrete context of integer valued functions on graphs, it coincides with the classical graph theoretic gradient. We also define left and right divergence  and left and right Laplacians. Again, in the discrete context we show that the  left and right coincide and reproduce the  divergence and the Laplacian for integer valued functions on graphs. Having established the general concepts and their basic properties, we restrict attention to the case where $\C$ is a finite poset and study some of the basic properties of our differential operators. }

The paper is motivated in part by the ever growing interest within the applied algebraic topology community in \hadgesh{generalised persistence modules}, which can be thought of as $k\C$-modules where $\C$ is typically a poset (possibly with nontrivial topology on its object set) \cite{BCB, BoLe, BdSS, BuSc, CZ_multi, OUDOT}. Applications of the  formalism presented here will be considered in subsequent work.


\section{Statement of Results}
\label{Sec:Results}
To state our main results some preparation is required. We provide a brief description here, and more details in Section \ref{Sec:preliminaries}. \hadgesh{Calculus on weighted directed graphs} is a discrete calculus for functions whose domain is the vertex set of a finite graph with weighted directed edges \cite{Lim}.  In this context the gradient is an operation which takes functions on the vertex set of a graph to functions on its edge set. We shall work in general with \hadgesh{quivers}, also known as \hadgesh{directed multigraphs}. We consider quivers as quadruples $(V, E, s, t)$, where $V$ is a set of vertices, $E$ is a set of edges and $s, t\colon E\to V$ are \hadgesh{source} and \hadgesh{target} functions. We will  assume throughout that quivers are \hadgesh{loop free}, namely that there is no $e\in E$, such that $s(e) = t(e)$.

Let $\G= (V, E)$ be a quiver. If $f$ is a function on $V$ with values in some additive group $A$, then the graph theoretic gradient is defined to be a function from $E$ to $A$, which takes an edge to the difference between the value at its target vertex and the value at its source vertex \cite{Lim}. We wish to categorify this idea, and so we work with categories that are generated, in the appropriate sense, by a quiver. To proceed,  we need two extra ingredients. 

First, we want the gradient of a $k\C$-module to be an object of the same type, namely a $k\widehat{\C}$-module, for some category $\widehat{\C}$ related to $\C$. To achieve this we associate with a quiver $\G$ another quiver $\widehat{\G}$, the line digraph of $\G$, whose vertices are the edges of $\G$ and whose edges are pairs of edges $(e,e')$ such that $t(e) = s(e')$. Thus the graph theoretic gradient can be thought of as an operator that takes functions on the vertices of $\G$ to functions on the vertices of the line digraph $\widehat{\G}$.  

The second ingredient we need is a structure that will allow us to take the difference of modules.  This motivates the use of the Grothendieck group of isomorphism classes of modules, where the sum is induced by direct sum of modules.  This  is frequently referred to as the \hadgesh{split Grothendieck group}. 

If $\G$ is a generating quiver for a small category  $\C$, namely the category $\C$  is a quotient category of the path category $\PP(\G)$ (See Section \ref{Sec:preliminaries} for details), then our gradient ought to be an operation that takes a  module  $M\in k\C\mod$ to a difference of two modules in $k\Cone_\G\mod$, where $\Cone_\G$ is a category that is obtained from $\C$ using the  line digraph $\widehat{\G}$ associated to $\G$, and is referred to as a \hadgesh{line category associated to $\C$ with respect to $\G$}. The category $\Cone_\G$ comes with two obvious functors $\tau, \sigma\colon \Cone_\G\to \C$ that associate with an object of $\Cone_\G$ its target and source object in $\C$, respectively. 

We define the gradient as the difference $\nabla\defeq \tau^*-\sigma^*$ of two natural homomorphisms $\tau^*, \sigma^*\colon \Gr(k\C)\to \Gr(k\Cone_\G)$ between the Grothendieck groups, (the \hadgesh{target} and \hadgesh{source} morphisms) that captures the variation of a module along each  morphism in $\C$ that is an edge in $\G$. Those morphisms in $\C$ can be thought of as  discrete analogs of infinitesimally small moves in a metric space. We show in  Example \ref{Ex:standard-grad} that our gradient, interpreted appropriately, is a generalisation of the graph theoretic gradient. The gradient we define is easily shown to be, as expected, a group homomorphism whose kernel contains all  \hadgesh{locally constant} modules, i.e., functors that take each morphism in $\C$ to an isomorphism of vector spaces, and it is natural with respect to inclusion of categories  (Proposition  \ref{Prop:Grad}). Continuing the analogy to calculus on graphs, we define two categorical versions of \hadgesh{divergence}, left and right, which can be shown to coincide with each other when restricted to the context of graphs, and reproduce the ordinary graph theoretic divergence (Example \ref{Ex:standard-div}). We then proceed with two types of bilinear, non-symmetric pairings on $\Gr(k\C)$, which under appropriate assumptions are shown to be very well behaved (Lemma \ref{Lem:Euler-form}), and again, when reduced to the context of graphs both pairings coincide, become symmetric and reproduce the standard inner product of  integer valued functions on graphs (Example \ref{Ex:pairing-discrete}). 

Having prepared the general background, we restrict attention to the main object of study - modules over posets. Finite posets are an important family of small categories that is  of  general interest, as well as being central to generalised persistence module theory \cite{BlBrHa, KiMe}. While the results stated below do not always require this restriction, some of the problems that arise in the general case are avoidable when one deals with modules over posets. The more general case will be studied in future works. One immediate advantage is that a finite poset $\P$ is generated by a canonical digraph that is its \hadgesh{Hasse diagram $\HH_\P$} - the digraph consisting of all irreducible morphisms in $\P$. 

By analogy to ordinary calculus, an obvious question is whether a vanishing gradient of a module $M$ implies that it is locally constant. The answer turns out to be quite subtle. A \hadgesh{directed tree} is a quiver $\T$ with the property that between any two distinct vertices $a$ and $b$ in $\T$ there is at most one directed path. We say that a poset $\P$ is \hadgesh{line connected} if the line digraph of its Hasse diagram is connected.  By convention, when we refer to a relation $x\le y$ in a poset $\P$ as irreducible, we exclude the possibility $x=y$. The following theorem is our first main result.

\begin{Th}[Thm.~\ref{Thm:grad0}]\label{Th:grad0}
	Let $\P$ be a finite poset, and let $\T\subseteq\HH_\P$ be a line connected subgraph that is a tree.  Let $\P_\T\subseteq \P$  denote the sub-poset generated by $\T$. Let $M\in k\P\mod$ be a module,  and let $M_\T$ denote the restriction of $M$ to $\P_\T$. Assume that $\nabla[M_\T]=0$ in $\Gr(k\P_\T)$. Then the following statements hold.
	\begin{enumerate}[(1)]
		\item 
		For any  objects $u,v\in\obj(\P_\T)$, there is an isomorphism \[\alpha_{u,v}\colon M(u)\to M(v),\] such that $\alpha_{u,u} = 1_{M(u)}$ and  $\alpha_{v,w}\alpha_{u,v} = \alpha_{u,w}$.
		\label{Th:grad0-1}
		\item 
		For every pair of irreducible relations $u\le w$  and $s\le t$ in $\P_\T$,  
		\[
		\alpha_{w,t}\circ M(u\le w) = M(s\le t)\circ \alpha_{u,s}.
		\]
		\label{Th:grad0-2}
		\item
		$M_\T$ is locally constant  if and only if $M(u < v)$ is an isomorphism for some  irreducible relation $u <  v$ in $\P_\T$.
		\label{Th:grad0-3}
	\end{enumerate}
Furthermore, if $M\in k\P\mod$ is a module such that for any irreducible relation $u\le v$ in $\P_\T$ there exists an isomorphism $\alpha_{u,v}\colon M(u)\to M(v)$ that satisfy  Conditions \ref{Th:grad0-1} and \ref{Th:grad0-2}, then $\nabla[M_\T]=0$.
\end{Th}

If $\P$ is a general finite poset, and $\Q\subseteq\P$ is a sub-poset whose Hasse diagram is a line connected tree, then Theorem \ref{Thm:grad0} gives information for any module $M\in k\P\mod$ on its restriction to $\Q$. This fits well with the idea of a local approach to studying modules. Notice that if $\P$ is line connected, then one can show that $\HH_\P$ contains a line-connected maximal tree $\T$. In that case statements \ref{Th:grad0-1} -- \ref{Th:grad0-3} of Theorem  \ref{Th:grad0} apply to all objects of $\P$.

We next examine the implication for a pair of  $k\P$-modules of having isomorphic gradients. A typical element in the Grothendieck group $\Gr(k\P)$, which is ordinarily referred  to as a \hadgesh{virtual module}, is a difference of two equivalence classes of genuine modules. Hence understanding modules of equal gradient is equivalent to gaining information about the kernel of the gradient homomorphism. As one can expect, in view of Theorem \ref{Th:grad0}, there are virtual modules of a vanishing gradient which are not locally constant. This in particular implies that the gradient is not a complete invariant, and so it makes sense to compare it to another incomplete invariant for modules - the rank invariant.

For $M\in k\P\mod$, define the \hadgesh{rank invariant} $\rk(M)\colon \Mor(\P)\to \N$ to be the function that takes a relation $x\le y$ to the rank of the homomorphism $M(x\le y)$. The rank invariant clearly extends by additivity to a group homomorphism $\rk\colon\Gr(k\P)\to\Z$, since it depends only on the isomorphism type of a module. Notice also that our definition of the rank invariant includes  the so called \hadgesh{Hilbert function}, which appears as the rank of $M$ applied to identity morphisms. It is well known that the rank invariant is a complete invariant for modules over \((\R,\leq)\) or any other linear order, but not for more general posets \cite{CZ_multi}. Our next theorem relates the gradient to the rank invariant. 

\begin{Th}[Thm.~\ref{Thm:Rank}]\label{Th:Rank}
	Let $\P$ be a finite poset. Let $ X = [M]-[N]\in \Gr(k\P)$ be any element. 
	\begin{enumerate}[(1)]
		\item Assume that $\nabla X = 0$.  Then $\rk X(u_0<v_0) = \rk X(u_1<v_1)$ for any pair of comparable objects $(u_0,u_1) < (v_0,v_1)$  in $\Pone$.  \label{Thm:Rank-Intro:1} 
	\end{enumerate}
	
Assume in addition that  $\T\subseteq\HH_\P$ is a line connected  tree, and let $\P_\T\subseteq \P$ be the sub-poset generated by $\T$ and let $[X_\T]$ denote the restriction of $ X$ to $\P_\T$. If $ X$ has a vanishing gradient on $\T$, then  
	
	\begin{enumerate}[(1)]
		\setcounter{enumi}{1}
		\item $\rk[X_\T]$ is constant on all identity morphisms in $\P_\T$, and 
		\label{Thm:Rank-Intro:2}
		\item $\rk[X_\T]$ is constant on all non-identity irreducible relations  in $\P_\T$, namely for any pair  $u_0<u_1$, and $v_0< v_1$ of such relations,   one has $\rk X(u_0<u_1) = \rk X(v_0<v_1)$. \label{Thm:Rank-Intro:3}
	\end{enumerate}
\end{Th}

Example \ref{Ex:Rank-best} shows that two modules may have equal gradients, but different rank invariants. On the other hand, Example \ref{Ex:Grad-Incomplete} gives infinitely many non-isomorphic modules with equal rank invariants and different gradients. Thus the gradient and the rank invariant are complementary in the data they detect.

In discrete calculus for finite weighted digraphs one considers real valued functions on vertices or edges as elements in finite dimensional real vector spaces, and as such one has the ordinary inner product defined on the spaces of vertex and edge functions, $\langle -, -\rangle_V$ and $\langle -, -\rangle_E$ respectively. This pairing allows one to define divergence and Laplacian for digraphs \cite{Lim}. The adjoint operator $\nabla^*$ with respect to the inner product, also referred to as the \hadgesh{divergence}, is then defined by the requirement that  the relation
\[\langle \nabla^*(f), g\rangle_V = \langle f, \nabla(g)\rangle_E\]
holds. 

In our context there is no obvious analog of an inner product on a real vector space, and hence defining an operator analogous to divergence is not straightforward. We propose two (related) substitutes. For modules $M, N\in k\P\mod$ define two pairings 
\[\langle [M], [N]\rangle_\P \defeq \dim_k(\Hom_{k\P}(M, N)),\quad\text{and}\quad \chi_\P([M], [N]) \defeq \chi(\Ext^*_{k\P}(M,N)),\]
which we refer to as the \hadgesh{Hom pairing} and the \hadgesh{Euler pairing}, respectively. The definition extended to  pairings on $\Gr(k\P)$ by additivity. Similarly define the corresponding pairings on $\Gr(k\Pone)$.  Since $\P$ is assumed to be a finite poset and we assume modules in $k\P\mod$ to be finitely generated, $\Ext^i_{k\P}(M,N)$ is a finite dimensional vector space for all $i$ and it vanishes for $i$ sufficiently large (See Lemma \ref{Lem:Euler-well-def}). Hence the Euler pairing is well defined. Going back to the motivating example of graph theory we show in Example \ref{Ex:pairing-discrete} that, with a suitable interpretation of functions on graphs as modules in an appropriate category, the two pairings coincide with each other, and are directly analogous to the ordinary inner product. In general, of course, neither pairing is symmetric.

\chg{While the Euler pairing offers some better  properties, the adjointness relations above do not hold for it in general. However, if the Hasse diagram of $\P$ is a tree, thus in particular an acyclic quiver, then $\Ext^i_{k\P}(M,N)$ vanishes for $i>1$ \cite[Theorem 2.3.2]{Derksen-Weyman}, and the Euler pairing can be computed as an ordinary inner product of Hilbert functions  \cite[Proposition 2.5.2]{Derksen-Weyman} (See Definition \ref{Def:dimvec} and Lemma \ref{Lem:Euler-form}).  This allows us to prove an important property of the Euler pairing, namely that if $\P$ is a poset whose Hasse diagram is a \hadgesh{rooted tree}, i.e., then $\chi_\P(-,-)$ is non-singular (Lemma \ref{Lem:Euler-form}). Furthermore, the Euler pairing $\chi_\P([M],[N])$ can be computed in general by means of a projective resolution of $M$ or an injective resolution of $N$, as our next theorem shows.}

\begin{Th}[Thm.~\ref{Thm:pairing}]\label{Th:pairing}
Let $\P$ be a finite poset, and let $M, N\in k\P\mod$. Let 
\[0\to P_n\to\cdots\to P_0\to M\to 0, \quad\text{and}\quad 0\to N\to I_0\to\cdots I_n\to 0\]
be a projective resolution for $M$ and an injective resolution for $N$.
Write $P_i \cong \bigoplus \epsilon^i_vF_v$ and $I_j = \bigoplus \delta_u^j G_u$, with $\epsilon^i_v, \delta^j_u\in \N$ and $v, u\in\P$. Then
\begin{equation*}\chi_\P([M],[N]) = \sum_{v\in\P}\sum_{i=0}^n(-1)^i\epsilon_v^i\dim_k N(v) = \sum_{u\in\P}\sum_{j=0}^n(-1)^j\delta_u^j\dim_k M(u).
\end{equation*}
\end{Th}

We next offer our analog of the divergence.  Considering the  Hom pairing as an analog of an inner product, the left and right Kan extensions offer themselves as a natural way of constructing a left adjoint  $\nabla^*$ and a right adjoint $\nabla_*$ of the gradient.  Thus for $N\in k\Pone\mod$, we define   \hadgesh{left divergence $\nabla^*$ and  right divergence $\nabla_*$} by
\[\nabla^*[N]\defeq [L_{\tau}(N)]- [L_{\sigma}(N)],\quad\text{and}\quad \nabla_*[N]=[R_{\tau}(N)]- [R_{\sigma}(N)],\]
where $L_\tau(N)$ and $L_\sigma(N)$ are the left Kan extensions, and $R_\tau(N)$ and $R_\sigma(N)$ are the right Kan extensions, in both cases along $\tau$ and $\sigma$ respectively.
 Example \ref{Ex:standard-div} demonstrates that with the right categorical setup for ordinary weighted digraphs, the left and right divergence operators coincide with each other  and amount to the ordinary definition of the graph theoretic divergence. With this setting the following relations 
	\[\langle \nabla X,  Y\rangle_{\Pone} = \langle  X, \nabla_* Y\rangle_\P\quad\text{and}\quad
	\langle \nabla^* X,  Y\rangle_{\P} = \langle  X, \nabla Y\rangle_{\Pone}\]
are straightforward, and can in fact be used as a definition of the left and right divergence.

\chg{Composing the gradient with the left and the right divergence, we obtain \hadgesh{left and right Laplacians $\Lap^0$ and $\Lap_0$}, respectively.  By analogy to graph theory, where the left and right Laplacians coincide, one may wonder whether virtual modules whose left and/or right Laplacian vanishes  have a vanishing gradient as well. The answer to this question turns out to be much more subtle, but in the case where $\P$ is a finite linear poset, the answer turns out to be affirmative, as our next theorem states.}

\chg{\begin{Th}[Thm \ref{Thm:Lap-Grad-Kernels}]\label{Th:Lap-Grad-Kernels}
\label{Thm:Lap-Grad-Kernels}
Let $\P$ denote the poset with objects $0,1,\ldots, n$, for some $n\geq 0$. Then the homomorphisms $\nabla, \Delta^0$ and $\Delta_0$ with domain $\Gr(kP)$ satisfy $\Ker\nabla = \Ker\Delta^0 = \Ker\Delta_0$.
\end{Th}}

\chg{Unlike the classical case,  the non-singularity of the Euler form (acting in place of an inner product) cannot be used to prove the theorem but rather only to show that a virtual module with a vanishing left or right Laplacian has a vanishing Hilbert function, which is a significantly weaker statement. Instead we use an explicit calculation of the Laplacians for $\P=[n]$ from which we are able to draw the required consequence.}

Computing the left and right divergence could  be quite involved for a general poset. However, here too we recall our local approach. Namely, while we may not be able to compute divergence for general posets, restricting to sub-posets will provide partial information. Thus we  consider the case where the Hasse diagram for the poset in question is a tree.  In that case our final theorem offers a rather easy calculation.

\begin{Th}[Thm.~\ref{Thm:Divergence-tree}]\label{Th:Divergence-tree}
Let $\P$ be a finite poset whose Hasse diagram is a tree, and let $N\in k\Pone\mod$.  For any object $y\in\P$ let $\caln_y$ denote the neighbourhood of $y$ in $\P$, and let $\widehat{\caln}_y$ be the associated line poset. Then the following statements hold.
\begin{enumerate}[(1)]
\item  
$L_{\tau}(N)(y) \cong \bigoplus_{(u,y)\in\Pone}N(u,y),\quad{and}\quad L_{\sigma}(N)(y)\cong \colim_{\widehat{\caln}_y}N|_{\widehat{\caln}_y}$.
\label{Th:Divergence-tree-1}
\item  
$R_{\tau}(N)(y) \cong \lim_{\widehat{\caln}_y} N|_{\widehat{\caln}_y},\quad\text{and}\quad R_{\sigma}(N)(y)\cong \bigoplus_{(y,v)\in\Pone} N(y,v)$.
\label{Th:Divergence-tree-2}
\end{enumerate}
\end{Th}

\chg{Based on this computation, we conjecture, although we have not been able to prove it, that Theorem \ref{Th:Lap-Grad-Kernels} can be generalised to any poset that is generated by a disjoint union of rooted trees.}

The paper is organised as follows. Section \ref{Sec:preliminaries} contains all the technical background material we use throughout the paper. In Section \ref{Sec:general-grad-div} we prepare the general setup for the construction of the gradient and the divergence in the context of module categories. In Section \ref{Sec:pairings} we study some properties of the Hom pairing and the Euler pairing in the context of modules over category algebras. Section \ref{Sec:Grad} is dedicated to the  study of the gradient of virtual modules over finite posets and the proof of Theorems \ref{Th:grad0} and  \ref{Th:Rank}. In Section \ref{Sec:Euler} we specialise the Hom and Euler pairings for modules over posets.  We compute the pairings in some interesting cases, study the relationships between them and prove Theorem \ref{Th:pairing}. Section \ref{Sec:divergence} is dedicated to the (left and right) divergence, some  adjointness relations with the gradient and the proof of Theorem \ref{Th:Divergence-tree}.  We then define  the corresponding Laplacians, and prove Theorem  \ref{Th:Lap-Grad-Kernels}. 

The authors are grateful to E. Meir for many helpful conversations on modules over category algebras and for finding an error in an early version of this paper.


\section{Preliminaries}
\label{Sec:preliminaries}

In this section we record the definitions, notation and all preliminary material that will be used throughout the paper. 

\subsection{$k\C$-Modules}
\label{Subsec:Cat-Alg-Mod}

Let $\C$ be a small category and let $\A$ be an abelian category. Then the functor category $\A^\C$ whose objects are functors from $\C$ to $\A$ and whose morphisms are natural transformations is also an abelian category. If $\tau\colon \C\to \D$ is a functor then pre-composition with $\tau$ induces a functor 
\[\tau^*\colon \A^\D\to \A^\C.\]
This functor, which is sometimes referred to as the \hadgesh{restriction along $\tau$}, will be used in our definition of the gradient in Section \ref{Sec:general-grad-div} (Definition \ref{Def:grad-general}). If $\A$ is complete and cocomplete, then the restriction $\tau^*$  generally has left and right adjoints given by the left and right Kan extensions, respectively (Section \ref{Subsec:Kan}). These will be used in Section \ref{Sec:general-grad-div} to define a left and a right divergence (Definition \ref{Def:div-general}). A good reference for the general theory of functor categories is \cite{GS}.

Fix a field $k$. Unless otherwise specified we will not assume that  $k$ is algebraically closed. Let $\VVect$ be the category of \hadgesh{finite dimensional} vector spaces over $k$.
 
\begin{Defi}\label{Def:kC-mod} 
Let $\C$ be a small category. A \hadgesh{$k\C$-module} is a functor $M\colon \C \to \VVect$. Let $k\C\mod$ denote the category of $k\C$-modules and natural transformations between them.
\end{Defi}

When the category $\C$ has finitely many objects the category $k\C\mod$ has an equivalent description, which is useful for some purposes.

\begin{Defi}\label{def:poset_algebra}
	Let \(k\) be a field and let $\C$ have finitely many objects. The \hadgesh{category algebra} \(k\C\)  is the unital algebra generated as a $k$-vector space by all morphisms \(x \to y\) in $\C$ (including identities). Two morphisms multiply by composition, and non-composable morphisms multiply to 0 \cite{Xu}. The unit is the $k$-algebra map $\eta\colon k\to k\C$ that sends $1\in k$ to the element $\mathbf{1}\in k\C$, that is the sum over all objects in $\C$ of the identity morphisms $1_x\defeq x\to x$. 
\end{Defi}
Clearly, finiteness is required in this definition only for the unit element to make sense. The following theorem due to Mitchell allows one to alternate between functor categories and categories of modules in the usual sense, when the category in question is finite.

\begin{Thm}[\cite{Mitchell}]\label{Thm:Mitchell]}
	Let \(\C\) be a category with finitely many objects and let $k$ be a field. The category of modules over the category algebra $k\C$ and $k\C$-linear homomorphisms is equivalent to the category  $k\C\mod$ of functors $M\colon\C\to \VVect$ and natural transformations between them.
\end{Thm}

The equivalence in Mitchell's theorem is given as follows. If $M\colon\C\to\VVect$ is a functor, let 
\[\mathbf{M} \defeq \bigoplus_{x\in\obj(\C)}M(x).\] 
The (left) action of $k\C$ on $\mathbf{M}$ is determined by the way in which each morphism  $\varphi\colon y\to z$ in $\C$ acts on a typical element of $\mathbf{M}$. Firstly, for any \(m \in M(y)\) we define
\[\varphi \cdot m = M(\varphi)(m) \text{ if } \varphi \in \Hom_\C(y,z) \text{ for some } z \in \C.\]
Otherwise we define \(\varphi \cdot m = 0\). Then extending this operation linearly to any \(\mathbf{m} \in \mathbf{M}\) and \(f \in k\C\) gives the desired \(k\C\)-module structure.

Conversely, if $\mathbf{N}$ is a left module over the category algebra $k\C$, let $N\colon\C\to\VVect$ be the functor that takes an object $x\in\C$ to $N(x) \defeq 1_x\cdot\mathbf{N}$ and a morphism $\varphi\colon y\to z$ to the homomorphism $N(\varphi)$ that maps $1_y\cdot\mathbf{a}\in 1_y\cdot\mathbf{N}$ to $\varphi\cdot\mathbf{a}\in 1_z\cdot\mathbf{N}$.

The following general terminology is rather standard and will become useful in our analysis.
\begin{Defi}\label{Def:loc_const-virt_triv}
	Let $\C$ be a small category. We say that a  module $M\in k\calc\mod$ is 
	\begin{itemize}
		\item \hadgesh{locally constant} if for every morphism $\varphi$ in $\C$ the induced homomorphism  $M(\varphi)$ is an isomorphism, 
		\item \hadgesh{virtually trivial} if for every non-identity morphism $\varphi$ in $\C$ the induced homomorphism  $M(\varphi)$ is trivial. 
	\end{itemize}
\end{Defi}

\begin{Lem}\label{Lem:virtually-trivial}
	Let $\calc$ be a small category and let $M, N\in k\calc\mod$ be virtually trivial modules. Let $\obj_\calc$ denote the discrete category (only identity morphisms) with objects set $\obj(\calc)$, and let $j_\calc\colon \obj_\calc\to \calc$ denote the inclusion. Then there is a natural isomorphism of groups \[J\colon \Hom_{k\calc\mod}(M,N) \to \Hom_{k\obj_\calc\mod}(j_\calc^*(M), j_\calc^*(N)), \]
	where on both sides $\Hom$ denotes the group of natural transformations between the respective modules, and $j_\calc^*$ denotes pre-composition with $j_\calc$. 
\end{Lem}
\begin{proof}
	Since $M$ and $N$ are virtually trivial, a natural transformation from $M$ to $N$ amounts exactly to a homomorphism of vector spaces $M(x)\to N(x)$ for each $x\in\obj(\calc)$. Thus $J$ is an isomorphism of sets. Since the functor categories are abelian and $j^*_\calc$ is additive, the lemma follows. 
\end{proof}

\begin{Cor}\label{Cor:virtually-trivial}
	Let $\calc$ be a small category and let $M, N\in k\calc\mod$ be virtually trivial modules. Then $M\cong N$ if and only if $M(x)\cong N(x)$ for each object $x\in\obj(\calc)$.
\end{Cor}

\subsection{Quivers  and their line digraphs}
\label{Subsec:line-digraphs}
A central idea in our work is to define a gradient on $k\C$-modules that yields a virtual $k\Cone$-module (See Section \ref{Subsec:Grothendieck}), where $\Cone$ is a certain category associated to $\C$ and a certain generating digraph. This is where the classical graph theoretic construction of the line digraph becomes useful. We start by recalling some basic concepts.

\begin{Defi}\label{Def:digraph}
	 A \hadgesh{quiver}  is  a quadruple  $\G = (V, E, s, t)$, where $V$ and $E$ are sets, called the \hadgesh{vertex set} and \hadgesh{edge set}, respectively, and $s, t\colon E\to V$ are  functions, called \hadgesh{source} and \hadgesh{target} maps.   A \hadgesh{directed graph}, or a \hadgesh{digraph}, \(\G\) is  a quiver where the function $s\top t\colon E\to V\times V$ is injective, namely  for any two vertices $u$ and $v$, there is either no directed edge from $u$ to $v$ or only one such edge.  
 \end{Defi}

We shall frequently denote a quiver or a digraph only by $\G = (V, E)$, namely with no mention of the source and target functions. The full notation will always be used when ambiguity may arise. 

Definition \ref{Def:digraph} allows multiple and reciprocal connections between any pair of vertices in a quiver. In a digraph multiple edges with the same source and target are not allowed, but reciprocal connections are allowed.  When $\G=(V, E, s, t)$ is a digraph, we will consider its edge set as a subset $E\subseteq V\times V$ and write $e= (u,v)$ to denote that $e$ is an edge from $u$ to $v$. 
A quiver is said to be \hadgesh{acyclic} if there is no directed path that starts and ends at the same vertex (in particular, no reciprocal connections are allowed). While acyclicity is not assumed throughout, in the context of this work we do assume  that all quivers are \hadgesh{loop-free}.

\begin{Defi}\label{Def:line-digraph}
	Given a quiver $\G = (V,E, s, t)$, the associated \hadgesh{line digraph} $\widehat{\G} = (\widehat{V}, \widehat{E}, \widehat{s}, \widehat{t})$ is the digraph with vertices $\widehat{V} = E$. If $e, e'\in E$ and $t(e) = s(e')$, then the pair $(e, e')$ is a directed edge $e\to e'$ in $\widehat{E}$. Thus  $\widehat{s}(e, e') = e$ and $\widehat{t}(e, e') = e'$. If $\G$ is a digraph (so there is at most one edge in a given direction between any two vertices), then the directed edges in $\widehat{E}$ will sometimes be denoted by ordered triples $(u,v,z)$, where $(u,v)$ and $(v,z)$ are edges in $E$.
\end{Defi}

Notice also that the terminology \hadgesh{line digraph} is justified because $\widehat{\G}$ is always a digraph (rather than a general quiver). Indeed, let $e, e'\in \widehat{V}$ be vertices. Then there is at most one edge in $\widehat{E}$ from $e$ to $e'$, which exists if and only if $t(e) = s(e')$. Thus $\widehat{G}$ has no multiple edges in the same direction between the same pair of vertices. For instance, if $\G$ has three vertices $u ,v, w$ with $n$ directed edges from $u$ to $v$ and $k$ directed edges from $v$ to $w$, then $\widehat{\G}$ is an $n$ to $k$ directed bipartite graph.

The line digraph construction is clearly functorial with respect to inclusions. Namely, if $\G'\subseteq \G$ is an inclusion of quivers,  then $\widehat{\G'}\subseteq\widehat{\G}$.

A quiver $\G$ is said to be \hadgesh{connected} if, ignoring edge direction and collapsing multiple and bidirectional edges between the same pair of vertices into a single undirected edge, one obtains a connected graph. The line digraph associated to a quiver $\G$ is not generally connected. But if \(\widehat{\G}\) is a connected line digraph of some quiver \(\G\), then \(\G\) is connected. The following concept of connectivity is useful for our purposes.

\begin{Defi}\label{Def:line-components}
	Let $\G$ be a quiver with a line digraph $\widehat{\G}$. A sub-quiver $\G_0\subseteq \G$ is said to be a \hadgesh{line component of $\G$} if there exists a connected component $\widehat{\G_0}\subseteq\widehat{\G}$ such that $\widehat{\G_0}$ is the line digraph associated to $\G_0$. A quiver $\G$ is said to be \hadgesh{line connected} if $\widehat{\G}$ is connected, or equivalently if the only line component of $\G$ is $\G$ itself.
\end{Defi}
\begin{figure}
    \centering
\[\begin{tabular}{|c|c|}
\hline
\xymatrix{
&d\\
b\ar@<-0.5ex>@[red][ur]^{e_4\;} \ar@<0.5ex>@[red][ur]_{\;e_5} && c\ar@[blue][ul]_{e_6}\\
&a\ar@[red][ul]^{e_1}\ar@[blue]@<-0.5ex>[ur]^{e_2\;}\ar@[blue]@<0.5ex>[ur]_{\;e_3}
}\qquad\qquad & 
\xymatrix{
e_4 && e_5 &&e_6\\
\\
&e_1\ar@[red][uul]\ar@[red][uur] && e_2\ar@[blue][uur] && e_3\ar@[blue][uul]
}\\
\hline
\end{tabular}
\]
\caption{A quiver (left) and its line digraph (right)  with the line components and their associated line digraphs in red and blue.}  \label{fig:line-components}
\end{figure}
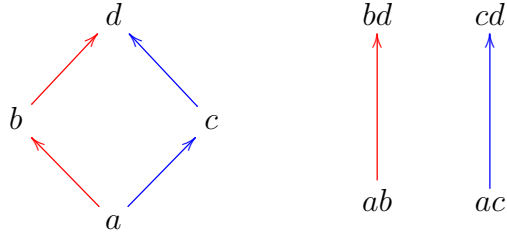
Clearly any quiver $\G$ is a union of its line components (which are not generally disjoint in $\G$). The line components of $\G$ can be constructed from the connected components of $\widehat{\G}$ by considering for each component $\widehat{\G}_0$, the  sub-quiver of $\G$ consisting of all edges corresponding to the vertices of $\widehat{\G}_0$ (See Figure \ref{fig:line-components}).

\subsection{The Grothendieck Group}
\label{Subsec:Grothendieck}
As is the case for any category of modules, many questions about $k\C$-modules can be reduced to questions about isomorphism classes.   In this context the concept of the Grothendieck group is very useful.

Let $k$ be a field, let $A$ be a unital $k$-algebra and let $A\mod$ denote the category of finitely generated left $A$-modules. Then $A\mod$ is a monoid with operation given by direct sum, which is associative and commutative up to isomorphism, and thus the set of isomorphism classes of $A$-modules forms a commutative monoid where 
\begin{equation}
	[M]+[N] \defeq [M\oplus N].
	\label{Eqn:Grothendieck}
\end{equation} 
The \hadgesh{Grothendieck group} $\Gr(A)$ of finitely generated $A$-modules is the group completion  of this monoid. Namely, it is the quotient of the free abelian group generated by isomorphism classes of $A$-modules, subject to the relation \textcolor{red}{(\ref{Eqn:Grothendieck})}. Thus any element  $ X\in \Gr(A)$ can be written uniquely as a difference $ X = [M]-[N]$, where $M, N\in A\mod$. In particular, an equation of the form $[M]-[N] = [U] - [V]$ in $\Gr(A)$ where $M, N, U, V\in A\mod$ should be interpreted as $M\oplus V \cong N\oplus U$. We refer to elements of $\Gr(A)$ as \hadgesh{virtual modules}.
If $f^*\colon B\mod\to A\mod$ is a functor that preserves direct sums then $f^*$ induces a group homomorphism
\[f^*\colon \Gr(B)\to \Gr(A).\]
In certain cases, including some of the cases considered in this article,   $\Gr(A)$ can be endowed with a ring structure by using tensor products. However, multiplicative structures will not be used in this article.

An important example is given by the Grothendieck group $\Gr(k)$ of finite dimensional vector spaces over $k$, which is 
isomorphic to the additive group of integers $\Z$. A much more interesting family of examples that is a subject of extensive study arises in group representation theory. If $G$ is a group and $kG$ its group algebra, then $\Gr(kG)$ is the Grothendieck group of all virtual linear representations of the group $G$ over $k$. In representation theory one typically uses the followin version of the construction. 

\begin{Defi}\label{Rem:No=SES}
	Let $A$ be a unital $k$-algebra. Define the \hadgesh{reduced Grothendieck group} $\Gr_e(A)$ to be the  quotient of the Grothendieck group $\Gr(A)$, by an additional family of relations:
	\[[M] = [M']+[M'']\]
	if there is a short exact sequence of $A$-modules $0\to M'\to M\to M''\to 0$.
\end{Defi}
The reduced Grothendieck group $\Gr_e(A)$ inherits a group structure from that of $\Gr(A)$, and the obvious projection $\Gr(A)\to \Gr_e(A)$ is a group homomorphism. However, $\Gr_e(A)$ is a much less interesting object than $\Gr(A)$ for our purposes, since the class of any module is equal to the sum of all simple modules in a composition series. Thus $\Gr_e(A)$ is isomorphic to the free abelian group generated by the simple $A$-modules. When considering posets, which are of major interest in this paper, simple modules are those that take the value $k$ at a single object and $0$ everywhere else.

\begin{Defi}\label{Def:dimvec}
For a small category $\C$ with object set $V$, define  $\Dimk_\C\colon \Gr(k\C)\to \Z^V$ by
\[\Dimk_\C[M](v) \defeq  \dim_k(M(v)),\]
for any $M\in k\C\mod$ and for $ X = [M]-[N]$ with $M,N\in k\C\mod$, let 
\[\Dimk_\C X \defeq \Dimk_\C[M]-\Dimk_\C[N].\]
For $ X\in\Gr(k\C)$ we refer to $\Dimk_\C X$ as the \hadgesh{Hilbert function of $ X$}.
\end{Defi}
Since the addition operation in $\Gr(k\C)$ is induced by direct sum, the function $\Dimk_\C$ is clearly a group homomorphism, whose kernel is the subgroup of all virtual modules $ X = [M]-[N]$ such that $\dim_kM(v) = \dim_kN(v)$ for every $v\in V$. Clearly for $M\in k\C\mod$, $\Dimk_\C[M] = 0$ if and only if $M=0$. 

\begin{Defi}\label{Rem:lc-vt}
	Let $\C$ be a small category. An element $ X\in\Gr(k\C)$ is said to be \hadgesh{locally constant}, if it can be represented as a difference of the isomorphism classes of two locally constant modules, and similarly for \hadgesh{virtually trivial} (Definition \ref{Def:loc_const-virt_triv}).
\end{Defi}

\begin{Rem}
	\label{Rem:No-ring-hom}
	Recall that if $\C$ is a category with a finite set of objects and $k$ is a field, then the functor category $k\C\mod$ is equivalent to the category of modules over the category algebra $k\C$ \cite{Xu}. If $F\colon\C\to \D$ is a functor, then one obtains a homomorphism $kF\colon k\C\to k\D$, which is a homomorphism of algebras if and only if $F$ is injective on object sets \cite[Proposition 2.2.3]{Xu}, and otherwise is only a homomorphism of vector spaces. 
\end{Rem}

\subsection{Functor categories and Kan extensions}
\label{Subsec:Kan}
Kan extensions will be used to define the divergence operators in Section \ref{Sec:general-grad-div}, and we recall here the basic notions.  Let $\C, \D$ be small categories, and let $\A$ be a category that is bicomplete (i.e., all small limits and colimits exist in $\A$). Let $\A^\C$ and $\A^\D$ denote the categories of functors from $\C$ and $\D$ to $\A$, respectively,  with morphisms given by natural transformations. If $F\colon \C\to\D$ is a functor, then the restriction $F^*\colon \A^\D\to \A^\C$ admits  left and right  adjoints, given by the left and right Kan extensions \cite[X.3, Corollary 2]{MacLane},
\[L_F, R_F\colon  \A^\C\to \A^\D.\]
For the sake of completeness, we briefly describe these constructions. The reader is referred to \cite[Chapter X]{MacLane} for details. For any object $d\in\D$, let $F\downarrow d$ denote the \hadgesh{overcategory of $d$ with respect to $F$}, whose objects are pairs $(c,\varphi)$, where $c$ is an object in $\C$ and $\varphi\colon F(c)\to d$ is a morphism in $\D$. A morphism $(c,\varphi)\to(c',\varphi')$ in  $F\downarrow d$ is a morphism $\sigma\colon c\to c'$, such that $\varphi'\circ F(\sigma) =\varphi$. There is a forgetful functor $\#\colon F\downarrow d \to \C$, 
sending an object $(c,\varphi)$ to $c$. By analogy one defines $d\downarrow F$, the \hadgesh{undercategory of $d$ with respect to $F$}, and similarly a forgetful functor to $\C$. For $M\in \A^\C$, left and right Kan extensions of $M$ along $F$ are defined by
\[L_F(M)(d) \defeq \colim_{F\downarrow d}M_\#, \quad\text{and}\quad
R_F(M)(d) \defeq \lim_{d\downarrow F} M_\#,\]
where in both cases $M_\#$ denotes the composition of $M$ with the respective forgetful functor. Thus if $N\in \A^\C$ and $M\in \A^\D$, we have the adjointness relations:
\[\Hom_{\A^\D}(L_F(N), M)\cong \Hom_{\A^\C}(N, F^*(M)),  \quad\text{and}\quad
\Hom_{\A^\C}(F^*(M), N)\cong \Hom_{\A^\D}(M, R_F(N)).\]

The following two  lemmas are standard homological algebra. Proofs are included for the convenience of the reader.
\begin{Lem}\label{Lem:gen-cogen}
Let $\A, \B$ be small abelian categories and assume that $\A$ has enough projectives. Let $F\colon \A\to\B$ be a functor with a right adjoint $G\colon\B\to\A$. Then the following statements are equivalent.
\begin{enumerate}[(1)]
\item $F$ sends projective objects in $\A$ to projective objects in $\B$. \label{gen-cogen-1}
\item $G$ is an exact functor. \label{gen-cogen-2}
\end{enumerate}
Dually, if $\B$ has enough injectives, then the following statements are equivalent.
\begin{enumerate}[(1)]
\setcounter{enumi}{2}
\item $G$ sends injective objects in $\B$ to injective objects in $\A$. \label{gen-cogen-3}
\item $F$ is an exact functor. \label{gen-cogen-4}
\end{enumerate}
\end{Lem}
\begin{proof}
We prove the equivalence of \ref{gen-cogen-1} and \ref{gen-cogen-2}. The equivalence of \ref{gen-cogen-3} and \ref{gen-cogen-4} follows by analogy. Let $P\in \A$ be a projective object. Then  $F(P)\in\B$ is projective if and only if the functor $\Hom_\B(F(P), -)\cong \Hom_\A(P, G(-))$ is exact.  Thus, if $G$ is  exact then $F(P)$ is projective. This shows that \ref{gen-cogen-2} implies \ref{gen-cogen-1}. 

Conversely, since $\A$ is assumed to have enough projectives, every object $X\in \A$ admits an epimorphism $P\to X$, where $P$ is projective in $\A$. Let 
\[0\to B\xto{\alpha} C\xto{\sigma} D\to 0\]
be an exact sequence in $\B$. Since $G$ is a right adjoint, it is left exact, so it suffices to show that 
\[G(C)\xto{G(\sigma)} G(D)\to 0\]
is exact. Let $\varphi\colon P\to G(D)$ be an epimorphism in $\A$, where $P$ is projective. By \ref{gen-cogen-1}, $F$ sends projective objects to projective objects. Thus 
\[\Hom_\B(F(P), C)\xto{\sigma_*} \Hom_\B(F(P), D)\to 0\]
is exact, and by adjointness 
\[\Hom_\A(P, G(C))\xto{G(\sigma)_*} \Hom_\A(P, G(D))\to 0\]
is exact. Hence there exists $\psi\in \Hom_\A(P, G(C))$ such that $G(\sigma)_*(\psi) = G(\sigma)\circ\psi = \varphi$. Since $\varphi$ is an epimorphism, so is $G(\sigma)$, as claimed. 
\end{proof}

As a corollary we have the following lemma.

\begin{Lem}\label{Lem:Ext-adjointness}
Let $\gamma\colon\C\to \D$ be a functor between small categories. Let $\A$ be a bicomplete abelian category, and let $\A^\C$ and $\A^\D$ denote the respective functor categories.  Let $\gamma^*\colon \A^\D\to \A^\C$ denote the restriction, and let $L_\gamma$ and $R_\gamma$ denote its left and right Kan extensions. Then for any $M\in \A^\C$ and $N\in \A^\D$,
\[\Ext^*_{\A^\D}(L_\gamma(M), N) \cong \Ext^*_{\A^\C}(M,\gamma^*N)\]
if $\A$ has enough injectives and either of the following two equivalent conditions holds.
\begin{enumerate}[(1)]
\item $\gamma^*$ sends injective modules  to injective modules. \label{Lem:Ext-adjointness-1}
\item $L_\gamma$ is an exact functor. \label{Lem:Ext-adjointness-2}
\end{enumerate}
Similarly, if $\A$ has enough projectives then 
\[\Ext^*_{\A^\D}(N, R_\gamma(M)) \cong \Ext^*_{\A^\C}(\gamma^*N,M)\]
if either of the following equivalent conditions holds.
\begin{enumerate}[(1)]
\setcounter{enumi}{2}
\item $\gamma^*$ sends projective modules  to projective modules. \label{Lem:Ext-adjointness-3}
\item $R_\gamma$ is an exact functor.\label{Lem:Ext-adjointness-4}
\end{enumerate}
\end{Lem}
\begin{proof}
We prove the first statement. The second follows by analogy.  Let $N\to I^*$ be an injective resolution of $N$ in $\A^\D$. Since $\gamma^*$ is exact and sends injective modules to injective modules, $\gamma^*N\to \gamma^*(I^*)$ is an injective resolution of $\gamma^*N$ in $\A^\C$. Then
\[\Ext^*_{\A^\C}(M, \gamma^*N) = H^*(\Hom_{\A^\C}(M, \gamma^*(I^*)) \cong H^*(\Hom_{\A^\D}(L_\gamma (M), I^*)) = \Ext_{\A^\D}^*(L_\gamma (M), N),\]
as claimed. The equivalence of Conditions \ref{Lem:Ext-adjointness-1} and \ref{Lem:Ext-adjointness-2} follows from Lemma \ref{Lem:gen-cogen}.
\end{proof}

\newcommand{\bbL}{\mathbb{L}}
\newcommand{\bblad}{\cals_\sqcup}
\newcommand{\bbladd}{\cala_\sqcup}
\section{Categorifying the discrete gradient and divergence}
\label{Sec:general-grad-div}
In this section we present a categorical analog  of the discrete gradient and divergence on digraphs. With the appropriate setup, the standard notions of these concepts then occur as special cases. In sections \ref{Sec:Grad} and \ref{Sec:divergence} we will specialise these constructions to the case where the category $\C$ is a  finite poset. This case is particularly well behaved and is strongly related to the notion of generalised persistence modules.

We start by recalling the graph-theoretic definitions of gradient and divergence. Let $\G = (V,E)$ be a digraph. Let $A$ be an abelian group, and let $f\colon V\to A$ and $F\colon E\to A$ be arbitrary functions. Then one defines the gradient $\nabla(f)\colon E\to A$ and the divergence  $\nabla^*(F)\colon V\to A$ by
\begin{equation}\nabla(f)(v\to w) \defeq f(w) - f(v),\quad\text{and}\quad 
\nabla^*(F)(v) = \sum_{u\to v} F(u\to v) - \sum_{v\to w} F(v\to w).\label{Eq:grad-div}
\end{equation}
If $V$ and $E$ are finite and $A=\R$, then the gradient and divergence are related to each other as adjoint operators by means of the standard inner product. The reader is referred to \cite{Lim} for a comprehensive discussion of the gradient and the divergence on digraphs.

\subsection{Categories generated by a quiver}
\label{Subsec:Path-Cat}
Let $\G = (V, E)$ be a quiver. The path category associated to $\G$ is the category $\PP(\G)$ with $V$ as an object set and whose morphisms are  all directed compositions of the directed edges in $E$. Thus $\PP(\G)$ is the category that is freely generated by $\G$. In particular, if $\G$ is acyclic, then the only endomorphisms in $\PP(\G)$ are identities. If, in addition $\G$ is finite, then $\PP(\G)$ is also finite, but if $\G$ contains cycles then $\PP(\G)$ is infinite since each morphism corresponding to a cycle can be iterated arbitrarily many times. In this article we will  work primarily with acyclic quivers.

Any small category $\C$ (with discrete topology on object and morphism sets) is a quotient category of the path category of a quiver whose vertex set is in $1-1$ correspondence with the object set of  $\C$, and whose directed edges form a subset of its morphisms, where the equivalence relations are given by the commutativity relations in $\C$ (See \cite[Section 5]{Borceux} for more detail). A canonical example of this type is given by considering the quiver of all morphisms in a category $\C$. Then $\C$  is reconstructed from that quiver by declaring two paths equivalent if the corresponding morphisms in $\C$ coincide. 

A poset  $\P$  can be thought of as a quotient of the path category of its  Hasse diagram $\calh_\P$ (the digraph of all objects in $\P$ and all irreducible morphisms between pairs of them)   by the relations imposed by requiring that between any two objects $u$ and $v$ there is exactly one morphism if there is a directed path in $\HH_\P$ from $u$ to $v$. \chg{Of course, if we let $\G_\P$ denote the quiver given by all morphisms in $\P$, then  $\P$ is also a quotient category of $\PP(\Q)$ for any quiver $\HH_\P\subseteq\Q\subseteq\G_\P$, with the relations determined in exactly the same manner.  Finite posets are of primary interest in this article, but categories generated by acyclic quivers are much more general than posets. }

\begin{Defi}\label{Def:generating-digraph}
	If $\C$ is a category  that is a quotient category of the path category $\PP(\G)$ of some quiver $\G$, we say that $\G$ is a \hadgesh{generating quiver} for $\C$, or that $\C$ is \hadgesh{generated} by $\G$.
\end{Defi}

If the quiver $\G$ is acyclic then every category it generates will also be cycle-free. Hence we say that a category $\C$ is acyclic if it is generated by an acyclic quiver.

\chg{Next we associate with any small   category $\C$ and a generating  quiver $\G$ another small category $\widehat{\C}_\G$ that is generated by the associated line digraph $\widehat{\G}$. The category $\widehat{\C}_\G$ will depend on the choice of the generating quiver, but for any such choice one has two functors $\tau, \sigma\colon\widehat{\C}_\G\to\C$ that will allow us to define a  gradient and two divergence operators. }

\begin{Defi}\label{Def:line-cat}
	Let $\C$ be a small category generated by a quiver  $\G = (V,E, s, t)$, and let $\widehat{\G} = (\widehat{V}, \widehat{E}, \widehat{s},  \widehat{t})$ denote its associated line digraph. Let $\PP(\G)$ and $\PP(\widehat{\G})$ denote the associated path categories. Define a category \hadgesh{$\widehat{\C}_\G$ with the same object set as $\PP(\widehat{\G})$}. Let $e: u\to x$ and $e'\colon y\to v$ be objects in $\PP(\widehat{\G})$.
Assume that 
	\begin{align*}\alpha\defeq &\{e\to e_1\to e_2 \to\cdots\to e_{m-1} \to e'\}, \quad\text{and}\quad\\
		\gamma\defeq &\{e\to f_1\to f_2 \to\cdots\to f_{k-1} \to e'\}
	\end{align*}
	are  morphisms in the path category $\PP(\widehat{\G})$, where $e_i\colon a_i\to a_{i+1},\;\; 1\le i\le m-1$ and $f_j\colon b_j\to b_{j+1},\;\; 1\le j \le k-1$, $a_1 = b_1 = x$, and $a_m = b_k =y$. For the pair of objects $e, e'$, define an equivalence relation on the morphism set $\PP(\widehat{\G})(e, e')$, to be the transitive closure of the relation $\alpha\sim \gamma$ if the compositions 
	\[x =  a_1\xto{e_1}\cdots\xto{e_{m-1}} a_m = y, \quad\text{and}\quad x =  b_1\xto{f_1}\cdots\xto{f_{k-1}} b_k = y \]
	coincide in $\C$. Define \hadgesh{$\widehat{\C}_\G(e, e')$} to be the set of equivalence classes of this relation. We refer to $\widehat{\C}_\G$ as the \hadgesh{line category associated to $\C$ with respect to the generating quiver $\G$}.
\end{Defi}

\subsection{The categorical gradient}
We are now ready to define a categorical version of the gradient. We start by defining a pair of graph maps $\widehat{\G}\to\G$, which in turn induce  functors $\tau, \sigma\colon\widehat{\C}_\G\to \C$, where $\C$ is generated by $\G$.

\begin{Defi}
	\label{Def:target-source}
	Let $\G = (V, E, s, t)$ be a quiver, and let $\widehat{\G}=(\widehat{V}, \widehat{E}, \widehat{s},  \widehat{t})$ be its associated line digraph. Let $\tau, \sigma\colon \widehat{\G}\to \G$ denote the \hadgesh{target} and \hadgesh{ source} graph morphisms, defined by $\tau(e) = t(e)$ and $\sigma(e) = s(e)$. For an edge $(e, e')\in \widehat{\G}$, define  $\tau(e,e') = e'$ and $\sigma(e,e') = e$. 
\end{Defi}

The maps $\tau$ and $\sigma$ induce target and source functors
\[\tau, \sigma\colon\PP(\widehat{\G})\to\PP(\G).\]
The following lemma shows that these functors induce the corresponding functors on the line category.

\begin{Lem}\label{Lem:generating-digraph}
	Let $\C$ be a small category generated  by a quiver $\G=(V,E, s, t)$.
	Let $\widehat{\C}_\G$ denote the  line category  associated to $\C$ with respect to $\G$. Then the target and source functors $\tau$ and $\sigma$ on path categories induce functors 
	\[\tau, \sigma\colon \widehat{\C}_\G\to \C.\]
\end{Lem}
\begin{proof}
	It suffices to show that if $\psi$ is either $\tau$ or $\sigma$ then there exists a functor $\bar{\psi}\colon\widehat{\C}_\G\to\C$, such that the square 
	\[\xymatrix{
		\PP(\widehat{\G})\ar[r]^\psi\ar[d]^{\widehat{\pi}} &\PP(\G)\ar[d]^\pi\\
		\widehat{\C}_\G\ar[r]^{\bar{\psi}}& \C
	}\]
	commutes. We prove the statement for $\tau$. The proof for $\sigma$ is essentially the same. 
	The objects in $\PP(\widehat{\G})$, and hence of $\widehat{\C}_\G$, are the edges $e\in E$. If  $\tau(e)=v$, define $\bar{\tau}(e) = v$. The square commutes on objects by definition. 
Let $e, e'\in\PP(\widehat{\G})$ be objects, where $e\colon u\to t$ and $e'\colon s\to v$. Let 
\begin{align*}\alpha\defeq &\{e\to e_1\to e_2 \to\cdots\to e_{m-1} \to e'\}, \quad\text{and}\quad\\
		\gamma\defeq &\{e\to f_1\to f_2 \to\cdots\to f_{k-1} \to e'\}
	\end{align*}
	be  morphisms in  $\PP(\widehat{\G})(e, e')$, such that $\widehat{\pi}(\alpha)=\widehat{\pi}(\gamma)$. Then, by Definition  \ref{Def:line-cat}, the compositions
	\[t =  a_1\xto{e_1}\cdots\xto{e_{m-1}} a_m = s, \quad\text{and}\quad t =  b_1\xto{f_1}\cdots\xto{f_{k-1}} b_k = s \]
	coincide in $\C$.  Thus
\[\pi(\tau(\alpha)) = t\xto{e_1}\cdots\xto{e_{m-1}}s\xto{e'} v\quad\text{coincides in $\C$ with}\quad \pi(\tau(\gamma)) = t\xto{f_1}\cdots\xto{f_{k-1}}s\xto{e'} v.\]
This shows that $\bar{\tau}$ is well defined and that the square commutes on morphisms, and hence the proof is complete. 
\end{proof}

Let $\cala$ be an abelian  category, and let $\C$ be a small category, with a generating quiver $\G$ and an associated line digraph $\widehat{\G}$.  The functors $\tau, \sigma\colon\widehat{\C}_\G\to\C$ induce \hadgesh{restriction functors} 
\[\tau^*, \sigma^*\colon \cala^\C\to \cala^{\widehat{\C}_\G},\]
where $\cala^\C$ and $\cala^{\widehat{\C}_\G}$ are the (abelian) functor categories from $\C$ and $\widehat{\C}_\G$ to $\cala$, respectively. We are now ready to define the gradient.

\begin{Defi}\label{Def:grad-general}
	Fix a small category $\C$ and an associated line category $\widehat{\C}$ with respect to some generating quiver  $\G$. Let $\cala$ be an abelian category. Let $\Gr(\cala^\C)$ denote the Grothendieck group of isomorphism classes of functors $\calc\to \cala$, with the operation $[F] + [G] = [F\sqcup G]$, where $(F\sqcup G)(c) = F(c)\sqcup G(c)$, and where $\sqcup$ denotes the coproduct in $\cala$.
	Define the \hadgesh{gradient}
	\[\nabla\colon \Gr(\cala^\C) \to \Gr(\cala^{\widehat{\C}})\]
	by \[\nabla[F] = [\tau^*(F)] - [\sigma^*(F)]\]
	for each functor $F\in\cala^\C$. Extend $\nabla$ by additivity to the whole Grothendieck group.
\end{Defi}

Next we show that the standard definition of a gradient for a digraph with vertex and edge weight functions taking values in the group of integers is a particular case of our setup.

\begin{Ex}\label{Ex:standard-grad}
	Let $\G=(V,E)$ be a quiver with line digraph $\widehat{\G}$,  consider $V$ and $E$ as discrete categories, and let $\tau, \sigma\colon E\to V$ be the source and target function, considered as functors. Let $\cala=\VVect$ be the category of finite dimensional vector spaces over a fileld $k$. Restriction of $\tau^*$ and $\sigma^*$ gives functors 
	\[\tau^*, \sigma^*\colon\cala^V\to\cala^{E}.\]
	 Thus $\Gr(\cala^V) = \Gr(kV)$ and $\Gr(\cala^E) = \Gr(kE)$ are the groups that consist  of functions on $V$ or $E$, respectively, which take a vertex or an edge to a finite dimensional $k$-vector space. 
	Since vector spaces are determined up to isomorphism by their dimension, these Grothendieck groups are easy to compute:
	\[\Gr(kV) = \Z^V\cong \bigoplus_{v\in V}\Z,\quad\text{and}\quad
	\Gr(kE) = \Z^E\cong \bigoplus_{e\in E}\Z.\]
	Thus the gradient is given by 
	\[\nabla[f](e) = \dim_k(f(\tau(e))) - \dim_k(\sigma(e)),\]
	where $[f]\in\Gr(kV)$. 
 
 Let $\varphi\colon V\to \Z$ be an arbitrary function that takes non-negative values. Considering $V$ again as a discrete category, let $f\colon V\to \VVect$ be the functor defined by $f(v) = k^{\varphi(v)}$. Then for any $e\in E$,
 \[\nabla[f](e) = \nabla(\varphi)(e),\]
 where the left hand side is the gradient of the element $[f]\in \Gr(\cala^V)$, and the right hand side is the graph theoretic definition for the gradient of the function $\varphi$, as in \textcolor{red}{(\ref{Eq:grad-div})}.
\end{Ex}

We next study some basic properties of the gradient. 

\begin{Prop}\label{Prop:Grad}
	Let $\C$ be a small category generated by a quiver $\G$, and let $\Cone_\G$ denote the associated line category. Then the gradient  $\nabla=\nabla_\C\colon \Gr(k\C)\to \Gr(k\Cone_\G)$ satisfies the following properties:
	\begin{enumerate}[(1)]
		\item $\nabla$ is a well defined group homomorphism. \label{Prop:Grad-a}
		\item If $ X\in\Gr(k\C)$ is locally constant then $\nabla X=0$.  \label{Prop:Grad-b}
	\end{enumerate}
Furthermore, $\nabla$ is natural with respect to restrictions, namely, if $\D\subseteq \C$  is a subcategory that is generated by a sub-quiver $\G'\subseteq \G$, and $\iota\colon\D\xto{\subseteq}\C$ is the inclusion, then $\iota^*\circ\nabla_\C= \nabla_\D\circ\iota^*$.
\end{Prop}
\begin{proof}
	Since $\tau^*, \sigma^*\colon \Gr(k\C)\to \Gr(k\Cone_\G)$ are homomorphisms of abelian groups,  so is their difference.  This proves Part \ref{Prop:Grad-a}.
	
	Let $M\in k\C\mod$ be locally constant. Define a natural isomorphism $\mu\colon \sigma^*M\to \tau^*M$, by sending an object  $e$ in $\Cone$ to the morphism
	\[\mu_{e}\colon \sigma^*M(e) = M(\sigma(e))\xto{M(e)} M(\tau(e))=\tau^*M(e).\]
	Naturality is clear.  It follows that  $\nabla[M]=0$. By definition, an element of $ X\in\Gr(k\C)$ is locally constant  if it is represented as a difference of locally constant modules. By additivity $\nabla$ vanishes on any locally constant virtual module in $\Gr(k\C)$. This proves Part \ref{Prop:Grad-b}. 
	
The last statement follows from commutativity of the square
\[\xymatrix{
k\C\mod\ar[r]^{\iota^*}\ar[d]^{\alpha^*} & k\D\mod\ar[d]^{\alpha^*}\\
k\Cone_\G\mod\ar[r]^{\iota^*} & k\widehat{\D}_{\G'}\mod
}\]
where $\alpha$ is either $\tau$ or $\sigma$.
\end{proof}

\subsection{The categorical divergence}
Next we define  the divergence operators in our context. For this we require that the category $\cala$ is abelian and bicomplete, namely that all small limits and colimits exist in $\cala$. In that case the functors $\tau^*$ and $\sigma^*$ have left and right adjoints given by the left and right Kan extensions. Recall that  finite products and coproducts coincide in  an abelian category. 

\begin{Defi}\label{Def:div-general}
	Fix a small category $\C$ and an associated line category $\widehat{\C}$ with respect to some generating quiver $\G$. Let $\cala$ be a small bicomplete abelian category. 
	Define the \hadgesh{left divergence} and the  \hadgesh{right divergence} (with respect to $\G$) 
	\[\nabla^*, \nabla_*\colon \Gr(\cala^{\widehat{\C}}) \to \Gr(\cala^{\C})\]
	by 
 \[\nabla^*[T] = [L_\tau(T)]-[L_\sigma(T)], \quad\text{and}\quad 
 \nabla_*[T] = [R_\tau(T)]-[R_\sigma(T)]\]
	for each functor $T\colon \widehat{\C}\to\cala$. Extend by additivity to the whole Grothendieck group.
\end{Defi}

Since limits commute with limits and colimits with colimits, and since finite products and coproducts coincide in an abelian category, both operators are well defined group homomorphisms.

Divergence in calculus on graphs is defined by the requirement that it must be adjoint to the gradient with respect to the inner product. Since in our context there is no direct analog to an inner product (but see Section \ref{Sec:pairings}), a justification of the definition is required. We proceed by showing how the standard divergence operator for digraphs, as in \textcolor{red}{(\ref{Eq:grad-div})}  with vertex and edge weight functions taking values in the integers, is a particular example of the general setup of Definition \ref{Def:div-general}.

	Let $\G=(V,E)$ be a quiver with line digraph $\widehat{\G}$. Let $\cala$ be an abelian bicomplete category.  
	Consider $V$ and $E$ as discrete categories, as in  Example \ref{Ex:standard-grad}.
	Restriction of $\tau^*$ and $\sigma^*$ gives functors 
	\[\tau^*, \sigma^*\colon\cala^V\to\cala^{E}.\]
	Since $V$ and $E$ are discrete, we have for $f\in\cala^E$
	\[L_\tau(f)(v) \defeq \colim_{\tau\downarrow v} f = \coprod_{e\in E\atop \tau(e) = v }f(e),\quad\text{and}\quad
	L_\sigma(f)(v) \defeq \colim_{\sigma\downarrow v} f = \coprod_{e\in E\atop \sigma(e) = v }f(e).\]
	Similarly 
	\[R_\tau(f)(v) \defeq \lim_{v\downarrow\tau} f = \prod_{e\in E\atop \tau(e) = v }f(e),\quad\text{and}\quad
	R_\sigma(f)(v) \defeq \lim_{v\downarrow\sigma} f = \prod_{e\in E\atop \sigma(e) = v }f(e).\]
	
\begin{Ex}\label{Ex:standard-div}	Restrict attention to the case where $\cala=\VVect$ and the quiver $\G$ is finite. Then, finite products (cartesian products) and finite coproducts (direct sums) are isomorphic. Hence,  the right and left Kan extensions coincide. 
	As in Example \ref{Ex:standard-grad}, we have
	\[\Gr(kV) = \Z^V\cong \bigoplus_{v\in V}\Z,\quad\text{and}\quad
	\Gr(kE) = \Z^E\cong \bigoplus_{e\in E}\Z.\]
	Hence 
	\[\nabla^*[f](v) = \nabla_*[f](v) = \sum_{e\in E\atop \tau(e) = v } \dim_k f(e) -\sum_{e\in E\atop \sigma(e) = v } \dim_k f(e)\]
	for any $f\colon E\to \VVect$. 
 
 Let $\gamma\colon E\to\Z$ be an arbitrary function that takes non-negative values. Considering $E$ as a discrete category again, let $g\colon E\to \VVect$ denote the functor defined by $g(e) = k^{\gamma(e)}$. Then 
 \[\nabla^*[g] = \nabla_*[g] = \nabla^*(\gamma)\]
 where the left and centre in the equation are the left and right divergence of $[g]$ in $\Gr(kV)$, and the right hand side is the graph theoretic divergence of $\gamma$, as in \textcolor{red}{(\ref{Eq:grad-div})}.
 \end{Ex}
 
Notice that restricting to finite quivers is essential because 
the category of finite dimensional vector spaces is only finitely bicomplete.


\section{Bilinear pairings on $\Gr(k\C)$}
\label{Sec:pairings}
In this section we define two bilinear forms on $\Gr(k\C)$, where  $\C$ is a finite acyclic category (i.e., a category generated by a finite acyclic quiver). 

\begin{Defi}\label{Def:Hom-pairing}
Let $\C$ be a finite  category. Define   the \hadgesh{Hom pairing}
\[\langle -,-\rangle_\C\colon\Gr(k\C)\times \Gr(k\C)\to \Z\]
	by
	\[\langle [M],[N]\rangle_\C \defeq \dim_k\left(\Hom_{k\C}(M,N)\right),\]
	where $M$ and $N$ represent their isomorphism classes in $\Gr(k\C)$. Extend the definition to the full Grothendieck group by additivity.
\end{Defi}

Notice that finiteness of the category $\C$ guarantees that  Hom objects are finite dimensional, so the pairing is well defined. The Hom pairing is clearly bilinear and can be defined by analogy also on $\Gr(k\widehat{\C}_\G)$, where $\G$ is a finite quiver that generates $\C$. Since the left and right Kan extensions are left and right adjoints to the restrictions \(\tau^*\) and \(\sigma^*\) we have
\begin{equation}
\langle \nabla^* X, Y\rangle_\C = \langle  X, \nabla Y\rangle_{\widehat{\C}},\quad\text{and}\quad
\langle \nabla X, Y\rangle_{\widehat{\C}} = \langle  X, \nabla_* Y\rangle_{\C}
\label{Eq:Adj}
\end{equation}
for $ X\in\Gr(k{\widehat{\C}})$ and $ Y\in\Gr(k{\C})$. Notice that the Hom pairing is not  symmetric in general  (however, see Example \ref{Ex:pairing-discrete} below).

Next we define another useful pairing on $\Gr(k\C)$. This pairing requires that $\C$ is finite and acyclic.

\begin{Defi}\label{Def:Euler-pairing}
	Let $\C$ be a finite acyclic category. Define   the \hadgesh{Euler pairing}
	\[\chi_\C\colon\Gr(k\C)\times \Gr(k\C)\to \Z\]
	by
	\[\chi_\C([M],[N]) \defeq \chi(\Ext_{k\C}^*(M,N)) = \sum_{n\geq 0}(-1)^n\dim_k(\Ext^n_{k\C}(M,N)),\]
	where $M$ and $N$ represent their isomorphism classes in $\Gr(k\C)$, and extend to the full Grothendieck group by additivity.
\end{Defi}

To show that this is well defined, we must argue that the $\Ext^n$ groups vanish for $n$ sufficiently large. This is implied by Lemma 
\ref{Lem:Euler-well-def} below (see also \cite{Xu}, Theorem 4.2.4). If $\C$ is an acyclic category and $\varphi$ is a morphism in $\C$ the \hadgesh{composition length of $\varphi$} is the minimal integer $n$ such that $\varphi$ can be expressed as a composition of $n$ irreducible morphisms. If, in addition the category $\C$ is finite, then the composition length of $\C$ is defined to be the largest composition length of any of its morphisms.

\begin{Lem}\label{Lem:Euler-well-def}
Let $\C$ be a finite acyclic   category. Then the category algebra $k\C$ has finite projective dimension.
\end{Lem}
\begin{proof}
For each object $x\in \C$, let $F_x\in k\C\mod$ denote the indecomposable projective module defined by 
\[F_x = k\Mor_\C(x,-),\]
namely $F_x(y) = k$ if there is a morphism $x\to y$ and $F_x$ of any such morphism is the identity, and $F_x$ takes the value $0$ on objects $z$, where  $\C(x,z)=\emptyset$.    Let $M$ be a finitely generated $k\C$-module. For each object $x\in\C$, choose a basis $\{v^x_1, \ldots, v^x_{n_x}\}$ for $M(x)$.  For $1\le i\le n_x$ let 
\[q^x_{i}\colon F_x\to M\]
denote the natural transformation determined by taking $1_x\in\C(x,x)$, to $v^x_{i}$. Thus one obtains a surjection of $k\C$-modules
\[q_M\colon\bigoplus_{x\in\C}F_x^{d_x}\to M,\] 
where $d_x = \dim_k(M(x))$ and $F_x^{d_x} = \bigoplus_{d_x}F_x$.
Let $M_1$ denote the $\Ker(q_M)$. Since $\C$ is finite an acyclic, it has minimal objects in the sense that $\Mor_\C(-,x)=\emptyset$, and  if $x$ is such an object, then $M_1(x)=0$. Thus $M_1$ vanishes on all minimal objects in $\C$. Notice that if $d_x=0$, then no copy of $F_x$ is included in the cover. Next, construct a projective cover of $M_1$ in a similar fashion, avoiding all minimal objects in $\C$, and all those of composition distance $1$ from a minimal object, for which $M_1$ vanishes. Let $M_2$ be the kernel of this cover. Then $M_2$ vanishes on all minimal objects and all objects of composition distance $1$ from a minimal object. Since $\C$ is assume to be finite and acyclic, its composition length is bounded. Hence, by induction   we obtain a finite projective resolution for $M$ of length bounded above by  the composition length of  $\C$.
\end{proof}

\begin{Rem}\label{Rem:pairing-with-Proj-Inj} 
The Hom and Euler pairings are generally not the same, but if $P, I\in k\P\mod$ are modules with $P$ projective and $I$ injective, and $X\in \Gr(k\P)$ is any element, then 
\[\langle[P], X\rangle_\P = \chi_\P([P], X),\quad\text{and}\quad \langle X, [I]\rangle_\P = \chi_\P(X, [I]).\]
This is because for any module $M\in k\P\mod$, $\Ext^i(M, I) = \Ext^i(P, M) = 0$ for all $i>0$. 
\end{Rem}

The pairing $\chi_\C$ is clearly bilinear, but it is not symmetric in general, and in all but some very special cases,
\begin{equation}
\chi_\C(\nabla^* X,   Y) \neq \chi_{\widehat\C}( X, \nabla Y),\quad\text{and}\quad
\chi_\C( X,  \nabla_* Y) \neq \chi_{\widehat\C}(\nabla X,  Y).
\label{Eq:Adj-chi}
\end{equation}
This is the case because the Kan extensions that define the divergence homomorphisms are not generally exact.

A special case occurs when the category $\C$ is generated by a finite  acyclic quiver, due to the fact that path algebras over acyclic quivers are hereditary (namely submodules of projective modules are projective) \cite[Theorem 2.3.2]{Derksen-Weyman}. Hence for any $M, N\in k\C$, one has $\Ext_{k\C}^i(M,N)=0$ for $i>1$. Thus in this case
\[\chi_\C([M],[N]) = \dim_k \Hom_{k\C}(M,N) - \dim_k \Ext^1_{k\C}(M,N),\]
and the Euler pairing is rather easy to compute,
as it depends only on the Hilbert function of the modules involved. To make this statement precise, we need the following definition.

\begin{Defi}\label{Def:Euler-form}
Let $\Q = (V, E)$ be a finite  quiver. Let $f, g \in \Z^V$ be integer valued functions  on $V$.  Define a function $\chi_\Q\colon \Z^V\times \Z^V\to \Z$
by
\[\chi_\Q(f,g) \defeq \sum_{x\in V} f(x)g(x) - \sum_{e\in E}f(\sigma(e))g(\tau(e)),\]
where $\sigma$ and $\tau$ are the source and target functions on $E$.
The function $\chi_\Q$ is called the \hadgesh{Euler form for $\Q$}.
\end{Defi}

\begin{Lem}
\label{Lem:Euler-form}
Let $\C$ be a  category generated by a finite acyclic quiver $\Q$. Then for any $X, Y \in k\C\mod$
\[\chi_\C( X, Y) = \chi_\Q(\Dimk_\Q X, \Dimk_\Q Y).\]
\chg{Furthermore, in that case the Euler pairing is non-singular.}
\end{Lem}
\begin{proof}
\chg{The first statement is well known and appears for instance as \cite[Proposition 2.5.2]{Derksen-Weyman}. For the second, notice that if $A$ is the adjacency matrix of an arbitrary quiver $\Q$ with respect to some ordering on its vertices and $f, g$ are real valued functions of its vertices considered as vectors with the same vertex order, then the Euler form is given by $\chi_\Q(f, g) = f^t(I-A)g$. Hence, the Euler form is non-singular if and only if $1$ is not an eigenvalue of $A$.   If $\Q$ is acyclic, then this can easily be observed as follows. Order the vertices of $\Q$ such that arrow always point from a lower index to a higher one. With this ordering the adjacency matrix is upper triangular with only $0$ as diagonal entries. Hence $I-A$ is upper triangular with only $1$ as diagonal entries, and thus it is non-singular.  The claim follows.}
\end{proof}

Notice that if the quiver $\Q$ is a finite tree, then the category it generates  is in fact a poset. Also, the lemma shows that symmetry almost always fails, even under rather favourable circumstances, because the second term in the definition of the Euler form is not symmetric. 

\chg{\begin{Lem}\label{Lem:Euler-adjointness}
Let $\P$ be a finite poset whose Hasse diagram  is a rooted tree. Then for any $X\in\Gr(k\P)$ and $Y\in\Gr(\Pone)$,
\[\chi_{\P}(X, \nabla_* Y) = \chi_{\Pone}(\nabla X, Y).\]
Similarly, if the Hasse diagram of $\P^{\op}$ is a rooted tree, then
\[\chi_{\P}(\nabla^*Y, X) = \chi_{\Pone}(Y, \nabla X).\]
\end{Lem}}
\begin{proof}
\chg{Let $\P$ be a poset whose Hasse diagram $\HH_\P$ is a rooted tree and consider  the functors $\tau^*, \sigma^*\colon k\P\mod\to k\Pone\mod$. We claim that in this case both functors  send projective modules to projective  modules. To do it suffices to consider their effect on indecomposable projective modules of the form $F_v$  for an arbitrary object $v\in\P$ (See Definition \ref{Def:Proj-Inj-Simp}\ref{Def:Proj-Inj-Simp-Proj}).}

\chg{Let $v\in\P$ be any object. Then 
\[\sigma^* F_v(x,y) = F_v(\sigma(x,y)) = F_v(x) = \begin{cases}
										k & x\ge v\\
										0 & x<v
									\end{cases}
\]
Let $u_1,\ldots u_k$ be the successors of $v$.  Then $(v,u_j)$ non-comparable objects in $\Pone$, and $x\ge v$ if either $x=v$, in which case $y=u_j$ for some $j$, or $x>v$, in which case $x\ge u_j$, and $(x,y)\ge (v,u_j)$ in $\Pone$. It follows that $\sigma^*F_v = \bigoplus_{j=1}^kF_{(v,u_j)}$,   so $\sigma^*F_v$ is projective for all $v\in \P$.}

\chg{Next, we have 
\[\tau^* F_v(x,y) = F_v(\tau(x,y)) = F_v(y) = \begin{cases}
										k & y\ge v\\
										0 & y<v
									\end{cases}
\]
If $v$ is not the root, let $u$ be its predecessor. If $y =v$, then $x=u$, by uniqueness of predecessors in rooted trees. On the other hand if $y>v$, then $x\ge v$, or else $x$ and $v$ are not comparable, but both have $y$ as a successor, which contradicts $\HH_\P$ being a rooted tree. Thus, in this case as well $(x,y)\ge (u,v)$ and $\tau^* F_v = F_{(u,v)}$. If $v$ is the root with successors $u_1, \ldots u_k$,  then $y>v$, since it has a predecessor $x$ and $v$ is minimal. Hence either $y=u_j$ for some $j$, in which case $x=v$, or $y>u_j$, in which case $x\ge u_j$ by the same argument as before. Hence, $\sigma^*F_v = \bigoplus_{j=1}^kF_{(v,u_j)}$. In either case $\tau^* F_v$ is a projective module for all $v\in\P$. }

\chg{By  Lemma \ref{Lem:gen-cogen}, the right Kan extensions $R_\sigma$ and $R_\tau$ are exact. Hence by Lemma \ref{Lem:Ext-adjointness} for any modules $M\in k\P\mod$ and $N\in k\Pone\mod$, 
\[\Ext^*_{k\P}(M, R_\sigma(N)) \cong\Ext^*_{k\Pone}(\sigma^*M, N)\quad\text{and}\quad \Ext^*_{k\P}(M, R_\tau(N)) \cong\Ext^*_{k\Pone}(\tau^*M, N).\]
It follows that 
\begin{align*}\chi_\P([M], \nabla_*[N])\defeq & \chi(\Ext^*_{k\P}(M, R_\tau N)) - \chi(\Ext^*_{k\P}(M, R_\sigma N)) =\\
& \chi(\Ext^*_{k\Pone}(\tau^*M, N)) - \chi(\Ext^*_{k\P}(\sigma^*M, N)) \defeq \chi_{\Pone}(\nabla[M], [N]).
\end{align*}}

\chg{The proof in the case where the Hasse diagram of $\P^{\op}$ is a rooted tree is very similar. One shows in this case that $\sigma^* G_v \cong G_{(v,u)}$ where $u$ is the unique successor of $v$, if $v$ is not maximal, and if $v$ is maximal $\sigma^* G_v \cong \bigoplus_{j=1}^kG_{(u_j,v)}$, where $u_j$ are all the predecessors of $v$. One then shows further that $\tau^* G_v \cong \bigoplus_{j=1}^k G_{(u_j,v)}$ where $u_j$ are again all predecessors of $v$. Then in all cases $\sigma^*$ and $\tau^*$ take injective modules to injective modules. The rest follows similarly, using Lemmas \ref{Lem:gen-cogen} and \ref{Lem:Ext-adjointness}.}
\end{proof}

\chg{The following easy example that shows that the hypotheses of Lemma \ref{Lem:Euler-adjointness} cannot be relaxed.} 

\chg{\begin{Ex}\label{Ex:Unrooted}
Let $\P$ denote the poset with objects $0,1, 2, 3$ with relations $0, 1 < 2 < 3$. The module $F_1\in k\P\mod$ is indecomposable projective, but $\tau^* F_1$ is the constant module on $\Pone$ and is not projective. Similarly, let $\Q$ denote the poset $\P^{\op}$. Then $G_2\in k\Q\mod$ is indecomposable injective, but  $\sigma^* G_2$ is the constant module on $\widehat{\Q}$ and is not injective. 
\end{Ex} }

Once again, it makes sense to compare these constructions to the corresponding concept in discrete calculus on graphs.  The next example shows that in this case Hom pairing and the Euler pairing coincide and in the appropriate setting they are equal to the inner product pairing.

\begin{Ex}\label{Ex:pairing-discrete}
	As in  Example \ref{Ex:standard-div}  consider the sets $V$ and $E$ as discrete categories. Hence the objects in the functor categories   $kV\mod$ and $kE\mod$  are  finite dimensional $k$-vector spaces indexed by $V$ and $E$, respectively. 
	\begin{align*}
		\chi_V([f],[g])  = \langle[f],[g]\rangle_V = \dim_k(\Hom_{kV}(f,g)) = & \sum_{v\in V}\dim_k(f(v))\dim_k(g(v)),\quad\text{and}\quad\\
		\chi_E([h],[s]) = \langle[h],[s]\rangle_E = \dim_k(\Hom_{kE}(h,s)) = &\sum_{e\in E}\dim_k(h(e))\dim_k(s(e))
	\end{align*}
The first equality in both equations results from the fact that all vector spaces are free modules over $k$, so all   $\Ext^i$ groups for $i\ge 1$ vanish.	Thus, if we let $\gamma', \gamma''\colon V\to \Z$ and $\delta', \delta''\colon E\to \Z$ be arbitrary functions with nonnegative values, and we define $g', g''\colon V\to \VVect$ and $h', h''\colon E\to \VVect$ to be the functors defined by $g'(v) = k^{\gamma'(v)}$ and $g''(v) = k^{\gamma''(v)}$, and similarly define $h',  h''$ using $\delta', \delta''$, then 
\begin{align*}
		\chi_V([g'],[g'']) &= \langle[f],[g]\rangle_V = (\gamma', \gamma''),\quad\text{and}\\
		\chi_E([h'],[h'']) & =  \langle[h],[s]\rangle_E = (\delta', \delta''),
	\end{align*}	
where  the right hand side in both equations denotes the ordinary inner product with the corresponding functions considered as vectors.
Notice that the \hadgesh{pairings in this case are commutative}. Also, by Example \ref{Ex:standard-div} the divergence operators  $\nabla^*$ and $\nabla_*$  coincide  in this case. In particular, the adjointness relations \textcolor{red}{(\ref{Eq:Adj})} trivially hold.
\end{Ex}


\section{The gradient for modules over posets}
\label{Sec:Grad}
From this section  on we specialise the ideas discussed so far to the case where the categories under investigation are finite posets.  Finiteness is not required in all statements to follow, but this is the most basic and best behaved case of our theory, and is also related to the theory of generalised persistence modules, and hence deserves particular attention.  

Any poset $\P$ is uniquely determined by its associated \hadgesh{Hasse diagram} $\calh_\P$, which is the transitive reduction of the acyclic digraph underlying the poset \(\P\). Namely the digraph $\calh_\P$ has as vertices the objects of $\P$, and directed edges are the non-identity irreducible relations in $\P$, that is, those relations $u<v$ in $\P$ such that  there is no intermediate  relation $u< y < v$. The corresponding edge in the digraph $\calh_\P$ is then denoted by $(u,v)$. One recovers $\P$ by taking the transitive closure of \(\calh_\P\).

We start by observing that the line category of a poset is again a poset.

\chg{\begin{Lem}
	\label{Lem:P-hat}
	Let $\P$ be a  poset. Then   $\widehat{\P}$ is also a poset.
\end{Lem}}
\begin{proof}
\chg{Let $\HH_\P$ be the Hasse diagram of $\P$. Then   $\widehat{\calh}_\P$ is the Hasse diagram of the line category $\Pone$, as in Definition \ref{Def:line-cat}. To prove this, note that a  Hasse diagram is, by definition,  an acyclic and transitively reduced digraph. Thus it suffices to show that if $\calh = (V,E)$ is a  transitively reduced acyclic digraph, then $\widehat{\calh}$ is also acyclic and transitively reduced.}
	
	\chg{It is immediate from the definition that the existence of a cycle in $\widehat{\calh}$ implies that $\calh$ itself contains a cycle. Let $(u,v,w)$ be an edge in $\widehat{\calh}$, and assume that it can be decomposed as $(u,v,z)\cdot(v,z,w)$ (See Definition \ref{Def:line-digraph} for the notation). Then $\calh$ contains the composable sequence}
	
\begin{center}
			\begin{tikzcd}
			u\rar{(u,v)}  & v \rar{(v,z)}  & z \rar{(z,w)} & w
		\end{tikzcd}
	\end{center}
\chg{	Hence either $(u,v)$ is not composable with $(z,w)$ or $v=z$. The first option contradicts the assumption that $(u,v,w)$ is an edge in $\widehat{\calh}$ and the second option contradicts the assumption that $(v,z,w)$ is an edge in $\widehat{\calh}$, or alternatively acyclicity of $\calh$ due to the self-loop \((v,v)\). Thus $(u,v,w)$ is irreducible, and so $\widehat{\calh}$ is transitively reduced.  }
\end{proof}

\chg{\begin{Rem}\label{Rem:Restrict-to-Hasse}
Let $\P$ be a finite poset and let $\HH_\P$ denote its Hasse diagram. Let $\G_P$ denote the poset $\P$ regarded as a quiver. Then  $\HH_\P$ is a generating digraph for $\P$ in the sense of Definition \ref{Def:generating-digraph}, as is any quiver $\HH_\P\subseteq \G\subseteq \G_\P$ and it is possible to consider resulting line categories $\Pone_\G$ (Definition \ref{Def:line-cat}) and the corresponding differential operators with respect to them. However, throughout the rest of this article we will consider only $\HH_\P$ as a generating quiver for $\P$.
\end{Rem}}

Figure \ref{Fig:gradient} illustrates some small posets $\P$, and for $M\in k\P\mod$  the associated target and source modules in $k\Pone\mod$, the difference of which is defined to be the gradient 
\[\nabla\colon \Gr(k\P)\to \Gr(k\Pone).\]

\begin{figure}[h!]
	\[\xymatrix{
		{x}\ar[r]^{x<y} & {y} && {\tau^*M(x,y)\atop=M(y)} && {\sigma^*M(x,y)\atop=M(x)} \\
		x\ar[r]^{x<y} &  y \ar[r]^{y<z} & z & {\tau^*M(x,y)\atop=M(y)}\quad\ar[r]^{M(y<z)} &\tau^*M(y,z)\atop= M(z)& {\sigma^*M(x,y)\atop=M(x)}\quad\ar[r]^{M(x<y)} &\sigma^*M(y,z)\atop= M(y)\\
		&&  z &&{\tau^*M(y,z)\atop=M(z)} && {\sigma^*M(y,z)\atop=M(y)} \\
		x\ar[r]^{x<y}&  y \ar[ru]^{y<z}\ar[rd]_{y<w}&&{\tau^*M(x,y)\atop=M(y)} \ar[ru]^{M(y<z)}\ar[rd]_{M(y<w)}&&{\sigma^*M(x,y)\atop=M(x)} \ar[ru]^{M(x<y)}\ar[rd]_{M(x<y)}\\
		&&  w &&{\tau^*M(y,w)\atop=M(w)}  &&{\sigma^*M(y,w)\atop=M(y)}\\
		&  y \ar[rd] && {\tau^*M(x,y)\atop=M(y)}\ar[r]^{M(y<w)}& {\tau^*M(y,w)\atop=M(w)}& {\sigma^*M(x,y)\atop=M(x)}\ar[r]^{M(x<y)}& {\sigma^*M(y,w)\atop=M(y)}\\
		x \ar[ru]\ar[rd] &&  w\\
		&  z \ar[ru] && {\tau^*M(x,z)\atop=M(z)}\ar[r]^{M(z<w)}& {\tau^*M(z,w)\atop=M(w)}&  {\sigma^*M(x,z)\atop=M(x)}\ar[r]^{M(x<z)}& {\sigma^*M(z,w)\atop=M(z)}
	}\]
	\caption{Left: Four sample posets given by their respective Hasse diagrams. Centre and Right: For an arbitrary  module $M\in k\P\mod$, the corresponding modules $\tau^*M$ and $\sigma^*M$  in $k\Pone\mod$. The gradient is given by the formal difference $[\tau^*M]-[\sigma^*M]$ in $\Gr(k\Pone)$ }
	\label{Fig:gradient}
\end{figure}
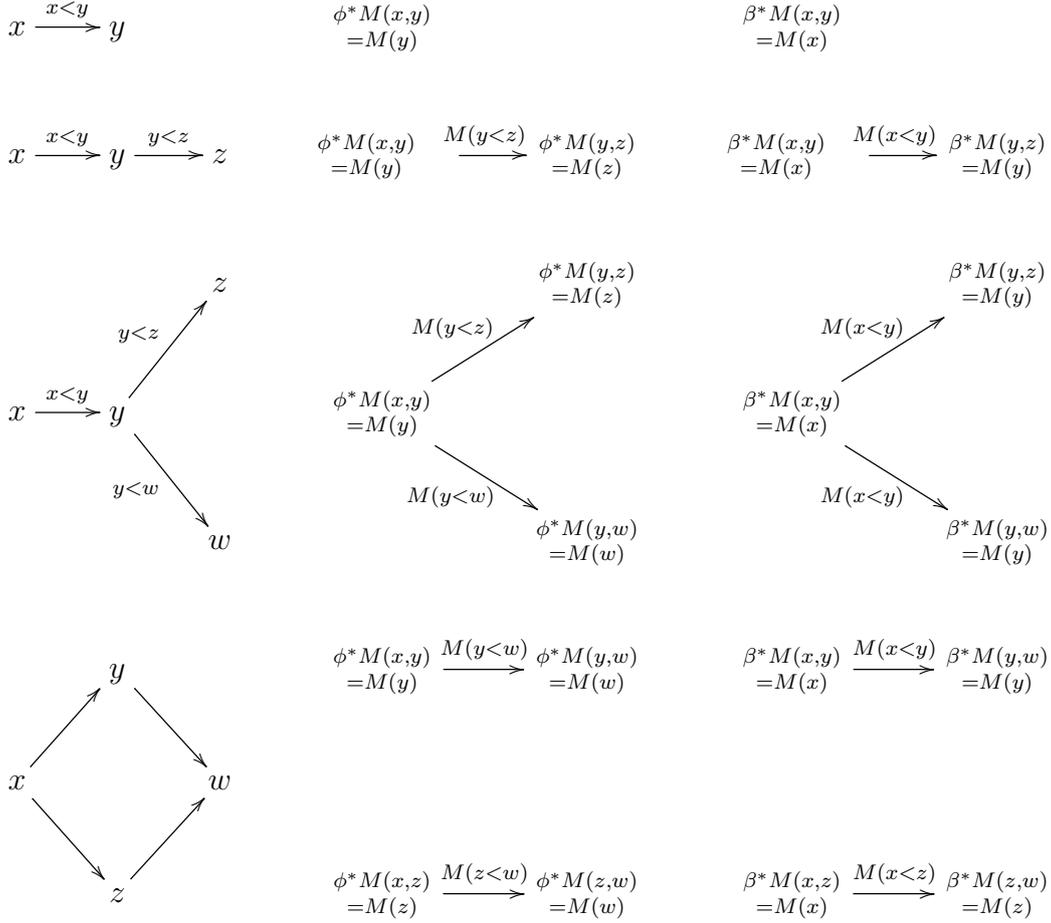

\chg{An important advantage that our approach to modules over poset offers is the possibility to study them locally. We next make this precise.  Let $\P$ be a finite poset, and let $\Q$ be a subgraph of $\HH_\P$. Let $\P_\Q\subseteq \P$ be the sub-poset generated by $\Q$, and let $\iota_\Q$ denote the inclusion functor. Functoriality of the line digraph construction gives an inclusion $\widehat{\iota}_\Q\colon\widehat{\P_\Q}\to \Pone$. Furthermore, one easily verifies that the inclusions are compatible with the functors $\tau$ and $\sigma$, namely that  $\iota_\Q\circ\alpha_\Q = \alpha\circ\widehat{\iota}_\Q$, where $\alpha = \tau\,\text{or}\,\sigma$.
Thus one gets homomorphisms on the respective Grothendieck groups, 
and hence  commutative square of group homomorphisms
\[\xymatrix{
	\Gr(k\P)\ar[r]^{\iota_\Q^*}\ar[d]^\nabla & \Gr(k\P_\Q)\ar[d]^\nabla\\
	\Gr(k\Pone)\ar[r]^{\widehat{\iota}_\Q^*} & \Gr(k\widehat{\P_\Q})
}\]}
 
 \chg{\subsection{Modules with a vanishing gradient} We now consider the consequences for a module $M\in k\P\mod$ of having a gradient that vanishes locally. The next definition makes this precise.}
 
\begin{Defi}\label{Def:Grad-0-on-tree}
	Let $\P$ be a finite  poset, and let $\Q\subseteq\HH_\P$ be a subgraph. We say that an element $ X\in\Gr(k\P)$ has a \hadgesh{vanishing gradient on $\Q$} if 
	\[\nabla(\iota_\Q^* X) = \widehat{\iota}_\Q^*(\nabla X)= 0.\]
\end{Defi}

Notice that for a module $M\in k\P\mod$, $\iota_\Q^*[M]$ is just the restriction of $M$ to the sub-poset $\P_\Q$, which justifies Definition \ref{Def:Grad-0-on-tree}. Notice also that if $\nabla X=0$ in $\Gr(k\Pone)$, then $\nabla(\iota^*_\Q X)=0$ as well, namely the restriction of $X$ to $\P_\Q$ has a vanishing gradient on $\Q$. When $\Q$ is a tree, and more specifically a line-connected tree, then the vanishing of the gradient of a module on $\Q$ has strong consequences, as Theorem \ref{Th:grad0} shows. We are now ready to restate and prove the theorem.

\begin{Thm}\label{Thm:grad0}
	Let $\P$ be a finite poset, and let $\T\subseteq\HH_\P$ be a line connected sub-tree.  Let $\P_\T\subseteq \P$  denote the sub-poset generated by $\T$. Let $M\in k\P\mod$ be a module,  and let $M_\T$ denote the restriction of $M$ to $\P_\T$. Assume that $\nabla[M_\T]=0$ in $\Gr(k\P_\T)$. Then the following statements hold.
	\begin{enumerate}[(1)]
		\item 
		For any pair of objects $u,v\in\P_\T$, there is an isomorphism $\alpha_{u,v}\colon M(u)\to M(v)$, such that $\alpha_{u,u} = 1_{M(u)}$ and  $\alpha_{v,w} \circ \alpha_{u,v} = \alpha_{u,w}$. 
		\label{Thm:grad0-1}
		\item 
		For every pair of irreducible relations $u \le w$  and $s \le t$ in $\P_\T$,  
		\[
		\alpha_{w,t}\circ M(u\le  w) = M(s\le  t)\circ \alpha_{u,s}.
		\]
		\label{Thm:grad0-2}
		\item
		$M_\T$ is locally constant  if and only if $M(u\le v)$ is an isomorphism for some irreducible relation $u\le v$ in $\P_\T$.
		\label{Thm:grad0-3}
	\end{enumerate}
Furthermore, if $M\in k\P\mod$ is a module such that for any irreducible relation $u\le v$ in $\P_\T$ there exists an isomorphism $\alpha_{u,v}\colon M(u)\to M(v)$ that satisfy  Statements \ref{Thm:grad0-1} and \ref{Thm:grad0-2}, then $M$ has a vanishing gradient on $\T$.
\end{Thm}

\begin{proof}
By hypothesis $\nabla[M_\T]=0$. Hence there is a natural isomorphism $\alpha\colon \sigma^*M_\T\xto{} \tau^*M_\T$. If $(u,v)$ is an edge in $\T$, and hence a vertex in $\Pone_\T$, then we have 
	\[\alpha_{u,v}\colon \sigma^* M_\T(u,v) = M(u)\to M(v)=\tau^* M_\T(u,v),\] that is, $\alpha_{u,v}$ is the evaluation of the natural isomorphism $\alpha$ on the vertex $(u,v)$. Define 
	\[\alpha_{v,u} \defeq \alpha_{u,v}^{-1}.\] Since $\P_\T$ is a poset and $u\lneq v$, the opposite inequality does not hold, so $(v,u)$ is not a vertex in $\widehat{\P_\T}$. Thus, $\alpha_{v,u}$ is well defined.  Notice that if $\T$ is a maximal sub-tree of $\HH_\P$, then  the vertices of $\P_\T$ coincide with those of $\P$.
	
	Let $v_0$ be any minimal object in $\P_\T$. Then $\alpha_{v_0,u}$ is defined and is an isomorphism for each irreducible relation $v_0<u$ and, by induction on the length of a chain starting at $v_0$, $\alpha_{u,v}$ is defined and is an isomorphism for any $u,v\in \P_{\T_{\ge v_0}}$. Since  $\P_\T$ has only finitely many minimal objects, it follows that $\alpha_{u,v}$ is an isomorphism on any sub-poset of the form $\P_{\T_{\ge v_0}}$ where $u_0$ is minimal. 
	
	Let $u_0, u_0'$ be two distinct minimal objects in $\P_\T$. Since $\P_\T$ is  line connected by construction, it is in particular connected. Hence there are minimal objects $u_0=v_1, v_2,\ldots, v_k=u_0'$, such that for each $1\le i\le k-1$, the intersection of sub-posets $\P_{\T_{\ge v_i}}\cap \P_{\T_{\ge v_{i+1}}}$ is nonempty. Let $a_i\in\P_{\T_{\ge v_i}}\cap \P_{\T_{\ge v_{i+1}}}$ be any object. Thus  one has an isomorphism
	\[\alpha_{a_i,v_{i+1}}\circ\alpha_{v_i,a_i} = \alpha_{v_{i+1},a_i}^{-1}\circ\alpha_{v_i,a_i}\colon M(v_i)\to M(v_{i+1}).\]
	Since between any two objects in $\P_\T$ there is by assumption at most one path in $\T$, compositions of isomorphisms of the form $\alpha_{x,y}$, where $(x,y)$ is an object in $\Pone_\T$, and the inverses of such isomorphisms define $\alpha_{u,v}$ for any pair of objects $u,v\in\P_\T$. Since $\P$ and $\P_\T$ coincide on objects, this proves Part \ref{Thm:grad0-1}.

	Let $(u_1,u_2,u_3)$ be an edge in $\Pone_\T$. Then we have a commutative square
	\[\xymatrix{
		M(u_1) \ar@{=}[r] & \sigma^*M_\T(u_1,u_2)\ar[d]_\cong^{\alpha_{u_1,u_2}} \ar[rr]^{M_\T(u_1\le u_2)} && \sigma^*M_\T(u_2,u_3)\ar[d]^{\alpha_{u_2,u_3}}_\cong  \ar@{=}[r] & M(u_2)\\
		M(u_2) \ar@{=}[r] & \tau^*M_\T(u_1,u_2) \ar[rr]^{M_\T(u_2\le u_3)} && \tau^*M_\T(u_2,u_3) \ar@{=}[r] & M(u_3)\\
	}\]
	Thus $M(u_2\le u_3) = \alpha_{u_2,u_3}\circ M(u_1\le u_2)\circ \alpha_{u_2,u_1}$.
	By induction, if $u_1\le u_2\le \cdots \le u_n$ is any chain of irreducible relations in $\P_\T$, then 
	\[M(u_{n-1}\le u_n) = \alpha_{u_2,u_n}\circ M(u_1\le u_2)\circ \alpha_{u_{n-1},u_1}.\]
	
	Let $u\le v$ and $a\le b$ be two irreducible relations in $\P_\T$. Let $u_0, a_0$ be   two minimal objects in $\P_\T$, such that $u_0\le u$ and $a_0\le a$. Let
	\[u_0\le u_1\le\cdots\le u_n\le u\le v,\quad\text{and}\quad a_0\le a_1\le \cdots\le a_k\le a\le b\]
	be chains of irreducible relations in $\P_\T$. Then 
	\[M(u\le v) = \alpha_{u_1,v}\circ M(u_0\le u_1)\circ\alpha_{u,u_0},\quad\text{and}\quad 
	M(a\le b) = \alpha_{a_1,b}\circ M(a_0\le a_1)\circ\alpha_{a,a_0}.\]
	Thus it suffices to prove that 
	\[M(u_0\le u_1) = \alpha_{a_1,u_1}\circ M(a_0\le a_1)\circ \alpha_{u_0,a_0}.\]
	Notice that $(u_0,u_1)$ and $(a_0,a_1)$ are minimal objects in $\Pone_\T$. Since $\Pone_\T$ is connected, there is a sequence of minimal objects in $\Pone_\T$
	\[(u_0,u_1) = (x_1,y_1), (x_2,y_2),\ldots, (x_r,y_r)=(a_0,a_1)\]
	such that $\Pone_{\T_{\ge (x_i,y_i)}}\cap \Pone_{\T_{\ge (x_{i+1},y_{i+1})}}$ is nonempty for each $1\le i\le r-1$. Let $(s_i,t_i)$ be an object in the intersection, so that $(s_i,t_i)\ge (x_i,y_i), (x_{i+1},y_{i+1})$. It follows that 
	\[\alpha_{y_i,t_i}\circ M(x_i\le y_i)\circ\alpha_{s_i,x_i}= M(s_i\le t_i) = \alpha_{y_{i+1},t_i}\circ M(x_{i+1}\le y_{i+1})\circ\alpha_{s_i,x_{i+1}}.\]
	Thus 
	\[M(x_{i+1}\le y_{i+1}) = \alpha_{t_i,y_{i+1}}\alpha_{y_i,t_i}\circ M(x_i\le y_i)\circ\alpha_{s_i,x_i}\alpha_{x_{i+1},s_i} = \alpha_{y_i,y_{i+1}}\circ M(x_i\le y_i)\circ\alpha_{x_{i+1},x_i}\]
	for each $1\le i\le r-1$. Part \ref{Thm:grad0-2} follows by induction on $r$. 
	
	Part \ref{Thm:grad0-3} follows at once from \ref{Thm:grad0-2}.

 Finally, assume that $M\in k\P\mod$ satisfies Statements \ref{Thm:grad0-1} and \ref{Thm:grad0-2}. Let $\alpha$ be a family of morphisms for which these statements hold. By   \ref{Thm:grad0-1}, for each object $(u,v)\in \widehat{\P_\T}$ we have an isomorphism 
\[\alpha_{u,v} \colon M(u) = \sigma^* M (u,v)\to \tau^* M(u,v) = M(v),\]
If $(u,v)\le (x,y)$ in $\widehat{\P_\T}$, then by Condition \ref{Thm:grad0-2} 
\[\alpha_{v,y}\circ M(u<v) = M(x< y)\circ\alpha_{u,x}.\]
This shows that $\alpha$ restricts to a natural isomorphism $\sigma^*M\to\tau^*M$, and hence $[M]$ has a vanishing gradient on $\T$. 
\end{proof}

\begin{Rem}\label{Rem:poset-too-simple}
Theorem \ref{Thm:grad0} shows that the vanishing of the gradient on a sub-poset generated by a tree gives a lot of information about the module in question. It is instructive however to note that if the poset in question consists of a single pair of comparable objects, then any module on such poset, where the point modules are isomorphic, will have a vanishing gradient. This of course does not contradict the conclusion of the theorem. We leave the details to the reader. 
\end{Rem}

Notice that line-connectivity of $\T$ plays an important role in the proof of Theorem \ref{Thm:grad0}. The next lemma shows that if $\P$ itself is line connected, then $\HH_\P$ contains a line connected maximal tree. Thus it is possible to get information about the behaviour of $M\in k\P\mod$ on all objects and all morphisms that form the maximal tree.

\begin{Lem}\label{Lem:line-connected-spanning-tree}
	Let $\P$ be a line connected finite poset. Then its Hasse diagram $\HH_\P$ has a line connected maximal tree.
\end{Lem}
\begin{proof}
	The Hasse diagram $\HH_\P$  is a connected digraph (a line connected digraph is in particular connected), and by definition it is acyclic and  transitively reduced. Let $\widehat{\HH}_\P$ denote its line digraph, as usual. If $\HH_\P$ is not a tree, then there are two objects $x< y \in \P$ and at least two distinct paths in $\HH_\P$ from $x $ to $y$. We proceed by showing that one can disconnect one of the paths from $x$ to $y$ in $\HH_\P$ with the resulting digraph remaining line connected. The claim then follows by induction on the number of multiple paths between pairs of points in $\HH_\P$. 
	
	Let $x<y\in \P$ be two distinct vertices, and assume that there are more than one directed path from $x$ to $y$. Denote the collection of all directed paths from $x$ to $y$ by $l_1,\ldots, l_n$, for $n>1$. We may assume that in all $l_j$, except possibly one of them (since double edges are not allowed in a poset), there is a vertex  $x<a_i<y$, such that either $x<a_i$ or $a_i<y$ is irreducible in $\P$.   We may assume without loss of generality that the $l_j$ have no common vertices except $x$ and $y$. 
 
    We consider four possibilities. 
    
    \noindent{\it (1) $x$ is not minimal and $y$ is maximal:} Then there exists some  relation  $x_0<x$ in $\P$, and removing a single edge $x<a_j$ from all $l_j$ except one of them, disconnects all the $l_j$ except the one that was not modified. Since this leaves exactly one path from $x$ to $y$, the resulting digraph is still line connected. 
    
    \noindent{\it (2) $y$ is not maximal and $x$ is minimal:} Then there exists a relation $y<y_0$ in $\P$, and removing a single edge $a_j<y$ from each $l_j$ except one of them, disconnects all the $l_j$ except the one that was not modified. Again the resulting digraph is line connected.

    \noindent{\it (3) $x$ is not minimal and $y$ is not maximal:} Then remove a single irreducible edge from each $l_j$ except one of them, as in (1) and (2). In that case as well the resulting digraph is line connected.
    
    \noindent{\it (4) $x$ is minimal and $y$ is maximal:}
     In that case  the associated line digraph splits into $n$ connected components, each containing the line digraph of $l_j$ for a unique $1\le j\le n$. This contradicts the assumption that $\P$ is line connected. 
     
     The proof is now complete by inductively performing this procedure for any set of multiple paths between pairs of vertices in $\HH_\P$.
\end{proof}

\chg{Next, with the notation and hypotheses of Theorem \ref{Thm:grad0} we define a \hadgesh{kernel submodule} and an \hadgesh{image submodule} of $M_\T$, where  $M\in k\P\mod$ is a module with a vanishing gradient on a line connected sub-tree $\T\subseteq\HH_\P$.}

\begin{Defi}\label{Def:Ker-Im}
With the notation of and hypotheses of Theorem \ref{Thm:grad0}, define $\Ker M_\T, \Ima M_\T\in k\P_\T\mod$ as follows.
\begin{enumerate}[(1)]
\item 
\[\Ker M_\T(u)\defeq 
\begin{cases}
    \Ker (M(u< v)) &  (u,v)\in\T\\
    0 & u\quad \text{maximal}
\end{cases}.\]
Define $\Ker M_\T(u<v)\defeq 0$ for any non-identity morphism in $\P_\T$.\label{Def:Ker-Im1}
\item 
\[\Ima M_\T(u)\defeq 
\begin{cases}
    \Ima(M(w < u)) & (w,u)\in\T\\
    M(u)/\Ker(M(u<v)) & (u,v)\in \T,\; u\; \textup{minimal}
\end{cases}\]
Define $\Ima M_\T(u<v)$  to be the restriction of $M_\T(u<v)$ to $\Ima M_\T(u)$.\label{Def:Ker-Im2}
\end{enumerate}
\end{Defi}

Since the definitions depend on choices of specific relations, we must show that $\Ker M_\T$ and $\Ima M_\T$ are well defined.

\begin{Lem}\label{Lem:Ker-Im}
\chg{With the notation and hypotheses of Theorem \ref{Thm:grad0}, the modules $\Ker M_\T$ and $\Ima M_\T$ are well defined $k\P_\T$-modules.
Furthermore, the following statements hold}
\begin{enumerate}[(1)]
\item $\Ker M_\T(u) \cong \Ker M_\T(v)$ for all non-maximal $u, v\in \P_T$. \label{Lem:Ker-Im-1}

\item $\Ima M_\T(u) \cong \Ima M_\T(v)$ for all $u, v\in \P_\T$. \label{Lem:Ker-Im-2}
\end{enumerate}
\end{Lem}
\begin{proof}
Let $u<v$ and $u<v'$ be  irreducible morphisms in $\P_\T$. Then by Theorem \ref{Thm:grad0}\ref{Thm:grad0-2}, 
\[\alpha_{v,v'}\circ M(u<v) = M(u<v').\]
Since $\alpha_{v,v'}$ is an isomorphism, $\Ker(M(u<v)) = \Ker(M(u<v'))$. This shows that $\Ker M$ is well defined. 

Let $w<u$ and $w'<u$ be two irreducible morphisms in $\P$. Then, again by \ref{Thm:grad0}\ref{Thm:grad0-2}, 
\[M(w<u) = M(w'<u)\circ\alpha_{w,w'}.\]
Thus $\Ima M_\T $ is well defined on non-minimal objects. If $u$ is minimal, then $\Ima M_\T(u)$ is also well defined since $\Ker(M(u<v))$ does not depend on the choice of $v$. The definition on morphisms is clear.

Let $u, s\in\P$ be non-maximal, and let $u<w$ and $s<t$ be irreducible morphisms in $\P$. Let $x\in\Ker(M(u<w))$, and let $y=\alpha_{u,s}(x)$. Then by Theorem \ref{Thm:grad0}\ref{Thm:grad0-2},
\begin{equation}
    0=M(u<w)(x) = \alpha_{t,w}\circ M(s<t)\circ\alpha_{u,s}(x) = \alpha_{t,w}\circ M(s<t)(y). \label{Eqn:Cor:Ker-Im}
\end{equation}
Consider the following diagram with exact rows. Equation \textcolor{red}{(\ref{Eqn:Cor:Ker-Im})} implies that the left square commutes, and hence that so does the right square. 
\[\xymatrix{
\Ker M_\T(u) = \Ker(M(u < w))\; \ar@{>->}[r]\ar[d] & M(u)\ar@{->>}[r]\ar[d]^{\alpha_{u,s}}_\cong & M(u) / \Ker (M(u < w)) \simeq \im M_\T(w) \ar[d]\\
\Ker M_\T(s) =\Ker(M(s < t)) \; \ar@{>->}[r] & M(s) \ar@{->>}[r] & M(s)/ \Ker (M(s < t)) \simeq \im M_\T(t) 
}\]
for all $u, s\in\P$. Hence $\alpha_{u,s}$ restricted to $\Ker M_\T(u)$ is a monomorphism and the induced map on $\Ima M_\T(w)$ is an epimorphism. The same argument using $\alpha_{s,u} = \alpha_{u,s}^{-1}$ and finite dimensionality of all point modules shows that $\Ker M_\T (u)\cong \Ker M_\T(s)$ for all non-maximal $u, s\in \P$ and proves Part \ref{Lem:Ker-Im-1}. Notice also that at the same time we have shown that $\Ima M_\T(w)\cong \Ima M_\T(t)$ for all $w, t\in \P$ that are not minimal. Finally, notice that if $u$ is minimal, then
\[\Ima M_\T(u)\defeq M(u)/\Ker(M(u<v)) = \Ima(M(u<v)) \defeq \Ima M_\T(v).\]
Thus $\Ima M_\T (u) \cong \Ima M_\T (v)$ for all $u, v\in \P$. This proves Part \ref{Lem:Ker-Im-2}.
\end{proof}

\subsection{\chg{Modules with equal gradients and the rank invariant}}
The next obvious question is what can be said about a virtual module $ X = [M] - [N]$  with a vanishing gradient on a tree, where $M ,N\in k\P\mod$, or equivalently by using  linearity if and only if 
\[\tau^*[M] - \sigma^*[M] = \nabla[M] = \nabla[N] = \tau^*[N] - \sigma^*[N], \]
This is the case if and only if there is a natural isomorphism
\begin{equation}
	\tau^*M \oplus \sigma^*N \cong \tau^*N\oplus \sigma^*M.
	\label{Eq:virtual-zero-grad}
\end{equation}

\chg{It is clear that if $M\in k\P\mod$ is any module and $X\in\Gr(k\P)$ is a virtual module with vanishing gradient, then the gradients of $[M]$ and $[M]+X$ coincide. Understanding the kernel of the gradient is therefore an important question, which we will study in greater generality is a subsequent paper.}

\chg{An important algebraic invariant that appear in a variety of mathematical disciplines, and in particular in generalised persistence module theory, is the rank invariant. Next we consider the relationship between the rank invariant and our  gradient.}

\begin{Defi}\label{Def:rk}
	Let $M\in k\P\mod$ be a  module. Define the \hadgesh{rank invariant} \[\rk[M]\colon \Mor(\P)\to \N\] by $\rk [M](x\le y) \defeq \rk M(x\le y)$ 
	for each relation $x\le y$ in $\P$. 
	For a virtual module $ X = [M]-[N]$ in $\Gr(k\P)$, where $M, N\in k\P\mod$,  define the rank invariant \[\rk X\colon \Mor(\P)\to \Z\] by $\rk X\defeq \rk[M] - \rk[N]$. 
\end{Defi}

\chg{Notice that $\rk[M](x\le x) = \dim_k M(x)$ for each object $x\in \P$, so the rank invariant, as defined above, includes the Hilbert functions $\curs{H}_X\colon V\to \Z$, where $V$ is the object set of $\P$, and for a virtual module $X = [M]-[N]\in\Gr(k\P)$, it is defined by $\curs{H}_X(v) =  \dim_k M(v) - \dim_k N(v)$. }Notice also that the rank invariant can be regarded as a group homomorphism 
\[\rk\colon \Gr(k\P)\to \Z^{\Mor(\P)},\]
where the codomain is the abelian group of all functions from $\Mor(\P)$ to $\Z$, or in other words the free abelian group generated by all morphisms in $\P$. 

\chg{\begin{Rem}
It may be interesting to consider the intersection of $\Ker\nabla$ and the kernel of the rank invariant homomorphism.
\end{Rem}}

Next we prove Theorem \ref{Th:Rank}, which we restate below.

\begin{Thm}\label{Thm:Rank}
	Let $\P$ be a finite poset. Let $ X=[M]-[N]\in \Gr(k\P)$ be an element with $M, N\in k\P\mod$. Then
	\begin{enumerate}[(1)]
		\item Let $(u_0,u_1) < (v_0,v_1)$ be a pair of comparable objects in $\Pone$. Assume that  $\nabla X = 0$. Then $\rk X(u_0<v_0) = \rk X(u_1<v_1)$. \label{Thm:Rank:1}
	\end{enumerate}
	
	Assume  in addition that $\P$ is line connected, let $\T$ be a line connected maximal tree for $\HH_\P$, and let $\P_\T\subseteq \P$ be the sub-poset generated by $\T$. If $ X$ has a vanishing gradient on $\T$, then  $\rk X$ has the following properties:
	
	\begin{enumerate}[(1)]
		\setcounter{enumi}{1}
		\item It is constant on all identity morphisms in $\P$.
		\label{Thm:Rank:2}
		\item For any pair of irreducible relations $u_0<u_1$, and $v_0< v_1$ in $\P_\T$,   one has $\rk X(u_0<u_1) = \rk X(v_0<v_1)$. \label{Thm:Rank:3}
      \end{enumerate}
\end{Thm}
\begin{proof} 
	 Since $\nabla X = 0$, we have a natural isomorphism $\eta\colon\tau^*M\oplus \sigma^*N\xto{\cong} \tau^* N\oplus \sigma^*M$ (See Equation \textcolor{red}{(\ref{Eq:virtual-zero-grad})}). Let $(u_0,u_1) < (v_0,v_1)$ be a pair of comparable objects in $\Pone$. Then one has a commutative square
	\[\xymatrix{
		M(u_1)\oplus N(u_0) \ar[rrrr]^{M(u_1<v_1)\oplus N(u_0<v_0)}\ar[d]^\cong_\eta &&&& M(v_1)\oplus N(v_0)\ar[d]^\cong_\eta\\
		N(u_1)\oplus M(u_0) \ar[rrrr]^{N(u_1<v_1)\oplus M(u_0<v_0)} &&&& N(v_1)\oplus M(v_0)
	}\]
	It follows that $\rk M(u_0<v_0) + \rk N(u_1<v_1) = \rk M(u_1<v_1) + \rk N(u_0<v_0)$, and hence that $\rk X(u_0<v_0) = \rk X(u_1<v_1)$. This proves Part \ref{Thm:Rank:1}.
	
	Throughout the rest of the proof assume that $\P$ is line connected. Fix a line connected maximal tree $\T$ for $\HH_\P$. Since $ X$ has a vanishing gradient on $\T$, we have a natural isomorphism $\tau^*M_\T\oplus\sigma^*N_\T\cong\tau^*N_\T\oplus\sigma^*M_\T$. 
	Thus, for any object $(x,y)\in\Pone_\T$ one has  $M(x)\oplus N(y)\cong N(x)\oplus M(y)$. Hence 
	\[\rk X(x\le x) = \dim_kM(x)-\dim_kN(x) = \dim_kM(y)-\dim_kN(y) = \rk X(y\le y).\]
	Since every $x\in\P_\T$ is either the source coordinate or the target coordinate  in an object of $\Pone_\T$, and since $\Pone_\T$ is connected, Part \ref{Thm:Rank:2} follows.

	Let $(u_0,u_1)$ be a minimal object in $\Pone_\T$. By Part \ref{Thm:Rank:1} and induction, $\rk X(u_r<u_{r+1}) = \rk X(u_0<u_1)$ for any object $(u_r,u_{r+1}) \in\Pone_\T$ that is the target of a directed path from $(u_0, u_1)$ in $\Pone_\T$. 
	Let $(u_0,u_1)$ and $(v_0,v_1)$ be minimal objects in $\Pone_\T$. Since $\Pone_\T$ is connected, there are minimal objects
	\[(u_0,u_1) = (x_1,y_1), (x_2,y_2),\ldots, (x_k,y_k) = (v_0,v_1)\]
	such that for each $1\le i \le k-1$ the intersection of sub-posets $\Pone_{\T_{\ge (x_i,y_i)}}\cap \Pone_{\T_{\ge (x_{i+1},y_{i+1})}}$ is nonempty. Let $(a_i,b_i)$ be any object in the intersection.  By the argument above, 
	\[\rk X(x_i<y_i) = \rk X(a_i<b_i) = \rk X(x_{i+1}<y_{i+1})\]
	and since this holds for all $i$, we have $\rk X(u_0<u_1) = \rk X(v_0,v_1)$. It follows that $\rk X$ is constant on all irreducible relations in $\P_\T$, as claimed in Part \ref{Thm:Rank:3}.
\end{proof}

The following example shows that the statements of Theorem \ref{Thm:Rank} are best possible in the sense that two modules with equal gradients may have different rank invariants when evaluated on morphisms of length greater than 1.

\begin{Ex}\label{Ex:Rank-best}
Let $\P$ denote the poset $[2]$ with objects $0,1,2$ and the ordinary order relation. Define $M, N\in[2]\mod$ as follows. 
\[M\colon\quad k^2\xto{(1,0)}k^2\xto{(0,1)}k^2,\qquad N\colon\quad k^2\xto{(1,0)}k^2\xto{(1,0)}k^2.\]
It is immediate that $\tau^*M \cong \sigma^*M$ and $\tau^*N \cong \sigma^*N$, so $\nabla[M]=\nabla[N]=0$. But $\rk M(0<2)=0$ while $\rk[N](0<2)=1$.
\end{Ex}

Above we considered the kernel and image modules $\Ker M_\T$ and $\Ima M_\T$ associated to a module $M\in k\P\mod$ in the case where $M$ has a vanishing gradient on some tree $\T\subseteq\HH_\P$ (See Definition \ref{Def:Ker-Im} and Lemma \ref{Lem:Ker-Im}). Next, we consider a kernel module and a cokernel module in $k\Pone\mod$, associated to any module $M\in k\P\mod$. \chg{Using the standard terminology in persistence module theory, these modules can be considered as containing information about ``births" and ``deaths" in $M$. Definition \ref{Def:Ker-Coker}, Lemma \ref{Lem:Nullity} and remark \ref{Rem:Life} below make this precise}

\chg{\begin{Defi}\label{Def:Ker-Coker}
	Let $M\in k\P\mod$ be any module. 
	Let $K_M, C_M\in k\Pone\mod$ denote the modules defined on objects by 
	\[K_M(x,y) \defeq \Ker(M(x,y)\colon M(x)\to M(y))\quad\text{and}\quad
	C_M(x,y) \defeq \coKer(M(x,y)\colon M(x)\to M(y))\]
	with the natural induced maps on morphisms.
\end{Defi}}
It follows by a straightforward diagram chase that both $K_M$ and $C_M$ are virtually trivial.

\begin{Lem}\label{Lem:Nullity}
	Let $\P$ be a finite  line connected poset, let $\T $be a line connected maximal tree for $\HH_\P$, and let $\P_\T\subseteq \P$ be the sub-poset generated by $\T$.  Let $ X = [M] - [N]\in \Gr(k\P)$ be an element of vanishing gradient on $\T$, with $M, N\in k\P\mod$. Then there is an integer $D$, such that $[K_M] = [K_N] + [\underline{D}]$ and $[C_M ] = [C_N] + [\underline{D}]$ in $\Gr(k\Pone_\T)$, where $[\underline{D}]$ denotes the virtually trivial module that associates  $k^D$ with every object of $\Pone_\T$. 
\end{Lem}
\begin{proof}
	By construction $K_M$ and $K_N$, as well as $C_M$ and $C_N$  are virtually trivial, and hence their isomorphism type is determined by their values on objects by Corollary \ref{Cor:virtually-trivial}. Thus, it suffices to prove that the appropriate ranks coincide on all objects of $\Pone_T$.
	
	By Theorem \ref{Thm:Rank}\ref{Thm:Rank:1}, $\rk X =\rk[M]-\rk[N]$ is constant on all identity morphisms  and all irreducible morphisms  $x< y$ in $\P_\T$. Thus we may write $\dim_k M(x) - \dim_k N(x) = K$ for all $x\in\P$, and $\rk[M](x<y) - \rk[N](x<y) = T$ for all irreducible relations $x<y$ in $\T$, where $K$ and $T$ are some fixed non-negative integers. Set $D = K-T$. Thus for an object $(x,y)\in\Pone_\T$,
	\begin{align*}\dim_k K_M(x,y) &= \rk[M](x\le x)-\rk[M](x<y) \\&= \rk[N](x\le x)-\rk[N](x<y)+ K - T\\&=\dim_k K_N(x,y)+ D.
	\end{align*}
 Since both $K_M$ and $K_N$ are virtually trivial, it follows that $[K_M] = [K_N]+[\underline{D}]$.
	 This proves the statement  for $K_M$ and $K_N$.

	Similarly,
	\begin{align*}
		\dim_k C_M(x,y) &= \rk[M](y\le y) - \rk[M](x< y)\\
		&=  \rk[N](y\le y) - \rk[N](x< y) + K - T \\& = \dim_k C_N(x,y)+ D.
	\end{align*}
	By the same argument as above $[C_M] = [C_N] + [\underline{D}]$, as claimed.
\end{proof}

\begin{Rem}\label{Rem:Life}
	The functor $K_M\in k\Pone\mod$ evaluated at an object $(x,y)$ returns the subspace of elements that are ``present" at $x$, but do not ``survive" to $y$. Thus for a fixed $x\in \P$, the intersection $\bigcap_{(x,y)\in\Pone}K_M(x,y)$ is the subspace of $M(x)$ of all elements that ``die'' at $x$.  
	
	Similarly, the functor $C_M\in k\Pone\mod$, evaluated at an object $(x,y)$ returns the quotient of $M(y)$ by the image of $M(x<y)$. For a fixed object $y$, consider the system of all homomorphisms $M(y)\to C_M(x,y)$ for all $(x,y)\in\Pone$. The coequaliser of this system can be thought of as the space representing all elements in $M(y)$ that are ``born" at $y$. The coequaliser is easily seen to be the quotient of $M(y)$ by the image of the composite
	\[\bigoplus_{(x,y)}M(x)\xto{\oplus M(x<y)}\bigoplus_{(x,y)}M(y)\xto{\sum} M(y),\]
	where $\Sigma$ is the map given by summing coordinates.
	
	Thus if $\P$ is line connected and $M, N\in k\P\mod$ are modules with equal gradients, then for any line connected maximal tree $\T$,  spaces of ``births" and ``deaths" of $M$ and $N$ restricted to $\P_\T$ object-wise coincide.
\end{Rem}

We end this section with an example of  modules with  non-isomorphic gradients and the same rank invariant.

\begin{Ex}\label{Ex:Grad-Incomplete}
	Let $\P$ be the  poset with objects $\emptyset, a, b, c, d, m, \infty$, and relations \[\emptyset<a, b, c, d < m<\infty.\]  
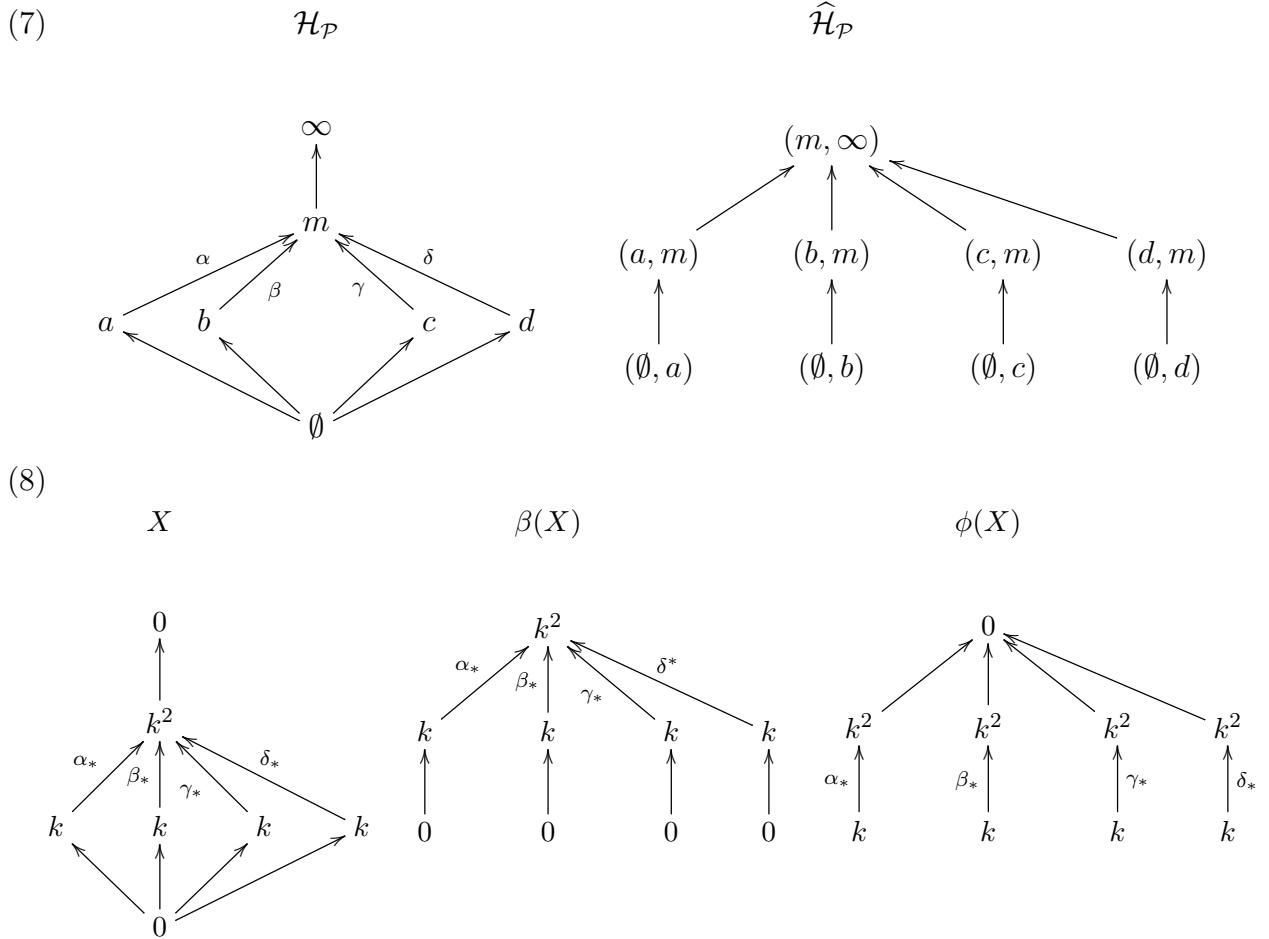
	\begin{equation}
		\xymatrix{
			&& \HH_\P\\
			&& \infty \\
			&& m\ar[u] & \\
			a\ar[urr]^\alpha & b\ar[ur]_\beta && c\ar[ul]^\gamma & d\ar[ull]_\delta\\
			&& \emptyset\ar[ul]\ar[ull]\ar[ur]\ar[urr] & \\
		}
		\qquad
		\xymatrix{
			&\widehat{\HH}_\P\\
			&(m,\infty) &\\
			(a,m)\ar[ur]& (b,m)\ar[u] & (c,m)\ar[ul] & (d,m)\ar[ull]\\
			(\emptyset, a)\ar[u] & (\emptyset, b)\ar[u] & (\emptyset, c)\ar[u] &(\emptyset,d)\ar[u]\\
		}
		\label{Eqn:P-lineP}
	\end{equation}
	\begin{equation}
			{\small \xymatrix{
					& X\\
					& 0 \\
					& k^2\ar[u] & \\
					k\ar[ur]^{\alpha_*} & k\ar[u]^{\beta_*} & k\ar[ul]^{\gamma_*}& k\ar[ull]_{\delta_*}\\
					& 0\ar[ul]\ar[u]\ar[ur]\ar[urr] & \\
				}
				\quad
				\xymatrix{
					& \sigma(X)\\
					&k^2 &\\
					k\ar[ur]^{\alpha_*}& k\ar[u]^{\beta*} & k\ar[ul]^{\gamma_*} & k\ar[ull]_{\delta^*}\\
					0\ar[u] & 0\ar[u] & 0\ar[u] & 0\ar[u]\\
				}
				\quad
				\xymatrix{
					&\tau(X)\\
					&0 &\\
					k^2\ar[ur]& k^2\ar[u] & k^2\ar[ul] & k^2\ar[ull]\\
					k\ar[u]^{\alpha_*} & k\ar[u]^{\beta*} & k\ar[u]_{\gamma_*} & k\ar[u]_{\delta_*}\\
			}}
		\label{Eqn:Module-Der}
	\end{equation}
	Diagram \textcolor{red}{(\ref{Eqn:P-lineP})} shows the poset $\P$ and its associated line digraph.
	Consider  $k\P$-modules $X$ that take the  value $k$ on $a, b, c, d$, the value $k^2$ on $m$, and $0$ on $\emptyset$ and $\infty$. Furthermore, assume that for each of the morphisms $\alpha, \sigma, \gamma, \delta$, the induced map under $X$ is an injection, and such that the images of each pair of homomorphisms form a basis for $X(m)=k^2$. Clearly all such modules have exactly the same rank invariant.
	
	Let $M, N\in k\P\mod$ be modules satisfying these requirements.  Let $X, Y\in k\Pone\mod$ denote the modules $\tau^*(M)\oplus\sigma^*(N)$ and $\tau^*(N)\oplus \sigma^*(M)$ respectively. Then $\nabla[M]=\nabla[N]$ if and only if $X$ and $Y$ are isomorphic. Let $X_0$ and $Y_0$ denote  the restrictions of $X$ and $Y$ to  the full sub-poset of $\Pone$ consisting of the objects $(a,m)$, $(b,m)$, $(c,m)$, $(d,m)$ and $(m, \infty)$. Denote by  $\alpha_*, \beta_*,\ldots $  the homomorphisms $N(\alpha), N(\sigma),\ldots$ and by $\alpha'_*, \beta'_*,\ldots $  the homomorphisms $M(\alpha), M(\sigma),\ldots$. Assume that there is an isomorphism $\Theta\colon X\to Y$ and let $\Theta_0\colon X_0\to Y_0$ be the restriction of $\theta$:
 	\[
	\xymatrix{
		&k^2 &\\
		k^2\oplus k\ar[ur]^{0 \oplus\alpha_*}& k^2\oplus k\ar[u]^{0 \oplus\beta_*} & k^2\oplus k\ar[ul]^{0 \oplus\gamma_*}& k^2\oplus k\ar[ull]_{0\oplus\delta_*}\\
	}
	\Longrightarrow
	\xymatrix{
		&k^2 &\\
		k\oplus k^2\ar[ur]^{\alpha'_*\oplus 0}& k\oplus k^2\ar[u]^{\beta'_*\oplus 0} & k\oplus k^2\ar[ul]^{\gamma'_*\oplus 0} & k^2\oplus k\ar[ull]_{\delta'_*\oplus 0}\\
	}
	\]
	
	Without loss of generality, using our assumption that the images of every pair of homomorphisms generate $k^2$, we may assume that $\alpha$ and $\alpha'$ take 1 to the vector $(1,0)\in k^2$ and that $\sigma$ and $\sigma'$ take 1 to $(0,1)\in k^2$. Let $\gamma$ and $\gamma'$ take 1 to the vectors $(x,y)$ and $(z,w)$ respectively, and let $\delta$ and $\delta'$ take 1 to $(s,t)$ and $(u,v)$ respectively.
	The upwards homomorphisms are given by the matrices (from left to right, with respect to the standard bases):
	\[
	\left(\begin{smallmatrix}
		0 & 0 & 1\\0 & 0 & 0\end{smallmatrix}\right),\quad
	\left(\begin{smallmatrix}
		0 & 0 & 0\\0 & 0 & 1\end{smallmatrix}\right),\quad
	\left(\begin{smallmatrix}
		0 & 0 & x\\0 & 0 & y\end{smallmatrix}\right),\quad
	\left(\begin{smallmatrix}
		0 & 0 & s\\0 & 0 & t\end{smallmatrix}\right)\qquad\text{and}\quad
	\left(\begin{smallmatrix}
		1 & 0 & 0\\0 & 0 & 0\end{smallmatrix}\right),\quad
	\left(\begin{smallmatrix}
		0 & 0 & 0\\1 & 0 & 0\end{smallmatrix}\right),\quad
	\left(\begin{smallmatrix}
		z & 0 & 0\\w & 0 & 0\end{smallmatrix}\right),\quad
	\left(\begin{smallmatrix}
		u & 0 & 0\\v & 0 & 0\end{smallmatrix}\right).
	\]
	Then $\Theta_0$ can be represented on the object $(m,\infty)$  by a $2\times 2$ matrix $A = (a_{i,j})$, and on the objects $(a,m), \ldots, (d,m)$ by four $3\times 3$ matrices $B^k = (b^k_{i,j})$, for $1\le k\le 4$. Computing the respective products, it is easy to observe that $A$ must be diagonal with non-zero entries $a_{1,1}$ and $a_{2,2}$. Similarly $a_{1,1} =  b^1_{1,3}$, $a_{2,2} = b^2_{1,3}$, and $b^k_{1,j}=0$ for $1\le k\le 4$ and $j=1,2$. Furthermore, we have 
	\begin{equation}
		b^3_{1,3} = \frac{x}{z}a_{1,1} = \frac{y}{w}a_{2,2},\quad\text{and}\quad
		b_{1,3}^4 = \frac{s}{u}a_{1,1} = \frac{t}{v}a_{2,2}.
		\label{Eqn:Ex}
	\end{equation}
	Thus, as long as these relations are satisfied, $A$ and $B^k$ can be constructed with any nonzero choices of those values, while making sure in an arbitrary manner that $B^k$ are nonsingular. 
	However, the relations in \textcolor{red}{(\ref{Eqn:Ex})} allow solving for $a_{2,2}$ in terms of the other variables in two ways, and by comparing them we obtain the relation
	\[wxtu = zyvs\]
	that must be satisfied for $\Theta_0$ to be well defined. Since this relation is not satisfied in general, this shows that there exist modules $M, N\in k\P\mod$  with the same rank invariant, such that $\nabla[M]\neq\nabla[N]$ (and hence also $[M]\neq [N]$).
\end{Ex}


\section{The Hom pairing and the Euler pairing for posets}
\label{Sec:Euler}
Let $\P$ be a finite poset.  \chg{Recall the Hom pairing and the Euler pairing  from Definitions \ref{Def:Hom-pairing} and   \ref{Def:Euler-pairing} respectively for modules $M, N\in k\P\mod$,
\[\langle [M], [N]\rangle_\P \defeq \dim_k(\Hom_{k\P}(M,N))\quad\text{and}
\quad\chi_\P([M],[N])\defeq\chi(\Ext^*_{k\P}(M, N)),\]
with the obvious additive extension to Grothendieck groups. }
Notice that $\Ext_{k\P}^i(M,N)=0$ for all $i>0$ if $M$ is projective or if $N$ is injective. In either case the two pairings coincide \chg{by definition.}

\chg{Recall the following three families of modules in $k\P\mod$ for any poset $\P$. }
\begin{Defi}\label{Def:Proj-Inj-Simp}
For a finite poset $\P$, and an object $v\in \P$,  define modules $F_v, G_v, S_v\in k\P\mod$ as follows.  For an object $u\in\P$,  
\begin{enumerate}[(1)]
\item 
$F_v(u) \defeq k\Mor_\P(v,u) = \begin{cases}
	k & v\le u\\
	0 & \text{otherwise}
\end{cases}$ \label{Def:Proj-Inj-Simp-Proj}
\item $G_v(u)\defeq k\Mor_\P(u,v) = \begin{cases}
	k & u\le v\\
	0& \text{otherwise}
\end{cases}$ \label{Def:Proj-Inj-Simp-Inj}
\item $S_v(u) \defeq \begin{cases}
    k & u=v\\
    0& \text{otherwise}
\end{cases}$ \label{Def:Proj-Inj-Simp-Simp}
\end{enumerate}
If $v\le u\le u'$, then $F_v(u)\to F_v(u')$ is the identity, and otherwise $F_v(u\le u')$ is the zero homomorphism. Similarly, if $u\le u'\le v$ then $G_v(u)\to G_v(u')$ is the identity and otherwise it is $0$. For $S_v$ all non-identity morphisms are $0$. 
\end{Defi}
By \cite[Proposition 2.2.3]{Derksen-Weyman} the modules $F_v$ are precisely the indecomposable projective modules and the modules $G_v$ are the indecomposable injective modules in $k\P\mod$. By \cite[Proposition 3.1.6]{Derksen-Weyman} the modules $S_v$ are the simple modules in $k\P\mod$. Notice that $F_v$ is locally constant on $\P_{\ge v}$, while $G_v$ is locally constant on $\P_{\le v}$. Also, any virtually trivial module $M$ is a direct sum of simple modules,
\[M =  \bigoplus_{v\in \P}\dim_kM(v)\cdot S_v.\]

\chg{We proceed with some basic  computations of the Hom and Euler pairings involving these three special families of modules. The following lemma shows  that when $\P$ is generated by a rooted tree, then the class of indecomposable injective module $G_v$ is the gradient of classes of certain modules in $k\P\mod$. }

\begin{Lem}\label{Lem:GF-integrable}
\chg{Let $\P$ be a finite poset such that $\HH_\P$ is a rooted tree. Then for every object $(x,y)\in\Pone$, there exists a module $M = M_{(x,y)}$, such that $\nabla[M] = [G_{(x,y)}]$.}
\end{Lem}
\begin{proof}
The definition of the module $M$  will be carried out in three steps. In the first step we define $M$ on the down-set $\P_{\le y}$. In the second step we define it on each branch emanating from an object $z\in\P_{\le y}$. Finally in the third step we show that $\nabla[M] = [G_{(x,y)}]$. 

\noindent{\bf Step 1.} Since $\HH_\P$ is a rooted tree, $\P_{\le y}$ is the unique path from $x_0$ to $y$. Set $M(y) = k^n$, and for any $v\in\P_{\le y}$ with $d(v,y)=r$ in $\HH_\P$, let $M(v) = k^{n-r}$. This is well defined since $\P_{\le y}$ is a linear order.  If $u<v$ is an irreducible relation in $\P_{\le y}$, then $d(u,y) = d(v,y)+1$. Define $M(u<v)$ to be the inclusion 
\[M(u<v) \colon M(u) = k^{n-r-1}\to k^{n-r} = M(v)\]
into the last $n-r-1$ coordinates. In matrix notation, with respect the the standard basis, one may write $M(u<v)$ as the $(n-r)\times(n-r-1)$ matrix
\[A_{u<v} = \left(\begin{matrix}
                            0\cdots 0  \\
                            \hline
                            I_{n-r-1}  
                \end{matrix}
                    \right)\]

\noindent{\bf Step 2.} Let $x_0\neq z\in\P_{\le y}$ be any object, and let $\P_z\subseteq \P_{\ge z}$ denote the sub-poset generated by all objects in $\P_{\ge z}$ except those which form the branch emanating from $z$ that  contains $y$, if $z\neq y$, and if $z=y$  set $\P_y=\P_{\geq y}$. Then  the Hasse diagram of  $\P_{z}$ is a rooted tree with root $z$ that intersects with $\P_{\le y}$ only on the vertex $z$. For each object $w\in\P_{z}$, define $M(w) = M(z)$. 
Let $t<z$ be the (unique) in-neighbour of $z$ in $\HH_\P$, and let $d(z,y)=r-1$. For  any irreducible relation $u<v$ in $\P_{z}$,  define 
\[M(u<v)\defeq M(t<z)\oplus 0.\]
Again, in matrix notation $M(u<v)$ in this case is represented by the $(n-r)\times (n-r)$ matrix 
\[B_{u<v} = \left(\begin{tabular}{c|c}
                        $A_{t<z}$ &  $0$\\
                    \end{tabular}\right) =
             \left(\begin{tabular}{c|c}
             $0 \cdots 0$ & $0$\\
             \hline
             $I_{n-r-1}$ & $0$
             \end{tabular}
                     \right)\]

If $z=x_0$, then $M(x_0)=0$, by definition, and for any object $w\in \P_{x_0}$ define $M(w)=0$.

\noindent{\bf Step 3.}  We are now ready to analyse $\sigma^*M$ and $\tau^*M$. By definition we have 
\[\sigma^*M(u,v) = M(u) = \begin{cases}
					k^{n-r} & d(u,y) = r\\
					M(z)    & u\notin \P_{\le y},\ z\in\P_{\le y}, \ d(z,u) \text{ minimal}
					\end{cases}
\]
Similarly,
\[
\tau^*M(u,v) = M(v) = \begin{cases}
					k^{n-r} &   d(v,y) = r\\
					M(z) & v\notin \P_{\le y},\ z\in\P_{\le y}, \ d(z,u) \text{ minimal}
					\end{cases}
\]
		
Let $(u,v)\in\Pone_{\le(x,y)}$. Then, either $(u,v)= (x,y)$ or $v\le x$. In the first case $\sigma^*M(x,y) = M(x) = k^{n-1}$, while $\tau^* M(x,y) = M(y) =k^n$. In the second case, $\sigma^* M(u,v) = M(u) = k^{n-d(u,y)}$, while $\tau^* M(u,v) = M(v) = k^{n-d(v,y)} = k^{n-d(u,y)+1}$. Thus by definition, for any $(u,v)\in \Pone_{\le (x,y)}$ one has $\sigma^* M(u,v) \oplus G_{(x,y)} \cong \tau^* M(u,v)$. 
Let $(u,v)<(v,w)$ be an irreducible morphism in $\Pone_{\le (x,y)}$, and let $d(u,y) = r+1$ for some $r\geq 1 $. Then \[\sigma^*M((u,v)<(v,w)) = M(u<v)\] is represented by the $(n-r)\times (n-r-1)$ matrix 
$\left(\begin{matrix}
                            0\cdots 0  \\
                            \hline
                            I_{n-r-1}  
                \end{matrix}
                    \right)$,  and so 
\[(\sigma^*M\oplus G_{(x,y)})((u,v)<(v,w)) = M(u<v)\oplus 1_r\]
is represented by the $(n-r+1)\times (n-r)$ matrix 
\[\left(\begin{tabular}{c|c}
             $0 \cdots 0$& $0$\\
             \hline
             $I_{n-r-1}$ & $0$\\
             \hline
             $0\cdots 0$ & 1
             \end{tabular}
                     \right)=
\left(\begin{matrix}
                            0\cdots 0  \\
                            \hline
                            I_{n-r}  
                \end{matrix}
                    \right),\]
which also represents $\tau^*M((u,v)<(v,w)) = M(v<w)$.
Thus  we see that 
\[(\sigma^*M\oplus G_{(x,y)})|_{\Pone_{\le(x,y)}} \cong \tau^*M|_{\Pone_{\le(x,y)}}\] as modules.

Let $x_0\neq z\in\P_{\le y}$, let $t<z$ be the unique in-neighbour of $z$, and let $d(t,y)=r+1$, where $r\geq 0$. Let $\P_z^+\subset\P$ denote the sub-poset generated by $\P_z$ and $t$. Thus, the Hasse diagram $\HH^+_z$ of $\P^+_z$ consists of the edge $t<z$ union with the Hasse diagram $\HH_z$ of $\P_z$.  Hence $\HH_{z}^+$ is a  rooted tree with root $t$.

By definition $M|_{\P_{z}^+}$ obtains the value $M(t)$ on $t$ and for each $t\neq w\in \P_{z}^+$ it obtains the value $M(z) \cong M(t)\oplus k$. Furthermore, the homomorphism $M(t<z)$ is represented by the $(n-r)\times (n-r-1)$ matrix 
$A_{t<z}=\left(\begin{matrix}
                            0\cdots 0  \\
                            \hline
                            I_{n-r-1}  
                \end{matrix}
                    \right)$,
and for any other irreducible relation $u<v$ in $\P_z\subset \P_z^+$ the representing matrix for $M(u<v)$ is the $(n-r)\times (n-r)$ matrix
$B_{u<v}=\left(\begin{tabular}{c|c}
             $0 \cdots 0$ & $0$\\
             \hline
             $I_{n-r-1}$ & $0$
             \end{tabular}
                     \right)$.

The restriction of $G_{(x,y)}$ to $\widehat{\P_z^+}$ obtains the value $k$ on the vertex $(t,z)$, since $(t,z)\le (x,y)$, and $0$ on all other vertices. The restriction of $\sigma^*M$ to   $\widehat{\P_z^+}$ take the value $M(t)$ on $(t,z)$ and $M(z)$ everywhere else, while the restriction of $\tau^*M$ to $\widehat{\P_z^+}$ obtains the value $M(z)\cong M(t)\oplus k$ on all objects there. Hence 
\[(\sigma^*M\oplus G_{(x,y)})|_{\widehat{\P_{z}^+}} \cong \tau^*M|_{\widehat{\P_{z}^+}}\]
on objects. Next, examine these modules on morphisms. Let $u$ be an out-neighbour of $z$, such that $u\nleq y$. Then \[(\sigma^*M\oplus G_{(x,y)})(t,z) = M(t)\oplus k,\quad (\sigma^*M\oplus G_{(x,y)})(z,u) = M(z)\oplus 0,\] and $(\sigma^*M\oplus G_{(x,y)})((t,z)<(z,u))$ is represented by the matrix $\left(\begin{tabular}{c|c}
                        $A_{t<z}$ &  $0$\\
                    \end{tabular}\right) = B_{z<u}$, which by definition also represents the homomorphism $\tau^*M((t,z)<(z,u))$.
On any other irreducible relation in $\widehat{\P_z^+}$ the modules $(\sigma^*M\oplus G_{(x,y)})$ and $\tau^*M$ also coincide. Hence 
\[(\sigma^*M\oplus G_{(x,y)})|_{\widehat{\P_{z}^+}} \cong \tau^*M|_{\widehat{\P_{z}^+}}\] as modules.

Next notice that, for any $x_0\neq z\in\P_{\le y}$, the sub-posets $\widehat{\P_z^+}$ and $\Pone_{\le (x,y)}$ have only the object $(t,z)$ in common and no morphisms, and on that object $\sigma^*M\oplus G_{(x,y)}$ and $\tau^*M$ take the same value. Hence, the restrictions of $\sigma^*M\oplus G_{(x,y)}$ and $\tau^*M$ to the connected component that contains $\Pone_{\le(x,y)}$ are isomorphic. 

If $(u,v)\in\Pone$ is an object such that $u\nleq y$ and the nearest object to $u$ in $\P_{\le y}$ is $x_0$, then in $\Pone$ this object is on a connected component that is disjoint from the component of $(x,y)$. By definition $M(u) = M(v) = 0$, and so   $\sigma^*M\oplus G_{(x,y)}$ and $\tau^*(M)$ coincide on $(u,v)$.

This shows that $\sigma^* M(u,v) \oplus G_{(x,y)} \cong \tau^* M(u,v)$ on the full line poset $\Pone$. Thus, in $\Gr(k\Pone)$,
\[\nabla[M] \defeq [\tau^* M] - [\sigma^*M] = [G_{(x,y)}].\]
This completes the proof.
\end{proof}

\chg{We  point out that the assumption in Lemma \ref{Lem:GF-integrable} that $\HH_\P$ is a rooted tree can be relaxed in many cases. The only place where this assumption is used is in the definition of $M$ outside $\P_{\le y}$, which requires that any object there has a well defined distance from $\P_{\le y}$. Also, by analogy, one can show that if the Hasse diagram of $\P^{\op}$ is a rooted tree, then for each object $(x,y)\in\Pone$, there is a module $N = N_{(x,y)}$ such that $\nabla[N] = F_{(x,y)}$.}

\begin{Lem}\label{Lem:Pairing-with-proj-inj}
	Let $\P$ be a finite poset and let  $ X = [M]-[N]$, where $M, N\in k\P\mod$.  Let $Q\cong\bigoplus_{v\in \P}\epsilon_v F_v$ be a finitely generated projective module and let $I\cong\bigoplus_{v\in \P}\delta_v G_v$ be a finitely generated injective module. Then
\begin{enumerate}[(1)]  
\item 
$\langle [Q],  X\rangle_{\P}   = \chi_\P([Q],  X) = \sum_{v\in\P}\epsilon_v(\dim_k M(v)-\dim_k N(v)).$
\label{Lem:Pairing-with-proj-inj-1}
\item
$\langle  X, [I]\rangle_{\P}   = \chi_\P( X, [I]) = \sum_{v\in\P}\delta_v(\dim_k M(v)-\dim_k N(v)).$
\label{Lem:Pairing-with-proj-inj-2}
\end{enumerate}
\end{Lem}
\begin{proof}
	The equality of the Hom and Euler pairings when the left variable is projective or the right variable injective explained in Remark \ref{Rem:pairing-with-Proj-Inj}.  Hence for any object $v\in\P$,  
\[\langle [F_v],  X\rangle_{\P}   = \dim_k \Hom_{k\P}(F_v, M)  - \dim_k \Hom_{k\P}(F_v, N) = \dim_k M(v) - \dim_k N(v).\]
	Thus
\[\langle [Q],  X\rangle_{\P}   = \dim_k \Hom_{k\P}(Q, M)  - \dim_k \Hom_{k\P}(Q, N)  = \sum_{v\in\P}\delta_v(\dim_k M(v)-\dim_k N(v)),
\]
as claimed in \ref{Lem:Pairing-with-proj-inj-1}.

 Similarly,    $\Hom_{k\P}(M,G_v) \cong M(v)^*$, and $\Hom_{k\P}(N,G_v) \cong N(v)^*$, so
\[\langle  X, [G_v]\rangle_{\P}  = \chi_\P( X, [G_v]) = \dim_k M(v) - \dim_k N(v).\]
Part  \ref{Lem:Pairing-with-proj-inj-2} follows similarly.
\end{proof}

The following corollary is immediate. 
\begin{Cor}\label{Cor:proj-inj-ortho}
Let $\P$ be a finite poset, and let $u,v\in \P$ be any two objects. Then
\begin{enumerate}[(1)]
\item $\chi_\P( [F_v], [F_u])=\langle [F_v], [F_u]\rangle_\P = 
							\begin{cases}
									1 & v\ge u\\
									0 & \text{otherwise}
							\end{cases}$
							
\item $\chi_\P( [G_v], [G_u])=\langle [G_v], [G_u]\rangle_\P = 
							\begin{cases}
									1 & v\le u\\
									0 & \text{otherwise}
							\end{cases}$
\end{enumerate}
\end{Cor}

\chg{We are now ready to prove Theorem \ref{Th:pairing}, which  allows explicit computation of the Euler pairing.}

\begin{Thm}
\label{Thm:pairing}
Let $\P$ be a finite poset, and let $M, N\in k\P\mod$. Let 
\[0\to P_n\to\cdots\to P_0\to M\to 0, \quad\text{and}\quad 0\to N\to I_0\to\cdots I_n\to 0\]
be a projective resolution for $M$ and an injective resolution for $N$.
Write $P_i \cong \bigoplus \epsilon^i_vF_v$ and $I_j = \bigoplus \delta_u^j G_u$, with $\epsilon^i_v, \delta^j_u\in \N$ and $v, u\in\P$. Then
\begin{equation*}\chi_\P([M],[N]) = \sum_{v\in\P}\sum_{i=0}^n(-1)^i\epsilon_v^i\dim_k N(v) = \sum_{u\in\P}\sum_{j=0}^n(-1)^j\delta_u^j\dim_k M(u).
\end{equation*}
\end{Thm}
\begin{proof}
Apply $\Hom_{k\P}(-, N)$ to the given projective resolution of $M$. The $i$-th term of the resulting cochain complex is $\Hom_{k\P}(P_i,N)$ and its $i$-th cohomology is $\Ext_{k\P}^i(M,N)$. Hence,
\begin{align*}
\chi_\P([M],[N]) & \defeq \sum_{i=0}^n(-1)^i\dim_k\Ext^i_{k\P}(M, N) \\ & = \sum_{i=0}^n(-1)^i\dim_k\Hom_{k\P}(P_i,N)
& = \sum_{i=0}^n(-1)^i\langle [P_i],[N]\rangle_\P
& = \sum_{i=0}^n(-1)^i\chi_\P([P_i],[N]),
\end{align*}
where the second equality  follows by standard homological algebra,
the third equality follows by definition of the $\Hom$ pairing, and the last from  Remark \ref{Rem:pairing-with-Proj-Inj}. Similarly, one shows  
\[
\chi_\P([M], [N]) = \sum_{j=0}^n (-1)^j\langle  [M], [I_j]\rangle_\P=\sum_{j=0}^n (-1)^j\chi_\P([M],[I_j]).
\]

By Lemma \ref{Lem:Pairing-with-proj-inj}\ref{Lem:Pairing-with-proj-inj-1},
\[\chi_\P([M], [N]) = \sum_{i=0}^n(-1)^i\langle [P_i], [N]\rangle_\P = \sum_{v\in\P}\sum_{i=0}^n(-1)^i\epsilon_v^i\dim_k N(v).\]
Similarly, using Lemma \ref{Lem:Pairing-with-proj-inj}\ref{Lem:Pairing-with-proj-inj-2} one has
\[
\chi_\P([M], [N]) = \sum_{j=0}^n (-1)^j\chi_\P([M],[I_j]) = \sum_{u\in\P}\sum_{j=0}^n(-1)^j\delta_u^j\dim_kM(u).
\]
as claimed.
\end{proof}

A nice interpretation of the Euler pairing occurs for the constant module on $\P$ with value $k$.
\begin{Ex}\label{Ex:inner-cohomology}
	Let $\underline{k}$ denote the constant $k\P$-module with value $k$ at each object and the identity map for each morphism. Then for each $M\in k\P\mod$, 
	\[\Ext^*_{k\P}(\underline{k}, M) =  H^*(\P, M).\]
	Thus $\chi_\P([\underline{k}], [M]) = \chi(H^*(\P, M))$. In particular $\chi_\P([\underline{k}], [\underline{k}]) = \chi(|\P|)$, where $|\P|$ denotes the nerve of $\P$.
\end{Ex}

Combining Lemma \ref{Lem:Pairing-with-proj-inj} and Example \ref{Ex:inner-cohomology}, one observes that if  $\P$ has an initial object $\emptyset$, then $\underline{k}  = F_\emptyset$ is projective. In that case the positive degree cohomology of any $k\P$-module vanishes, and 
\[\chi_\P([\underline{k}], [M]) = \dim_k H^0(\P, M) = \dim_k M(\emptyset).\]

We end this section with two remarks concerning the Hom and Euler pairings.

\begin{Rem}\label{Rem:Hom-singular}
Unlike the Euler pairing that is non-singular when $\HH_\P$ is an acyclic quiver, the Hom pairing is generally singular on $\Gr(k\P)$. For instance for any $v\in\P$, the virtual module $X_v = [F_v]-[I_v]$ satisfies
\[\langle X_v,X_v\rangle_\P = \langle F_v, F_v\rangle_\P + \langle I_v, I_v \rangle_\P -\langle F_v, I_v\rangle_\P -\langle I_v, F_v\rangle_\P = 0.\]
However  $\langle [M], [M]\rangle_\P \geq 1$  for a genuine module $M\in k\P\mod$, because $\Hom_{k\P}(M,M)$ contains the identity transformation.
\end{Rem}

\begin{Rem}\label{Rem:Euler-pos-def-not}
Recall that if $\HH_\P$ is a tree, then by Lemma \ref{Lem:Euler-form} for any $M, N\in k\P\mod$,
\[\chi_\P([M],[N]) = \chi_{\HH_\P}(\Dimk_\P[M],\Dimk_\P[N]),\]
where $\Dimk_\P$ is the Hilbert homomorphism (Definition \ref{Def:dimvec}), and $\chi_{\HH_\P}$ is the Euler form for the digraph $\HH_\P$ (Definition \ref{Def:Euler-form}). The Euler square $\chi_\P([M],[M])$ coincides with the Tits form  $T_\P\colon \Gr(k\P)\to \Z$, \cite[Definition 4.1.2]{Derksen-Weyman}, and by \cite[Theorem 8.6]{Schiffler} the Tits form is positive definite if and only if $\HH_\P$ is of Dynkin type ADE, and positive semi-definite if and only if $\HH_\P$ is of extended type ADE. This corresponds to $\P$ being of finite and tame representation type, respectively. In general however, $T_\P$, and hence the Euler pairing, is indefinite. 
\end{Rem}


\section{Divergence, adjointness and the Laplacian}
\label{Sec:divergence}

In this section we define and study the basic properties of the divergence and  the Laplacian for  modules over finite posets.
Recall Definition \ref{Def:div-general}

\begin{Defi}\label{Def:Div}
	Let $\P$ be a poset. The \hadgesh{left and right divergence} are defined, respectively, to be the homomorphisms
	\[\nabla^*, \nabla_*\colon\Gr(k\Pone)\to \Gr(k\P)\]
	given by $\nabla^*[N] \defeq [L_\tau(N)]- [L_\sigma(N)]$, and $\nabla_*[N] \defeq [R_\tau(N)]- [R_\sigma(N)]$.
\end{Defi}

For $ X\in \Gr(k\P)$ and $ Y\in\Gr(k\Pone)$, we have the adjointness relations with respect to the Hom pairing: 
\[\langle \nabla^* Y,  X\rangle_\P = \langle  Y, \nabla X\rangle_{\Pone}\quad\text{and}\quad
\langle  X, \nabla_* Y\rangle_\P = \langle \nabla X,  Y\rangle_{\Pone}\]

We start by proving  Theorem \ref{Th:Divergence-tree}, which gives a formula for the computation of left and right divergence on $k\P$-modules, under the assumption that the Hasse diagram $\HH_\P$ is a tree. For any $y\in\P$ let $\caln_y$ denote the \hadgesh{neighbourhood of $y$ in $\P$}, i.e.\ the subgraph of $\HH_\P$ generated by $y$ and all its in- and out-neighbours.

\begin{Thm}\label{Thm:Divergence-tree}
Let $\P$ be a finite poset whose Hasse diagram is a tree, and let $N\in k\Pone\mod$.  For any object $y\in\P$ let $\caln_y$ denote the neighbourhood of $y$ in $\P$, and let $\widehat{\caln}_y$ be the associated line poset. Then the following statements hold.
\begin{enumerate}[(1)]
\item  The left divergence $\nabla^*[N]=[L_{\tau}(N)]- [L_{\sigma}(N)]$, where the left Kan extensions are given by 
\[L_{\tau}(N)(y) \cong \bigoplus_{(u,y)\in\Pone}N(u,y),\quad{and}\quad L_{\sigma}(N)(y)\cong \colim_{\widehat{\caln}_y}N|_{\widehat{\caln}_y}.\]
\label{Thm:Divergence-tree-1}
\item The right divergence $\nabla_*[N]=[R_{\tau}(N)]- [R_{\sigma}(N)]$, where the right Kan extensions are given by 
\[R_{\tau}(N)(y) \cong \lim_{\widehat{\caln}_y} N|_{\widehat{\caln}_y},\quad\text{and}\quad R_{\sigma}(N)(y)\cong \bigoplus_{(y,v)\in\Pone} N(y,v).\]
\label{Thm:Divergence-tree-2}
\end{enumerate}
\end{Thm}
\begin{proof}
The proof is carried out in four steps. In Step 1 we analyse the categories $\tau\downarrow y$ and $y\downarrow\sigma$. In Step 2 we study the categories $\sigma\downarrow y$ and $y\downarrow\tau$. Finally, in Steps 3 and 4, we prove Statements \ref{Thm:Divergence-tree-1} and  \ref{Thm:Divergence-tree-2}. Notice first that if $y\in\P$ is minimal then the overcategory $\tau\downarrow y$ and the undercategory $\tau\downarrow y$ are empty. Hence $L_\tau N(y) = 0 = R_\tau N(y)$. Similarly, if $y$ is maximal, then $y\downarrow\sigma=\emptyset = \sigma\downarrow y$. Hence $R_\sigma N(y) = 0 = L_\sigma N(y)$. 

\noindent{\bf Step 1.}
If $y$ is not minimal then  objects of the form $((u,y), y=y)$ are maximal in  $\tau\downarrow y$. Furthermore, if $(u,y)$ and $(u',y)$ are two distinct objects in $\Pone$, then
\[(\tau\downarrow y)_{\le((u,y), y=y)}\cap (\tau\downarrow y)_{\le((u',y), y=y)} = \emptyset,\]
since otherwise there is some object $((a,b), b\le y)\in \tau\downarrow y$, where $(a,b)\le (u,y)$ and $(a,b)\le (u', y)$. Thus there are two distinct paths $b\to u\to y$ and $b\to u'\to y$ contradicting the assumption that $\HH$ is a tree. 

If $y$ is not maximal then objects  of the form $((y,v), y=y)$ are minimal in $y\downarrow \sigma$, and if $(y,v), (y,v')\in\Pone$ are two distinct objects, then
\[(y\downarrow \sigma)_{\ge ((y,v), y=y)} \cap (y\downarrow \sigma)_{\ge ((y,v'), y=y)}  = \emptyset,\]
since otherwise there is an object  $((a,b), y\le a)\in y\downarrow\sigma$, such that
$(y,v)\le (a,b)$ and $(y,v')\le (a,b)$.
Thus one obtains two distinct paths $y\to v \to a$ and $y\to v'\to a$, contradicting the assumption that $\HH_\P$ is a tree. 

\smallskip

\noindent{\bf Step 2.} 
 If $y$ is not maximal then objects of the form $((y,v), y=y)$ are maximal in $\sigma\downarrow y$. If $(u,y), (u',y)\in\Pone$ are distinct objects, then
\[(\sigma\downarrow y)_{\le ((u,y), u<y)}\cap(\sigma\downarrow y)_{\le ((u',y), u'<y)}=\emptyset,\]
since otherwise there is some object $((a,b), a\le y) \in \sigma\downarrow y$, where $(a,b)\le (u,y)$ and $(a,b)\le (u',y)$. Thus there are two distinct paths $b\to u\to y$ and $b\to u'\to y$, in contradiction to the assumption that $\HH_\P$ is a tree. Notice that the full subcategory of $\sigma\downarrow y$ with objects given by the maximal objects of the form $((y,v), y=y)$ and their in-neighbours of the form $((u,y), u<y)$ is exactly the line poset $\widehat{\caln}_y$ associated to the neighbourhood  $\caln_y$ of $y$ in $\P$. 

If $y$ is not minimal, then objects of the form $((u,y), y=y)$ are minimal in  $y\downarrow \tau$. Let $(y,v), (y,v')\in \Pone$ be distinct objects. Then 
\[(y\downarrow\tau)_{\ge ((y,v), y<v)} \cap (y\downarrow\tau)_{\ge ((y,v'), y<v')} = \emptyset,\]
since otherwise there is some $((a,b), y\le b)\in y\downarrow\tau$, such that $(y,v)\le(a,b)$ and $(y,v')\le(a,b)$.
Then one has two distinct paths $y\to v\to a$ and $y \to v' \to a$ contradicting the hypothesis that $\HH_\P$ is a tree. 
Notice  that  in this case as well, the full subcategory of $y\downarrow\tau$ consisting of the minimal objects $((u,y), y=y)$ and their out-neighbours of the form $((y,v), y<v)$ is 
exactly the line poset $\widehat{\caln}_y$. 

\smallskip

\noindent{\bf Step 3.}
By the first part of Step 1, the overcategory $\tau\downarrow y$ splits as a disjoint union of subcategories, one for each object of the form $(u,y)\in\Pone$, and in each such subcategory the object $((u,y), y=y)$ is a terminal object. Hence
\[L_{\tau}(N)(y) \defeq \colim_{\tau\downarrow y} N_\# \cong \bigoplus_{(u,y)\in\Pone}N(u,y),\]
where $N_\#$ denotes the restriction of $N$ to $\tau\downarrow y$. This proves the first statement in Part \ref{Thm:Divergence-tree-1}.

Similarly, by the second part of Step 1, the undercategory $y\downarrow\sigma$ splits as a disjoint union of subcategories, one for each object of the form $(y,v)\in\Pone$, and in each such subcategory the object $((y,v), y=y)$ is initial. Hence 
\[R_{\sigma}(N)(y) \defeq \lim_{y\downarrow\sigma} N_\# \cong \prod_{(y,v)\in\Pone} N(y,v)\cong \bigoplus_{(y,v)\in\Pone} N(y,v),\]
where $N_\#$ is the restriction of $N$ to $y\downarrow\sigma$ and the second isomorphism follows because the product is over a finite set. This proves the second statement in Part \ref{Thm:Divergence-tree-2}

\smallskip

\noindent{\bf Step 4.} 
Recall that if $F\colon \calc\to\A$ is any functor and $\A$ is a bicomplete category (i.e.\ limits and colimits exist in $\A$), then for any $M\in\A$ there is a natural bijection
\[\Hom_\A(\colim_\calc F , M) \leftrightarrow \Hom_{\A^\calc}(F, \underline{M}),\]
where $\underline{M}$ is the constant functor with value $M$ on $\C$. 

Let $\iota_y\colon\widehat{\caln}_y\to \sigma\downarrow y$ denote the inclusion. Then for any $N\in k\Pone\mod$ one has the restriction functor $\iota_y^*\colon(\sigma\downarrow y)\mod\to \widehat{\caln}_y\mod$ and an induced map 
\begin{equation}\colim_{\widehat{\caln}_y}N|_{\widehat{\caln}_y}=\colim_{\widehat{\caln}_y}  (N_\#\circ\iota_y) \xto{\iota_y^*}\colim_{\sigma\downarrow y} N_\#\defeq L_{\sigma}(N)(y).
\label{Eqn:iota_y^*}\end{equation}

Thus we get a commutative square for any module $N\in k\Pone\mod$ and $M\in\VVect$:
\[\xymatrix{
\Hom_{\VVect}\left(\colim_{\sigma\downarrow y}N_\#, M\right)\ar[r]\ar[d]^{\cong} & \Hom_{\VVect}\left(\colim_{\widehat{\caln}_y} (N_\#\circ\iota_y), M\right)\ar[d]^{\cong}\\
\Hom_{k(\sigma\downarrow y)\mod}(N_\#, \underline{M}) \ar[r] & \Hom_{k\widehat{\caln}_y\mod}(N_\#\circ\iota_y, \underline{M})
}\]
where both horizontal arrows are  induced by $\iota_y^*$. By the first part of Step 2, every module $N\colon\sigma_P\downarrow y\to \VVect$ factors uniquely through $\widehat{\caln}_y$ because if an object $((a,b), a<y)$ is not in $\widehat{\caln}_y$, then there is a unique minimal object $((u,y),u<y)\in\widehat{\caln}_y$ such that $((a,b), a<y)<((u,y),u<y)$, and so any $(\sigma\downarrow y)$-diagram is uniquely determined by its restriction to $\widehat{\caln}_y$. It follows that the bottom horizontal arrow is also an isomorphism, and by commutativity, so is the top horizontal arrow. Since this holds for any $M\in\VVect$, it follows that $\iota_y^*$ in Equation \textcolor{red}{(\ref{Eqn:iota_y^*})} is an isomorphism.

Similarly, for $F\colon\calc\to \A$, one has a natural bijection
\[\Hom_\A(M, \lim_\calc F ) \leftrightarrow \Hom_{\A^\calc}(\underline{M}, F).\]
For any $N\in k\Pone\mod$, the inclusion $\eta^y\colon \widehat{\caln_y}\to y\downarrow \tau$ induces a map 
\begin{equation}R_{\tau}(N)(y)\defeq\lim_{y\downarrow \tau} N_\# \xto{\eta^y_*}\lim_{\widehat{\caln}_y} (N_\# \circ \eta^y) =  \lim_{\widehat{\caln}_y} N|_{\widehat{\caln}_y}.\label{Eqn: eta^y_*}\end{equation}

Again, we have a commutative square for $N\in (y\downarrow\tau)\mod$ and $M\in\VVect$:
\[\xymatrix{
\Hom_{\VVect}\left(M, \lim_{y\downarrow \tau} N_\#\right) \ar[r]\ar[d]^\cong & \Hom_{\VVect}\left(M, \lim_{\widehat{\caln}_y}(N_\# \circ \eta^y)\right)\ar[d]^\cong\\
\Hom_{k(y\downarrow \tau)\mod}(\underline{M}, N_\#)\ar[r] &\Hom_{k\widehat{\caln}_y\mod}(\underline{M}, N_\# \circ \eta^y)
}\]
where both horizontal arrows are induced by $\eta_*^y$.
By the second part of Step 2, every $(y\downarrow\tau)$-diagram is determined uniquely by its restriction to $\widehat{\caln}_y$, because if an object $((a,b), y<b)$ is not in $\widehat{\caln}_y$ then there is a unique maximal object  $((y,v), y<v)\in\widehat{\caln}_y$ such that $((a,b), y<b)>((y,v), y<v)$. It follows that the bottom horizontal arrow is an isomorphism, and hence so is the top horizontal arrow. Since this holds for an arbitrary $M\in\VVect$, it follows that $\eta_*^y$ in Equation \textcolor{red}{(\ref{Eqn: eta^y_*}) }is an isomorphism.
\end{proof}

The following example motivates referring to $\nabla^*$ and $\nabla_*$ as divergence. 

\begin{Ex}
Let $\P$ be the poset whose Hasse diagram is shown on the left below, with its line digraph on the right.
\[\xymatrix{
3 && 4 && 03 && 04\\
& 0\ar[ul]\ar[ur]&&\\
1\ar[ur] && 2\ar[ul]&& 10\ar[uu]\ar[uurr] && 20\ar[uu]\ar[uull]
}\]
We compute the left and right divergence of a module $M\in k\Pone\mod$ at the object $0$.

Starting with left divergence, by Theorem \ref{Thm:Divergence-tree},
\[L_\tau M(0) \cong M(10)\oplus M(20), \quad\text{and}\quad  L_\sigma M(0) \cong \colim_{\Pone} M.\]
The computation of this colimit is elementary and can be shown to be isomorphic to the cokernel of the map
\begin{equation}
M(10)\oplus M(20) \to M(03) \oplus M(04)\label{lim-colim-map}\end{equation}
that takes an element $(x,z)$ to $(M_{103}(x) + M_{203}(z), -M_{104}(x) - M_{204}(z))$, 
where $M_{xyz}$ stands for $M(xy<yz)$ for short.

The computation of right divergence is similarly basic. By Theorem \ref{Thm:Divergence-tree} again we have
\[R_\sigma M (0) \cong M(03)\oplus M(04),\quad\text{and}\quad  R_\tau M(0) \cong \lim_{\Pone} M.\]
 The computation of the limit is again elementary, and can be shown to be the kernel of the same map  \textcolor{red}{(\ref{lim-colim-map})}. Thus we have an exact sequence
\[0 \to R_\tau M(0) \to M(10)\oplus M(20) \to M(03)\oplus M(04)\to L_\sigma M (0) \to 0.\]

Consider the values of these functors in terms of  ``flow" relative to the object $0$. Thus one may think of $L_\tau M(0)\cong M(10)\oplus M(20)$ as the in-flow at $0$, and of $L_\sigma M(0)$ as the quotient of the out-flow at $0$,  where the ``effect" (image) of the in-flow has been divided out, or in other words, the net out-flow at $0$. Similarly, $R_\sigma M(0)\cong M(03) \oplus M(04)$ can be thought of as the out-flow at $0$, while $R_\tau M(0)$ is the flow that is ``wasted" (vanishes) on the way to $0$. 

Thus  $\Dimk_\P(\nabla^* [M])(0)<0$ indicates that the net out-flow at $0$ is larger than the in-flow, or that passing through $0$ ``amplifies''  flow. Similar interpretation can be given for $\Dimk_\P(\nabla^* [M])(0)>0$ and in the corresponding situations for $\nabla_*[M](0)$. 
\end{Ex}

\chg{Next we observe that the Euler pairing is symmetric when one of the components is a gradient.
\begin{Prop}\label{Prop:pairing-with-grad}
Let $M\in k\P\mod$ and let $N\in k\Pone\mod$. Then
\[\chi_{\Pone}(\nabla[M], [N]) = \sum_{(u,v)\in\Pone}\Dimk_{\Pone}(\nabla[M])(u,v)\cdot\chi_{\Pone}([S_{u,v}], [N]) = \chi_{\Pone}([N],\nabla[M]).\]
\end{Prop}
}
\begin{proof}
	\chg{Let $K_M, C_M\in k\Pone\mod$ be the virtually trivial modules, defined as the kernel and cokernel of the natural transformation $\eta\colon\sigma^*\to\tau^*$ (See Definition \ref{Def:Ker-Coker}).  
	Thus, one has an exact sequence of $k\Pone$-modules
	\[0\to K_M \to \sigma^* M \xto{\eta_{M}} \tau^* M \to C_M\to 0,\]
	which can be split into two short exact sequences in $k\Pone\mod$
	\begin{equation}
		0\to K_{M}\to \sigma^* M \to I_M\to 0 \quad\text{and}\quad 0\to I_M\to \tau^* M \to C_{M}\to 0,
		\label{Eqn:SES-Image}
	\end{equation}
	where $I_M$ is the image functor. By applying $\Hom_{k\Pone}(-, N)$, these give long exact $\Ext$ sequences and it follows by additivity of the Euler characteristic, that 
	\[\chi(\Ext_{k\Pone}^*(\sigma^* M, N)) = \chi(\Ext_{k\Pone}^*(K_M, N)) +\chi(\Ext_{k\Pone}^*(I_M, N))\]
	and that 
	\[\chi(\Ext_{k\Pone}^*(\tau^* M, N)) = \chi(\Ext_{k\Pone}^*(C_M, N)) +\chi(\Ext_{k\Pone}^*(I_M, N)).\]
	Thus
	\[
		\chi_{\Pone}(\nabla[M], [N])   =\chi_{\Pone}([\tau^* M], [N])- \chi_{\Pone}([\sigma^* M], [N]) 
		= \chi_{\Pone}([C_M], [N])-\chi_{\Pone}([K_M], [N]).
	\]
	Similarly, by applying $\Hom_{k\Pone}(N, -)$ to the short exact sequences \textcolor{red}{(\ref{Eqn:SES-Image})}, we have 
	\[\chi_{\Pone}([N],\nabla[M])  
		= \chi_{\Pone}([N], [C_M])-\chi_{\Pone}([N], [K_M]).
	\]
By Definition \ref{Def:Ker-Coker} both $K_M$ and $C_M$ are virtually trivial. Hence they are isomorphic to a direct sum of simple modules. Thus
\[[K_M] = \sum_{(u,v)\in\Pone} \Dimk_{\Pone}[K_M](u,v)\cdot [S_{u,v}], \quad\text{and}\quad [C_M] = \sum_{(u,v)\in\Pone} \Dimk_{\Pone}[C_M](u,v)\cdot [S_{u,v}].\]
Thus,
\begin{align*}
\chi_{\Pone}(\nabla[M], [N]) & = \chi_{\Pone}([C_M], [N])-\chi_{\Pone}([K_M], [N])=\\
		& \sum_{(u,v)\in\Pone} \left(\Dimk_{\Pone}[C_M](u,v)- \Dimk_{\Pone}[K_M](u,v)\right)\cdot \chi_{\Pone}([S_{u,v}],[N]) =\\
		& \sum_{(u,v)\in\Pone} \left(\Dimk_{\Pone}[\tau^*M](u,v)- \Dimk_{\Pone}[\sigma^*M](u,v)\right)\cdot \chi_{\Pone}([S_{u,v}],[N]) =\\
		&  \sum_{(u,v)\in\Pone} \Dimk_{\Pone}(\nabla[M])(u,v)\cdot \chi_{\Pone}([S_{u,v}],[N]).
\end{align*}
This proves the first equality, and the second follows by analogy.}
\end{proof}

Next we examine an implication of vanishing left divergence in the case where $\HH_\P$ is a rooted tree. 

\begin{Prop}
	\label{Prop:zero-divergence}
	Let $\P$ be a finite poset  such that its Hasse diagram $\HH_\P$ is a rooted tree. Let $ X = [U]-[V]\in \Gr(k\Pone)$ be any element, with $U, V\in k\Pone\mod$, such that $\nabla^* X = 0$. Then $\dim_kU(x,y) = \dim_kV(x,y)$ for all $(x,y)\in\Pone$. \chg{In particular for any $U\in k\Pone\mod$, $\nabla^*[U] =0$ if and only if $U=0$.}
\end{Prop}
\begin{proof} 
Let $x_0$ denote the unique minimal object in $\P$. For any $z\in \P$ there is a unique path from $x_0$ to $z$, and since $\HH_\P$ is a tree, it is the longest path in $\P$ that ends in $z$. Furthermore, the assumption that $\HH_\P$ is a rooted tree implies that the in-degree at each vertex in $\HH_\P$ except $x_0$ is exactly $1$. Fix an object $(x,y)\in\Pone$, and let $n = d(x_0,y)$, where $d(-,-)$ denotes the directed path distance function on $\HH_\P$.  Let $G_{(x,y)}$ be the indecomposable injective $k\Pone$-module determined by the object $(x,y)$ (see Definition \ref{Def:Proj-Inj-Simp}\ref{Def:Proj-Inj-Simp-Inj}). By Lemma \ref{Lem:GF-integrable} the exists a module $M\in k\P\mod$ such that $\nabla[M] = [G_{(x,y)}]$. Hence,
	\begin{align*}
		\dim_kU(x,y) = & \langle [U], [G_{(x,y)}]\rangle_{\Pone}  =  \langle [U],\nabla[M]\rangle_{\Pone} = \\
		& \langle\nabla^*[U], [M]\rangle_\P =   \langle\nabla^*[V], [M]\rangle_\P=  \\ 
		& \langle[V], \nabla[M]\rangle_{\Pone} = \langle[V], [G_{(x,y)}]\rangle_{\Pone} = \dim_k V(x,y),
	\end{align*}
where the first and last equalities follow from injectivity of $G_{(x,y)}$ in $k\Pone\mod$ and Lemma \ref{Lem:Pairing-with-proj-inj}. The third and sixth equalities follow by adjointness. \chg{The second statement follows at once by setting $V=0$.}
\end{proof}

\chg{An immediate consequence of Proposition \ref{Prop:zero-divergence} is that under its hypotheses, if $\nabla^*\nabla[M] = 0$, then $\dim_kM(x) =\dim_kM(y)$ for all objects $(x,y)\in \Pone$.}

With gradient and divergence operators in place, we can now define the corresponding Laplacians.
\begin{Defi}\label{Def:Laplacians}
	Let $\P$ be a finite poset. Define the \hadgesh{left and right Laplacians} 
	$\Lap^0$ and $\Lap_0$ respectively in $\End(\Gr(k\P))$, to be the  group endomorphisms
	\[\Lap^0  \defeq \nabla^*\circ\nabla\quad\text{and}\quad \Lap_0\defeq \nabla_*\circ\nabla.\]
\end{Defi}

\begin{Cor}
	\label{Cor:Harmonic}
	Let $\P$ be a finite poset, whose Hasse diagram is a rooted tree, and let $ X = [M]-[N]\in\Gr(k\P)$ with $M, N\in k\P\mod$ be a  virtual module such that $\Delta^0[X]=0$. Then, for each object $(u,v)\in\Pone$,
	\[\dim_k M(v) - \dim_k M(u) = \dim_k N(v) - \dim_k N(u).\]
	In particular, if $N=0$, then for each $x, y\in \P$, $\dim_kM(x) = \dim_kM(y)$.
\end{Cor}
\begin{proof}
	Write 
	\[\nabla X = [\tau^*M\oplus\sigma^*N] - [\tau^*N\oplus\sigma^*M].\]
	Then $0=\Lap^0 X = \nabla^*(\nabla X)$ and  Proposition \ref{Prop:zero-divergence} applies.  Thus, for each $(u,v)\in\Pone$, 
	\begin{align*}
		\dim_k M(v) + \dim_k N(u)  =&  \dim_k(\tau^*M(u,v)\oplus\sigma^*N(u,v)) = \\
		& \dim_k(\tau^*N(u,v)\oplus\sigma^*M(u,v)) = \\
		& \dim_kN(v) + \dim_kM(u).\\
	\end{align*}
	The first claim follows. 				
	
	Since $\P$ is generated by a rooted tree, it is in particular connected. The second claim follows from the first by connectivity of $\P$ and induction on the length of paths in $\P$. 
	\end{proof}

\chg{Recall that our gradient is a categorified version of the discrete gradient on digraphs \cite{Lim}, and similarly, the left and right divergence operators are a categorification of the discrete divergence (See Examples \ref{Ex:standard-grad} and \ref{Ex:standard-div}). In graph theory one can define for a digraph $\G = (V,E)$ a cochain complex $C_*(\G)$, such that $C_0(\G)$ and $C_1(\G)$ are the vector spaces of functions from  $V$ and $E$ to the ground field $k$, respectively. With this setup one shows that the kernel of the Laplacian and that of the gradient coincide \cite[Section 4.3]{Lim}. 
It makes sense to ask whether the same holds in our context. Theorem \ref{Th:Lap-Grad-Kernels} states that this is indeed the case when $\P = [n]$, namely the poset whose  objects are all $0\le i \le n$ with the usual order relation. The theorem follows rather easily from our next, more general,  proposition.}

\chg{\begin{Prop}\label{Prop:Vanishing-Div-Line}
Fix a positive integer $n$ and let   $\P= [n]$.  Then $\nabla^*, \nabla_*\colon \Gr(k\Pone)\to\Gr(k\P)$ are  monomorphisms. 
\end{Prop}}
\begin{proof}
\chg{Notice first  that for any $k\in \P$, the  over and under categories of $\tau$ and $\sigma$ take the following form:
\begin{itemize}
\item $\tau\downarrow k = \emptyset$ if $k=0$ and $(0,1)<(1,2)<\cdots<(k-1,k)$ otherwise.
\item $\sigma\downarrow k = (0,1)<(1,2)<\cdots<(k,k+1)$ if $k<n$ and $(0,1)<(1,2)<\cdots<(n-1,n)$ otherwise.
\item $k\downarrow\tau = (k-1,k)<(k,k+1)<\cdots (n-1,n)$ if $k>0$ and $(0,1)<\cdots<(n-1,n)$ otherwise. 
\item $k\downarrow\sigma = \emptyset$ if $k=n$ and $(k,k+1)<\cdots<(n-1,n)$  otherwise.
\end{itemize}
Thus for a module $M\in k\Pone$, the left Kan extension at an object $x$ is determined by the value of $M$ on the maximal object in $\alpha\downarrow x$, where $\alpha = \tau$ or $\sigma$. Similarly, the right Kan extension at $x$ is determined by the value of $M$ on the minimal object in $x\downarrow \alpha$. The following table gives the values of the corresponding functor on each of the objects.}

\chg{\[\begin{array}{| l | c | c | c | c |}
\hline
x & L_\tau(M)(x) & L_\sigma(M)(x) &  R_\tau(M)(x) & R_\sigma(M)(x)\\
\hline
0 & 0 & M(0,1)& M(0,1) & M(0,1)\\
\hline
0<k<n & M(k-1,k) & M(k, k+1) & M(k-1,k) & M(k, k+1)\\
\hline
n & M(n-1,n) & M(n-1,n) &M(n-1,n) & 0\\
\hline
\end{array}\]
Thus the modules $L_\tau(M), L_\sigma(M)\in k\Pone\mod$ take the following form:
\[\xymatrix{
L_\tau(M)\colon& 0\ar[r] &M(0,1)\ar[r] & M(1,2)\ar[r] &\cdots\ar[r]&  M(n-1, n)\\
L_\sigma(M)\colon &  M(0,1)\ar[r] & M(1,2)\ar[r]& M(2,3)\ar[r] & \cdots\ar[r] & M(n-1, n).
}\]
In particular for interval modules $[(m-1,m),(k-1,k)]$ in $\Pone$ one has $L_\tau[(m-1,m),(k-1,k)] = [m,k]$ and $L_\sigma[(m-1,m),(k-1,k)] = [m-1,k-1]$, if $k<n$ and $L_\sigma[(m-1,m),(n-1,n)] = [m-1,n]$. Thus both $L_\tau$ and $L_\sigma$ are injective on classes of intervals, and hence they are injective as homomorphisms.
}

\chg{Let $M, N\in k\Pone\mod$ be any modules such that $\nabla^*[M] = \nabla^*[N]$. We will show that $M\cong N$. Write $M$ and $N$ as a sum of intervals,
\[M = \bigoplus_s I_s,\quad\text{and}\quad  N = \bigoplus_r J_r.\]
We may assume without loss of generality that $I_s\neq J_r$ for all $s$ and $r$, since any such duplication can be removed without changing the equality of divergence. Similarly, we may assume that there are no intervals $I, J$ in either side such that $L_\tau[I] = L_\sigma[J]$, since any such intervals will not be seen by the divergence. By hypothesis we have 
\[\sum_sL_\tau[I_s] + \sum_rL_\sigma[J_r] = L_\tau[M] + L_\sigma[N] = L_\sigma[M] + L_\tau[N] = \sum_sL_\sigma[I_s] + \sum_rL_\tau[J_r],\]
which gives an interval decomposition for both sides. Let $I_s$ be any interval on the left hand side of this equation. Then, either $L_\tau[I_s]=L_\sigma[I_t]$ for some $t$ or $L_\tau[I_s]=L_\tau[J_r]$ for some $r$. The first possibility is excluded by the discussion above, and so we must have $L_\tau[I_s]=L_\tau[J_r]$. But the calculation of $L_\tau$   on intervals shows that it is  injective on intervals. Hence $I_s = J_r$ and both modules can be removed from the sum without changing the equality. Similarly, if we let $J_r$ be an interval on the left hand side, we can show, using our assumption and injectivity of $L_\sigma$ on intervals, that $J_r = I_s$ for some $s$, and hence we may remove both from the equation. Since all sums are finite, proceeding by induction we see that the interval decomposition of $M$ and $N$ coincide, and hence $M\cong N$. }

\chg{The proof for the right divergence is similar. One observes that $R_\sigma[(m,m+1),(k,k+1)] = [m,k]$ for all $m\geq 0$ and $k\le n-1$, and $R_\tau[(m,m+1),(k,k+1)] = [m+1,k+1]$ if $m>0$ and $R_\tau[(0,1), (k,k+1)] = [0,k+1]$, for all $k\le n-1$. In particular $R_\tau$ and $R_\sigma$ are injective. The argument then proceeds similarly to the case of left divergence. This completes the proof.}
\end{proof}

\chg{As a corollary we obtain a proof of Theorem \ref{Th:Lap-Grad-Kernels}, our $0$-dimensional version of Hodge's theorem. 
\begin{Thm}\label{Thm:Lap-Grad-Kernels}
Let $\P=[n]$ for $n\geq 0$. Then $\Ker\nabla = \Ker\Delta^0 = \Ker\Delta_0$ 
\end{Thm}
\begin{proof}
Clearly $\Ker\nabla\subseteq\Ker\Delta^0$. Thus it remains to prove the converse inclusion.
Let $X = [M]-[N]\in\Ker\Delta^0$ be any element. Thus $\Delta^0[M] = \Delta^0[N]$. By definition and rearranging of summands, this holds if and only if
\[L_\tau(\tau^*M\oplus\sigma^*N)\oplus L_\sigma(\tau^*N\oplus\sigma^*M) \cong L_\tau(\tau^*N\oplus\sigma^*M)\oplus L_\sigma(\tau^*M\oplus\sigma^*N).\]
Equivalently we have
\[\nabla^*(\tau^*M\oplus\sigma^*N) \cong \nabla^*(\tau^*N\oplus\sigma^*M),\]
and by Proposition \ref{Prop:Vanishing-Div-Line}, $\tau^*M\oplus\sigma^*N \cong \tau^*N\oplus\sigma^*M$, which holds by definition if and only if $\nabla[M] = \nabla[N]$, namely if and only if $X\in\Ker\nabla$. The proof for $\Delta_0$ is essentially the same.
\end{proof}
}

\newcommand{\out}{\mathrm{out}}
\newcommand{\inn}{\mathrm{in}}

\section{An Example}
\label{Sec:Applications}
In this final section we demonstrate some nice properties of the gradient by means of an example that highlight the advantages that it offers, particularly in potential applications. As already pointed out in the introduction, almost all quivers (including almost all finite posets) are of infinite or wild representation type and so classifying their indecomposable representations is hard or, in the wild case, impossible. Hence understanding modules globally by means of their indecomposable summands is not feasible.

Our calculus  methods lend themselves naturally to  \hadgesh{local} investigation of modules, and hence are not tied to the representation type of posets. We consider   a family of posets that is generally of wild representation type, but where the associated line poset has finite representation type, and hence the gradient of any module over these posets can be well understood. 

An in-depth investigation of modules over so called \hadgesh{commutative ladders} appeared in \cite{EH}. While classification of indecomposable modules over those ladder posets may be hard or even impossible, the gradient is much easier to understand.

Denote a left-to-right arrow by $F$ and a right-to-left arrow by $B$. Following terminology from quiver representations, we call the finite poset with \(n\) objects obtained schematically as a juxtaposition of arrows of type $F$ or $B$ in any order as an \hadgesh{$\mathbb{A}_n$ poset}. By Gabriel's theorem \cite{Gabriel}, \cite[Theorem 4.2.4]{Derksen-Weyman} such posets are of finite  representation type, i.e.\ they admit  finitely many isomorphism classes of indecomposable representations. Any $\mathbb{A}_n$ poset is uniquely characterised by a sequence $X_1X_2\cdots X_{n-1}$, where each $X_i$ is either $F$ or $B$. We refer to the characterising sequence as the \hadgesh{type} of the poset. A \hadgesh{commutative ladder of length $n$} is a poset that can be written schematically as two $\mathbb{A}_n$ posets $L_1$ and $L_2$ of the same type with arrows from each object of $L_1$ to the corresponding object in $L_2$. By \cite[Theorem 3, 4]{EH} general commutative ladders of any type with length at most $4$ are of finite representation type, whereas they are of infinite representation type if their length is \(\geq 5\).

Consider two specific types of commutative ladders. The first is of any length and is made of a related pair of alternating sequences of type either $FBFB\cdots$ or $BFBF\cdots$ (either may end with $B$ or $F$). We refer to a poset of this form and length $n$ as a \hadgesh{zig-zag ladder of length $n$}. The second is of even length $2n$ and is made of a related pair of sequences of the form $FFBBFF\cdots$ or $BBFFBB\cdots$ (either may end with $BB$ or $FF$). This type will be referred to as a \hadgesh{double zig-zag ladder of length $2n$}. As pointed out any commutative ladder poset of length at least $5$ is of infinite representation type. 

\begin{Lem}\label{Lem:zig-zag-ladders}
Let $\P$ be a commutative ladder poset. Then
\begin{enumerate}[(1)]
\item if $\P$ is a zig-zag ladder of any length, then $\Pone$ is a disjoint union of $\mathbb{A}_n$ posets with \(n \leq 2\), and \label{Lem:zig-zag-ladders:1}
\item if $\P$ is a double zig-zag ladder of any length, then $\Pone$ is an $\mathbb{A}_n$ poset.  \label{Lem:zig-zag-ladders:2}
\end{enumerate}
\end{Lem}
\begin{proof}
The Hasse diagram of a zig-zag commutative ladder of type \(FBFB\cdots\) has the general form drawn in the left diagram below.
One easily observes that the associated line digraph has the form drawn in the right diagram.
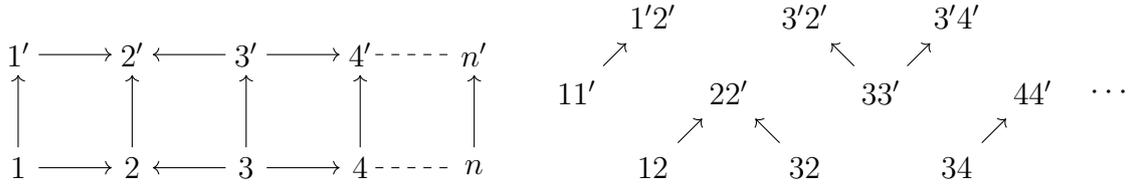
\begin{center}
	\begin{tikzpicture}
	\tikzstyle{point}=[circle,thick,draw=black,fill=black,inner sep=0pt,minimum width=2pt,minimum height=2pt]
		\tikzstyle{arc}=[shorten >= 2pt,shorten <= 2pt,->]
		\tikzstyle{arcl}=[shorten >= 2pt,shorten <= 2pt, <-]
		\node[] (1) at (0,0) {$1$};
		\node[] (2) at (1.5,0) {$2$};
		\node[] (3) at (3,0) {$3$};
		\node[] (4) at (4.5,0) {$4$};
		\node[] (n) at (6,0) {$n$};
		\node[] (1') at (0,1.5) {$1'$};
		\node[] (2') at (1.5,1.5) {$2'$};
		\node[] (3') at (3,1.5) {$3'$};
		\node[] (4') at (4.5,1.5) {$4'$};
		\node[] (n') at (6,1.5) {$n'$};
		\draw[arc] (0,0.2) to (0,1.3);
		\draw[arc] (1.5,0.2) to (1.5,1.3);
		\draw[arc] (3,0.2) to (3,1.3);
		\draw[arc] (4.5,0.2) to (4.5,1.3);
		\draw[arc] (6,0.2) to (6,1.3);
		\draw[arc] (0.2,0) to (1.3,0);
		\draw[arcl] (1.7,0) to (2.8,0);
		\draw[arc] (3.2,0) to (4.3,0);
		\draw[dashed] (4.7,0) to (5.7,0);
		\draw[arc] (0.2,1.5) to (1.3,1.5);
		\draw[arcl] (1.7,1.5) to (2.8,1.5);
		\draw[arc] (3.2,1.5) to (4.3,1.5);
		\draw[dashed] (4.7,1.5) -- (5.7,1.5);
	\end{tikzpicture}
	\hskip.2in
	\begin{tikzpicture}
		\tikzstyle{point}=[circle,thick,draw=black,fill=black,inner sep=0pt,minimum width=2pt,minimum height=2pt]
		\tikzstyle{arc}=[shorten >= 2pt,shorten <= 2pt,->]
		\node[] (11') at  (0,0) {\(11'\)};
		\node[] (1'2') at  (1,1) {\(1'2'\)};
		\node[] (12) at  (1,-1) {\(12\)};
		\node[] (22') at  (2,0) {\(22'\)};
		\node[] (3'2') at  (3,1) {\(3'2'\)};
		\node[] (32) at  (3,-1) {\(32\)};
		\node[] (33') at  (4,0) {\(33'\)};
		\node[] (3'4') at  (5,1) {\(3'4'\)};
		\node[] (34) at  (5,-1) {\(34\)};
		\node[] (44') at  (6,0) {\(44'\)};
		
		\draw[arc] (11') to (1'2');
		\draw[arc] (12) to (22');
		\draw[arc] (32) to (22');
		\draw[arc] (33') to (3'2');
		\draw[arc] (33') to (3'4');
		\draw[arc] (34) to (44');	
		
		\node[] at (7,0) {\(\cdots\)};
	\end{tikzpicture}
\end{center}
This proves Part \ref{Lem:zig-zag-ladders:1}.

A double zig-zag ladder of type \(FFBBFF\cdots\) has the form 

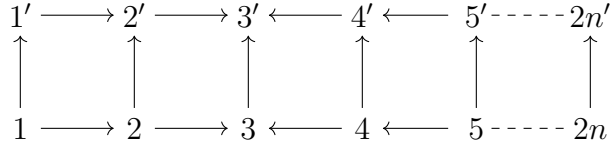
\begin{center}
\begin{tikzpicture}
	\tikzstyle{point}=[circle,thick,draw=black,fill=black,inner sep=0pt,minimum width=2pt,minimum height=2pt]
		\tikzstyle{arc}=[shorten >= 2pt,shorten <= 2pt,->]
		\tikzstyle{arcl}=[shorten >= 2pt,shorten <= 2pt, <-]
		\node[] (1) at (0,0) {$1$};
		\node[] (2) at (1.5,0) {$2$};
		\node[] (3) at (3,0) {$3$};
		\node[] (4) at (4.5,0) {$4$};
		\node[] (5) at (6,0) {$5$};
		\node[] (n) at (7.5,0) {$2n$};
		\node[] (1') at (0,1.5) {$1'$};
		\node[] (2') at (1.5,1.5) {$2'$};
		\node[] (3') at (3,1.5) {$3'$};
		\node[] (4') at (4.5,1.5) {$4'$};
		\node[] (5') at (6,1.5) {$5'$};
		\node[] (n') at (7.5,1.5) {$2n'$};
		\draw[arc] (0,0.2) to (0,1.3);
		\draw[arc] (1.5,0.2) to (1.5,1.3);
		\draw[arc] (3,0.2) to (3,1.3);
		\draw[arc] (4.5,0.2) to (4.5,1.3);
		\draw[arc] (6,0.2) to (6,1.3);
		\draw[arc] (7.5, 0.2) to (7.5, 1.3);
		\draw[arc] (0.2,0) to (1.3,0);
		\draw[arc] (1.7,0) to (2.8,0);
		\draw[arcl] (3.2,0) to (4.3,0);
		\draw[arcl] (4.7,0) to (5.7,0);
		\draw[dashed] (6.2,0) to (7.2,0);
		\draw[arc] (0.2,1.5) to (1.3,1.5);
		\draw[arc] (1.7,1.5) to (2.8,1.5);
		\draw[arcl] (3.2,1.5) to (4.3,1.5);
		\draw[arcl] (4.7,1.5) to (5.7,1.5);
		\draw[dashed] (6.2,1.5) to (7.2,1.5);
	\end{tikzpicture}
\end{center}
The associated line digraph has the form
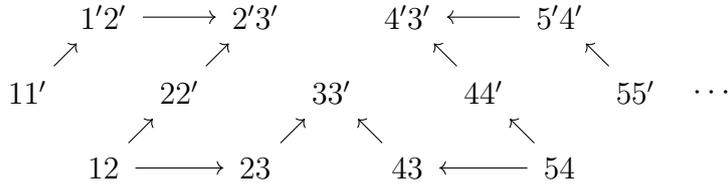
\begin{center}
	\begin{tikzpicture}
		\tikzstyle{point}=[circle,thick,draw=black,fill=black,inner sep=0pt,minimum width=2pt,minimum height=2pt]
		\tikzstyle{arc}=[shorten >= 2pt,shorten <= 2pt,->]
		\node[] (11') at  (0,0) {\(11'\)};
		\node[] (1'2') at  (1,1) {\(1'2'\)};
		\node[] (12) at  (1,-1) {\(12\)};
		\node[] (22') at  (2,0) {\(22'\)};
		\node[] (2'3') at  (3,1) {\(2'3'\)};
		\node[] (23) at  (3,-1) {\(23\)};
		\node[] (33') at  (4,0) {\(33'\)};
		\node[] (4'3') at  (5,1) {\(4'3'\)};
		\node[] (43) at  (5,-1) {\(43\)};
		\node[] (44') at  (6,0) {\(44'\)};
		\node[] (54) at  (7,-1) {\(54\)};
		\node[] (5'4') at  (7,1) {\(5'4'\)};
		\node[] (55') at  (8,0) {\(55'\)};
		
		\draw[arc] (11') to (1'2');
		\draw[arc] (12) to (22');
		\draw[arc] (1'2') to (2'3');
		\draw[arc] (22') to (2'3');
		\draw[arc] (12) to (23);
		\draw[arc] (23) to (33');
		\draw[arc] (43) to (33');
		\draw[arc] (44') to (4'3');
		\draw[arc] (54) to (44');
		\draw[arc] (54) to (43);
		\draw[arc] (5'4') to (4'3');
		\draw[arc] (55') to (5'4');	
		
		\node[] at (9,0) {\(\cdots\)};
	\end{tikzpicture}
\end{center}
which is an $\mathbb{A}_n$ poset, thus proving Part \ref{Lem:zig-zag-ladders:2}.

The proof for zig-zag  and double zig-zag ladders of types $BFBF\cdots$ and \(BBFFBB\cdots\), respectively, is analogous, with arrows going the opposite way.
\end{proof}

We obtain an immediate corollary of Lemma \ref{Lem:zig-zag-ladders} and Gabriel's theorem.
\begin{Cor}\label{Cor:ladder-fin-rep-type}
Let $\P$ be a zig-zag or double zig-zag ladder of any length \(n\). Then for any $M\in k\P\mod$ the gradient $\nabla[M]$ is equal to a difference of two finite sums of isomorphism classes of $k\mathbb{A}_n$-modules.
\end{Cor}

There are obviously more types of commutative ladders, and more general posets, whose gradients can be decomposed as a finite combination of isomorphism classes of modules of finite representation type. For instance a ladder poset whose horizontal parts are juxtapositions of sequences of type $BBF$ also has a line poset that is a disjoint union of  $\mathbb{A}_n$ posets. Similarly certain trees have line posets of type $\mathbb{A}_n$ or disjoint union of them, see Figure \ref{Fig:gradient}. The classification of posets whose associated line posets are of finite representation type is an interesting question for which we do not know the answer.

\chg{This discussion also highlights a limitation of our approach. In both cases described in Lemma  \ref{Lem:zig-zag-ladders} the Grothendieck group $\Gr(k\P)$ is an infinitely generated free abelian group, if $k$ is algebraically closed and the length of the ladder is at least $5$, while $\Gr(k\Pone)$ is always finitely generated. This suggests that the gradient, as defined with respect to the Hasse diagram, does not carry sufficient information to distinguish modules up to some controllable  error term. The gradient with respect to larger generating quivers will be studied in subsequent work.}

\appendix
\chg{\section{Critical points}
We take the analogy with calculus on graphs a bit further by describing the relationship between our gradient and what is considered as critical points in graph theory. Unlike the case of ordinary calculus, the notion of critical points in graph theory does not necessarily mean vertices where the gradient vanishes. These concepts are very useful in applied graph theory (specifically for networks), but in our context they represent weak but potentially useful invariants of modules. A more comprehensive study of critical points for modules and potential applications will be investigated in future publication. Here we satisfy ourselves with the basic definitions. }

\chg{\begin{Defi}\label{Def:Crit-Pts}
Let $\P$ be a finite poset and let $\HH_\P$ be its Hasse diagram. For each vertex $v\in\HH_\P$, let $\caln_v^\inn$ and $\caln_v^\out$ denote the in- and out-neighbourhoods of $v$.  Let $f\colon \P\to\R$ be a function. A vertex $v\in\HH_\P$ is said to be 
\begin{itemize}
\item \hadgesh{a weak local maximum} if $f(v)\geq f(u)$ for all $u\in\caln_v^\inn$,
\item \hadgesh{a weak local minimum} if $f(v)\leq f(u)$ for all $u\in\caln_v^\out$,
\item \hadgesh{a strong max-type critical point} if $f(v)\geq f(u)$ for all $u\in\caln_v^\inn$ and $f(v)\geq f(w)$  for all  $w\in\caln_n^\out$,   
\item \hadgesh{a strong min-type critical point} if $f(v)\leq f(u)$ for all $u\in\caln_v^\out$ and $f(v)\leq f(w)$  for all  $w\in\caln_n^\inn$,
\item \hadgesh{a flat point} if $f(v) =f(u)$ for all in- and out-neighbours $u$ of $v$.
\end{itemize}
\end{Defi}
}

\chg{For a fixed vertex $v\in\HH_\P$, consider the associated line posets. For in- and out-neighbourhood the line poset is a set of isolated vertices, indexed by the in- and out-edges respectively. The line-poset of a full neighbourhood is a complete  bipartite digraph with source vertices indexed by in-edges and target vertices by out-edges. The gradient of a virtual $k\P$-module is very easy to compute and its Hilbert function gives a way of talking about the  concepts corresponding to Definition \ref{Def:Crit-Pts} in our context. We say that a vector $\mathbf{v} \in \R^n$ is positive if all its entries are positive. Similarly, for negative, non-positive and non-negative.}

\chg{\begin{Defi}
For any poset $\Q = (V,E)$, let $\Dimk=\Dimk_Q\colon\Gr(k\Q)\to\Z^V$ denote the Hilbert  function. Let $\P$ be a finite poset, and for any virtual module $X\in\Gr(k\P)$ and a sub-poset $\Q\subseteq\P$ denote by $X_\Q$ the restriction of $X$ to $\Q$. We say that an object $v\in\P$ is a
\begin{itemize}
\item \hadgesh{a weak local maximum} if $\Dimk(\nabla(X)_{\widehat{\caln}_v^\inn})$ is non-negative,
\item \hadgesh{a weak local minimum} if $\Dimk(\nabla(X)_{\widehat{\caln}_v^\out})$ is non-positive,
\item \hadgesh{a strong max-type critical object} if $\Dimk(\nabla(X)_{\widehat{\caln}_v})$ is non-negative, 
\item \hadgesh{a strong min-type critical object} if $\Dimk(\nabla(X)_{\widehat{\caln}_v})$ is non-positive,  and
\item \hadgesh{a flat object} if $\Dimk(\nabla(X)_{\widehat{\caln}_v})$ is the zero vector.
\end{itemize}
\end{Defi}
}
\bibliographystyle{plain}
\bibliography{references}

\end{document}